\documentclass[A4paper, 12pt]{article} 
\usepackage[T1]{fontenc}
\usepackage[utf8]{inputenc} 
\usepackage[all,cmtip]{xy}
\usepackage[new]{old-arrows}
\usepackage{extarrows}

\usepackage[margin=1in, top=1.2in, bottom=1.4in]{geometry}
\usepackage{graphicx} 

\usepackage{booktabs} 
\usepackage{array} 
\usepackage{paralist} 
\usepackage{verbatim} 
\usepackage{subfig} 
\usepackage{mathtools}
\usepackage{amssymb}
\usepackage{changepage}
\usepackage{amsmath}
\usepackage{amsthm}
\usepackage{textcomp}
\usepackage{mathdots}
\usepackage{mathrsfs}
\usepackage{enumitem}
\usepackage[colorlinks, allcolors=blue]{hyperref}
\usepackage{appendix}
\usepackage{stmaryrd}
\usepackage{hyperref}
\usepackage[bottom]{footmisc}
\hypersetup{colorlinks,linkcolor={blue},citecolor={blue},urlcolor={black}} 

\usepackage{subfiles}

\newcommand\blfootnote[1]{%
  \begingroup
  \renewcommand\thefootnote{}\footnote{#1}%
  \addtocounter{footnote}{-1}%
  \endgroup
}

\usepackage[tightpage,psfixbb]{preview}
\usepackage{sectsty}
\sectionfont{\large}
\subsectionfont{\normalfont \itshape}
\subsubsectionfont{\normalfont\itshape}
\makeatletter

\usepackage[nottoc,notlof,notlot]{tocbibind} 
\usepackage[titles,subfigure]{tocloft} 
\theoremstyle{sltheoremstyle}
\newtheorem{theorem}{Theorem}[section]
\newtheorem{lemma}[theorem]{Lemma}
\newtheorem{claim}[theorem]{Claim}
\newtheorem{corollary}[theorem]{Corollary}
\newtheorem{question}[theorem]{Question}
\newtheorem{proposition}[theorem]{Proposition}
\newtheorem{conjecture}[theorem]{Conjecture}

\newtheorem{condition}[theorem]{Condition}
\newtheorem{conditions}[theorem]{Conditions}
\newtheorem*{corollary-non}{Corollary}
\newtheorem*{lemma-non}{Lemma}
\newtheorem*{theorem-non}{Theorem}
\newtheorem*{proposition-non}{Proposition}
\newtheorem*{condition-non}{Condition}
\newtheorem*{conditions-non}{Conditions}

\theoremstyle{definition}
\newtheorem{notation}[theorem]{Notation}

\newtheorem{remark}[theorem]{Remark}

\newtheorem{definition}[theorem]{Definition}

\newtheorem{example}[theorem]{Example}
\newtheorem{examples}[theorem]{Examples}

\usepackage[x11names]{xcolor}
\usepackage{tabularx}
\usepackage{geometry}
\usepackage{setspace}
\usepackage{lipsum}
\usepackage[explicit]{titlesec}
\usepackage{authblk}

\usepackage[backend=bibtex,style=alphabetic]{biblatex}

\newcommand*{\img}[1]{%
    \raisebox{-.02\baselineskip}{%
        \includegraphics[
        height=\baselineskip,
        width=\baselineskip,
        keepaspectratio,
        ]{#1}%
    }%
}

\newcommand{\lr}[1]{{\longrightarrow{#1}}}

\newcommand{\ol}[1]{{\overline{#1}}}

\newcommand{\n}[1]{\left\| #1 \right\|}
\newcommand{\set}[1]{\left\{ #1 \right\}}
\newcommand{\va}[1]{\left| #1 \right|}

\newcommand{\Isom}{\textnormal{Isom}}

\newcommand{\Id}{\textnormal{Id}}

\newcommand{\white}{\;\;\;}

\newcommand{\Stab}{\textnormal{Stab}}
\newcommand{\sm}{\setminus}

\newcommand{\CC}{\mathbb{C}}

\newcommand{\BB}{\mathbb{B}}

\newcommand{\OO}{\mathcal{O}}

\newcommand{\HH}{\mathbb{H}}

\newcommand{\QQ}{\mathbb{Q}}
\newcommand{\PP}{\mathbb{P}}

\newcommand{\PGL}{\textnormal{PGL}}

\newcommand{\GL}{\textnormal{GL}}
\newcommand{\SO}{\textnormal{SO}}

\newcommand{\RR}{\mathbb{R}}

\newcommand{\ZZ}{\mathbb{Z}}

\newcommand{\vep}{\varepsilon}
\newcommand{\ES}{\emptyset}

\newcommand{\CCH}{\bb C H}
\newcommand{\RRH}{\bb R H}

\newcommand{\Ker}{\textnormal{Ker}}

\newcommand{\id}{\textnormal{id}}

\newcommand{\Hom}{\textnormal{Hom}}

\newcommand{\Lie}{\textnormal{Lie}}

\newcommand{\Gal}{\textnormal{Gal}}

\newcommand{\End}{\textnormal{End}}

\newcommand{\PO}{\textnormal{PO}}

\newcommand{\eis}{\mathrm{eis}}
\newcommand{\GamEis}{\Gamma_{\mathrm{eis}}}

\newcommand{\Aeisfour}{\mr A_\eis^4}

\newcommand{\Aut}{\textnormal{Aut}}

\newcommand{\Proj}{\textnormal{Proj}}

\newcommand{\ca}[1]{{\mathcal{#1}}}
\newcommand{\bb}[1]{{\mathbb{#1}}}

\newcommand{\mr}[1]{{\mathscr{#1}}}
\newcommand{\mf}[1]{{\mathfrak{#1}}}
\newcommand{\tn}[1]{{\textnormal{#1}}}
\let\rm\relax 
\newcommand{\rm}[1]{{\mathrm{#1}}}

\newcommand{\wt}[1]{{\widetilde{#1}}}

\DeclareRobustCommand\longtwoheadrightarrow
     {\relbar\joinrel\twoheadrightarrow}

\DeclareSymbolFont{bbm}{U}{bbm}{m}{n}
\DeclareSymbolFontAlphabet{\mathbbm}{bbm}

\begin{document}

\title{\Large{\textbf{Non-arithmetic hyperbolic orbifolds \\ attached to unitary Shimura varieties}}}
\author{Olivier de Gaay Fortman}
\date{\vspace{-7ex}}




\maketitle 


%


%

\abstract{
\textbf{Abstract:}\blfootnote{
\emph{Date:} \today.  \emph{Mathematics Subject Classification:} 20F65, 22E40, 53C35.} We develop a new method of constructing non-arithmetic lattices in the projective orthogonal group $\text{PO}(n,1)$ for every integer $n$ larger than one. The technique is to consider anti-holomorphic involutions on a complex arithmetic ball quotient, glue their fixed loci along geodesic subspaces, and show that the resulting metric space carries canonically the structure of a complete real hyperbolic orbifold. The volume of various of these non-arithmetic orbifolds can be explicitly calculated.
}


\newpage


{\footnotesize \tableofcontents}


\section{Introduction}


To study locally symmetric spaces $M = \Gamma \backslash G / K$ of finite volume, one would like to classify 
lattices in semi-simple Lie groups. This difficult problem can be simplified somewhat 
via the notion of arithmeticity. 
Namely, arithmetic lattices can be classified in a way which is close to the classification of semi-simple algebraic groups over number fields, something which is well-understood. Moreover, if a semi-simple Lie group $G$ is non-compact and not isogenous to $\tn{PO}(n,1)$ or $\tn{PU}(n,1)$, then every irreducible lattice $\Gamma \subset G$ is arithmetic \cite{MR0492072, MR1147961, MR1215595}. Thus, to describe all lattices, one is left with the problem of classifying non-arithmetic lattices in $\tn{PO}(n,1)$ and in $\tn{PU}(n,1)$, a complicated task which remains completely open in general. 

In this paper, we focus on $\rm{PO}(n,1)$. 
Our goal is twofold. Firstly, we provide a new construction of lattices in the Lie group $\rm{PO}(n,1)$ for arbitrary $n  \geq 2$. 
Secondly, we show that for an explicit sequence of lattices $\set{\Gamma_n \subset \PO(n,1)}_{n \geq 2}$ constructed in this way, the lattice $\Gamma_n$ is non-arithmetic for each $n \geq 2$. 

Our construction can be sketched as follows. We consider complex hyperbolic space $\CC H^n$ for some $n \geq 2$, together with an arithmetic lattice $L \subset \rm{PU}(n,1)$ and a set of anti-isometric involutions $\alpha \colon \CC H^n \to \CC H^n$ that are chosen in a suitable way. Since the fixed locus $\RR H^n_\alpha \coloneqq (\CC H^n)^\alpha$ is isomorphic to real hyperbolic $n$-space $\RR H^n$, one obtains a real hyperbolic orbifold quotient space $L_\alpha \setminus \RR H^n_\alpha$ for each of our chosen involutions $\alpha$, where $L_\alpha \coloneqq \rm{Stab}_L(\RR H^n_\alpha)$ is the stabilizer of $\RR H^n_\alpha \subset \CC H^n$ in the group $L$. We glue the spaces $L_\alpha \setminus \RR H^n_\alpha$ for varying $\alpha$ together along suitable geodesic subspaces to obtain a metric space $M$, which we then provide with a real hyperbolic orbifold structure. The latter turns out to be complete, so that each connected component $M_i \subset M$ is isometric to $\RR H^n$ modulo a certain lattice $\Gamma_i \subset \rm{PO}(n,1)$. For a precise statement, see Theorem \ref{theorem:introduction:uniformization} below. The conjugacy classes of the lattices $\Gamma_i$ form the output of our construction. 

Subsequently, we show that for some choices made in the above construction, there exists a connected component $M_i \subset M$ such that any lattice $\Gamma_i \subset \PO(n,1)$ with $M_i \cong \Gamma_i \setminus \RR H^n$ is non-arithmetic; see Theorem \ref{theorem:introduction:non-arithmeticity} below. Moreover, the volume of $M_i \cong \Gamma_i \setminus \RR H^n$ can be explicitly calculated in some cases, see Section \ref{section:volume-comp}. 

Philosophically, this construction seems to ressemble the construction of Gromov--Piatetski-Shapiro \cite{gromovshapiro}, in the sense that we glue hyperbolic orbifolds along geodesic subspaces. In technical detail, however, the two constructions are different (see Section \ref{section:arithmetic-nature}). Therefore, we ask in Question \ref{question:gromovshapiro} whether for each $n \geq 2$, the conjugacy class of non-arithmetic lattices $\Gamma \subset \PO(n,1)$ we obtain via Theorem \ref{theorem:introduction:non-arithmeticity} is commensurable with any of the Gromov--Piatetski-Shapiro conjugacy classes.

Our results also have applications to the theory of real moduli spaces. Let $\mr M_s^{\RR}$ be the moduli space of stable real cubic surfaces. 
For one of the non-arithmetic lattices $\Gamma \subset \PO(4,1)$ we construct in the above way, we provide a homeomorphism $\mr M_s^{\RR}  \cong \Gamma \setminus \RR H^4$ restricting to an isomorphism of orbifolds on the locus of smooth cubics. 
Recall that Allcock, Carlson and Toledo have provided $\mr M_s^\RR$ with a non-arithmetic complete hyperbolic orbifold structure in their work \cite{realACTsurfaces}. The hyperbolic orbifold structure we define on $\mr M_s^\RR$ via the above homeomorphism $\mr M_s^{\RR}   \cong \Gamma \setminus \RR H^4$ agrees with the one in \emph{loc.\ cit.}\ (cf.\ Theorem \ref{theorem:reprove-ACT}). 

\newpage 
\subsection{Construction of the lattices} \label{section:outline}

Let us explain our first main result in more detail. For any number field $K$, we let $\OO_K$ denote its ring of integers. Recall that a number field $K$ is called a CM field when $K$ is an imaginary quadratic extension $K \supset F$ of a number $F$ which is totally real, meaning that the image of every embedding $F \hookrightarrow \CC$ is contained in $\RR \subset \CC$. Significant examples of CM fields are imaginary quadratic field extensions of $\QQ$ and cyclotomic number fields. 

Central in this paper is the notion of \emph{admissible hermitian lattice}. By this, we mean a pair $$\mr L = (K, \Lambda)$$ where $K$ is a CM field whose different ideal is generated by a totally imaginary element, and where, for some $n \geq 1$, $\Lambda$ is a finite free $\OO_K$-module equipped with a non-degenerate hermitian form $h \colon \Lambda \times \Lambda \to \OO_K$ of signature $(n,1)$ with respect to one embedding $\tau \colon K \hookrightarrow \CC$, and of signature $(n+1,0)$ with respect to the embeddings $\varphi \colon K \hookrightarrow \CC$ non-conjugate to $\tau$ (for an explanation of these notions, see Definition \ref{definition:introduction:admissible} in Section \ref{sec:preliminaries}). The \emph{rank} of $\mr L$ is the rank of $\Lambda$ as finite free $\OO_K$-module. 

Let $\mr L = (K, \Lambda)$ be an admissible hermitian lattice of rank $n+1$. Define $$V \coloneqq \Lambda \otimes_{\OO_K, \tau} \CC.$$ Then $V$ is a complex vector space of dimension $n+1$ equipped with a hermitian form $h \colon V \times V \to \CC$ of signature $(n,1)$. Let $\CC H^n$ denote the complex hyperbolic space of lines $\ell \subset V$ which are negative with respect to the hermitian form $h$. 

Let $\Aut(\Lambda)$ be the group of $\OO_K$-linear automorphisms $\phi \colon \Lambda \xrightarrow{\sim} \Lambda$ compatible with the hermitian form $h$. 
Let $\mu_K \subset \OO_K^\ast$ be the torsion subgroup of $\OO_K^\ast$. Since $\mu_K$ consists of roots of unity, it acts on $\Lambda$ by multiplication. We define $$L \coloneqq \Aut(\Lambda)/\mu_K \subset \rm{Isom}(\CC H^n). $$

Let $\mr R$ be the set of elements $r \in \Lambda$ with $h(r,r) = 1$, and for $r \in \mr R$, we let $H_r \coloneqq \set{\ell \in \CC H^n \colon h(r, \ell) = 0}$. Define $\mr H \coloneqq \cup_{r \in \mr R} H_r$. The hyperplane arrangement $\mr H \subset \CC H^n$ 
turns out to be an \emph{orthogonal arrangement}, in the sense of \cite{orthogonalarrangements}, see Theorem \ref{orthogonal} below. In other words, if two different hyperplanes $H_r \neq H_t$ for $r, t \in \mr R$ meet, then they are orthogonal along their intersection. 

By Dirichlet's unit theorem, $\mu_K$ is a finite cyclic group; let $\zeta \in \mu_K$ be a generator. For $r \in \mr R$, define an isometry $\phi_r \colon \Lambda \to \Lambda$ as follows:
\begin{equation*}
\phi_r(x) = x-(1-\zeta)h(x,r) \cdot r, \quad \quad x \in \Lambda. 
\end{equation*}
We call $\phi_r$ a \emph{reflection in the hyperplane $H_r$}. 
For $x \in \CC H^n$, let $$G(x) \coloneqq \langle \phi_r \mid x  \in H_r \rangle \subset L$$ be the subgroup generated by reflections $\phi_r$ in the hyperplanes $H_r$ with $x \in H_r$. 

Let
$
\sigma \colon K \to K
$
be the involution whose fixed locus is the maximal totally real field $F \subset K$. An $\OO_F$-linear bijection $\alpha \colon \Lambda \xrightarrow{\sim} \Lambda$ is called \emph{anti-unitary} if $\alpha(\lambda x) = \sigma(\lambda) \alpha(x)$ and $h(\alpha(x), \alpha(y)) = \sigma(h(x,y))$ for each $\lambda \in \OO_K$ and $x,y \in \Lambda$. 
We let $\mr A$ be the set of anti-unitary involutions, and let $P\mr A = \mu_K \setminus \mr A$ be the quotient of $\mr A$ by the action of $\mu_K$ on $\mr A$ by multiplication. For each $\alpha \in P\mr A$, we define 
\begin{align} \label{align:definition:RH-STAB}
\RR H^n_\alpha \coloneqq \left( \CC H^n \right)^\alpha  \quad \text{ and }  \quad L_\alpha \coloneqq \rm{Stab}_{L}(\RR H^n_\alpha). 
\end{align}
Then $\RR H^n_\alpha$ is isometric to real hyperbolic $n$-space $\RR H^n$ (see Lemma \ref{hyperbolic}), and we prove that the natural embedding $L_\alpha \hookrightarrow \rm{Isom}(\RR H^n_\alpha) \cong \PO(n,1)$ identifies $L_\alpha$ with an arithmetic lattice in $\PO(n,1)$ (see Theorem \ref{th:crucialthm-finitevolume}). 

The following definition is the starting point of our gluing construction.

\begin{definition} \label{definition:introduction:relation}
Define a relation $\sim$ on the disjoint union $\sqcup_{\alpha \in P\mr A} \RR H^n_\alpha$ in the following way. Let $(x,\alpha), (y, \beta) \in \sqcup_{\alpha \in P\mr A} \RR H^n_\alpha$, where $x \in \RR H^n_\alpha$ and $y \in \RR^n_\beta$, for $\alpha, \beta \in P\mr A$. Then $(x,\alpha) \sim (y,\beta)$ if and only if 
$x = y \in \CC H^n$ and 
$\beta \circ \alpha \in G(x)$. 
\end{definition}
Then $\sim$ is an equivalence relation, see Lemma \ref{eqrel}. 
Moreover, by Lemma \ref{lemma:pgammaaction}, the action of $L$ on $\CC H^n$ induces an action of $L$ on $\sqcup_{\alpha \in P\mr A} \RR H^n_\alpha$ compatible with $\sim$. Therefore, $L$ acts naturally on the quotient space $\left( \sqcup_{\alpha \in P\mr A} \RR H^n_\alpha \right)/_\sim$. 
We define
\begin{align} \label{align:M}
M(\mr L)  
 \coloneqq L \setminus \left(\left( \coprod_{\alpha \in P\mr A} \RR H^n_\alpha \right)/_\sim\right).  
\end{align}
Let $C \mr A \subset P\mr A$ be a set of representatives for the quotient $L \setminus P\mr A$, where, as above, $P\mr A$ is the set of $\mu_K$-equivalence classes of anti-unitary involutions $\alpha \colon \Lambda \to \Lambda$ and where $L = \Aut(\Lambda)/\mu_K$ acts on $P\mr A$ by conjugation. 
Recall that a \emph{real hyperbolic orbifold} is a second countable Hausdorff space locally modeled on $\RR H^n$ modulo finite isometry groups, 
see e.g.\ \cite[Chapter 13]{Thurston80}, \cite[Chapter 1]{ruan-stringy} or \cite{lange-orbifolds-metric}. 

The following theorem is the first main result of this paper. 
 
\begin{theorem} \label{theorem:introduction:uniformization}
Let 
$\mr L$ be an admissible hermitian lattice of rank $n+1 \geq 2$. 
Define 
the topological space $M(\mr L)$ as 
in equation \eqref{align:M}.  The following assertions are true.  
\begin{enumerate}
\item 
\label{oopensuborbifold}
There is a canonical real hyperbolic orbifold structure on the topological space $M(\mr L)$ together with a natural open immersion of hyperbolic orbifolds \begin{align}\label{align:open-immersion} \coprod_{\alpha \in C\mr A}  L_\alpha \setminus \left(\RR H^n_\alpha - \mr H \right)  \longhookrightarrow M(\mr L)\end{align}
whose complement has measure zero. 
\item \label{connunif} For each connected component $M_i$ of the orbifold $M(\mr L)$, there exists a lattice $\Gamma_i \subset \textnormal{PO}(n,1)$ and an isomorphism of real hyperbolic orbifolds $$M_i \cong \Gamma_i \setminus  \RR H^n.
$$ 
\item \label{item:volume}The cardinality of the set $C\mr A \cong L \setminus P\mr A$ and the volume $\rm{vol}(M(\mr L))$ of the hyperbolic orbifold $M(\mr L)$ are both finite. Furthermore, we have: 
$$
\rm{vol}(M(\mr L)) = \rm{vol}\left(\coprod_{\alpha \in C\mr A} L_\alpha \setminus \RR H^n_\alpha \right) = \sum_{\alpha \in C\mr A}\rm{vol}\left( L_\alpha \setminus \RR H^n_\alpha\right). 
$$
\end{enumerate}
\end{theorem}

\begin{example} \label{introduction:example1}
Let $\mr L$ be an admissible hermitian lattice of rank $n+1\geq 2$. For an anti-unitary involution $\alpha \colon \Lambda \to \Lambda$, let $M(\mr L, \alpha) \subset M(\mr L)$ be the connected component containing the image of the natural map $\RR H^n_\alpha \to M(\mr L)$. If we have $\RR H^n_\alpha \cap \mr H = \emptyset$, 
then $M(\mr L, \alpha) \cong L_\alpha \setminus \RR H^n_\alpha$. 
\end{example}

\subsection{Arithmetic nature of the lattices} \label{section:arithmetic-nature}

It is natural to ask whether which of the lattices one produces in the above way, are arithmetic. Let us make this question more precise. 
%
As a corollary of Theorem \ref{theorem:introduction:uniformization}, for each $n \geq 2$ one gets an assignment 
\begin{align} \label{align:function:Sn-lattice}
\begin{split}
\left\{
(\mr L, \alpha) \text{ with } \mr L = (K,\Lambda) \text{ an admissible \phantom{nnn}} \right. \\ \left. \text{hermitian lattice of rank $n+1$ and $\alpha \in P\mr A$}
\right\}
\end{split}
\; \longrightarrow \;
\begin{split}
\left\{
\text{conjugacy classes of \phantom{n}}
\right. \\ \left. 
\text{lattices }
\Gamma \subset \PO(n,1)
\right\}
\end{split}
\end{align}
defined by sending $(\mr L, \alpha)$ to 
the conjugacy class 
of any lattice $$\Gamma(\mr L, \alpha) \subset \PO(n,1)$$ that satisfies $M(\mr L, \alpha) \cong \Gamma(\mr L, \alpha) \setminus \RR H^n$, where $M(\mr L,\alpha)$ is the connected component of $M(\mr L)$ that contains the image of the natural map $\RR H^n_\alpha \to M(\mr L)$. We would like to understand which conjugacy classes in the image of \eqref{align:function:Sn-lattice} are arithmetic. 

\begin{example} 
Let $n \geq 1$ be an integer, and let $\mr L$ be an admissible hermitian lattice of rank $n+1$. If $\alpha \colon \Lambda \to \Lambda$ is an anti-unitary involution such that $\RR H^n_\alpha \cap \mr H = \emptyset$, 
then $[\Gamma(\mr L, \alpha)]
 = [L_\alpha]$ as conjugacy classes of lattices in $\PO(n,1)$, 
see Example \ref{introduction:example1}. Moreover, the lattice $L_\alpha \subset \PO(n,1)$ is arithmetic as we prove in Theorem \ref{th:crucialthm-finitevolume}.  
\end{example}
It seems natural to ask whether this is the only obstruction to non-arithmeticity. 

\begin{question} \label{question:introduction}
Let $n \geq 2$ be an integer. Let $\mr L = (K, \Lambda)$ be an admissible hermitian lattice of rank $n+1$, and $\alpha \colon \Lambda \to \Lambda$ an anti-unitary involution such that $\RRH^n_\alpha \cap \mr H \neq \emptyset$. 
Is it true that the conjugacy class of lattices $\Gamma(\mr L, \alpha) \subset \PO(n,1)$ associated to $(\mr L, \alpha)$ by the assignment \eqref{align:function:Sn-lattice} is non-arithmetic?
\end{question}

For each $n \geq 2$, we can answer this question positively in one particular case, yielding the second main result of this paper. Let $\zeta_3 \coloneqq e^{2 \pi i / 3} \in \CC$ 
and consider the admissible hermitian lattice  \begin{align} \label{align:eisenstein-lattice}
\mr L_{\rm{eis}}^n \coloneqq \left( \QQ(\zeta_3), \ZZ[\zeta_3]^{n,1} \right) ,
\end{align}
where 
$\ZZ[\zeta_3]^{n,1}$ is the free $\ZZ[\zeta_3]$-module of rank $n+1$ equipped with the hermitian form $h = \rm{diag}(-1, 1, \dotsc, 1) $ defined as $h(x,y) = - x_0 \bar y_0 + x_1 \bar y_1 + \cdots + x_n \bar y_n$. We call $\mr L_{\rm{eis}}^n$ the \emph{Eisenstein} hermitian lattice of rank $n+1$. Define an anti-unitary involution $\alpha_0 \colon \ZZ[\zeta_3]^{n,1} \to \ZZ[\zeta_3]^{n,1}$ by $\alpha_0(x) = \bar x$, 
let $M(\mr L_{\rm{eis}}^n, \alpha_0)  \subset M(\mr L_{\rm{eis}}^n)$ be the connected component containing the image of the map $\RR H^n_{\alpha_0} \to M(\mr L_{\rm{eis}}^n)$, and let
\begin{align} \label{lattice:eisenstein:non-arithmetic}
\Gamma^n_{\rm{eis}} \subset \PO(n,1) 
\end{align}
be a lattice such that  $M(\mr L_{\rm{eis}}^n, \alpha_0) \cong \Gamma^n_{\rm{eis}} \setminus \RR H^n$, see Theorem \ref{theorem:introduction:uniformization}.

\begin{theorem} \label{theorem:introduction:non-arithmeticity}
For each $n \in \ZZ_{\geq 2}$, 
the lattice $\Gamma^n_{\rm{eis}} \subset \PO(n,1)$ is non-arithmetic.
\end{theorem}

We point out that for $n = 2,3,4$, the lattice $\Gamma^n_{\rm{eis}} \subset \PO(n,1)$ is not cocompact. See Remark \ref{remark:non-cocompact} below.  

The way in which we obtain a non-arithmetic lattice in $\PO(n,1)$ for $n \geq 2$ seems to be an orbifold analogue of the construction of non-arithmetic lattices provided by Gromov and Piatetski--Shapiro \cite{gromovshapiro}. Apart from the analogy, their construction is technically different from ours. 
Gromov and Piatetski--Shapiro glue hyperbolic manifolds with boundary whose fundamental groups are Zariski dense in $\rm{PO}(n,1)$ and lie in non-commensurable arithmetic lattices. Here, we glue real hyperbolic orbifolds with boundary and corners whose orbifold fundamental groups are generally not Zariski dense in $\rm{PO}(n,1)$ (but do lie in arithmetic lattices that are generally not all commensurable). In light of these differences, it makes sense to ask whether the lattices we construct are essentially different.

\begin{question} \label{question:gromovshapiro}
Let $n \in \ZZ_{\geq 2}$. Consider the non-arithmetic lattice $\Gamma_{\rm{eis}}^n \subset \PO(n,1)$ defined in \eqref{lattice:eisenstein:non-arithmetic} above (cf.\ Theorem \ref{theorem:introduction:non-arithmeticity}). 
Is the lattice $\Gamma^n_{\rm{eis}} \subset \PO(n,1)$ commensurable with any of the Gromov--Piatetski-Shapiro lattices (cf.\ \cite{gromovshapiro})? Is $\Gamma^n_{\rm{eis}} \subset \PO(n,1)$ commensurable with any other non-arithmetic lattice in $\PO(n,1)$ currently known to exist, such as one of those constructed by  Belolipetsky and Thomson \cite{MR2821431}?
\end{question}

Next, we apply Theorem \ref{theorem:introduction:uniformization} to another sequence of admissible hermitian lattices $\mr L$, defined in the following way. Let $p>3$ be a prime number, let $\zeta_p = e^{2 \pi i/p} \in \CC$, and assume that the class number of $\QQ(\zeta_p)$ is odd. Let $\lambda$ be a unit in the ring of integers of $\QQ(\zeta_p+\zeta_p^{-1})$ such that $\lambda$ is positive for one embedding $\QQ(\zeta_p+\zeta_p^{-1}) \hookrightarrow \RR$ and negative for all other embeddings $\QQ(\zeta_p +\zeta_p^{-1})\hookrightarrow \RR$, and such that $\lambda \not \equiv -1 \bmod (1-\zeta_p)$ (such an element exists, see Lemma \ref{lemma:narkiewicz} in Section \ref{sec:standard}). 
For $n \geq 2$, put 
$$\mr L_{\zeta_p}^n(\lambda) \coloneqq (\QQ(\zeta_p), \ZZ[\zeta_p]^{n,1}_\lambda).$$
Here, $\ZZ[\zeta_p]^{n,1}_\lambda$ is the free $\ZZ[\zeta_p]$-module of rank $n+1$ equipped with the hermitian form $h$ defined as
$h(x,y) = - \lambda \cdot x_0 \bar y_0 + x_1 \bar y_1 + \cdots + x_n \bar y_n$. It is readily verified that $\mr L_{\zeta_p}^n(\lambda)$ is an admissible hermitian lattice. Define an anti-unitary involution $\alpha_0 \colon \ZZ[\zeta_p]^{n,1}_{\lambda} \to \ZZ[\zeta_p]^{n,1}_{\lambda}$ as $\alpha_0(x) = \bar x$, let $M(\mr L_{\zeta_p}^n(\lambda), \alpha_0)  \subset M(\mr L_{\zeta_p}^n(\lambda))$ be the connected component containing the image of the natural map $\RR H^n_{\alpha_0} \to M(\mr L_{\zeta_5}^n(\lambda))$, and let
\begin{align} \label{lattice:cyclotomic-5:non-arithmetic:new}
\Gamma^n_{\zeta_p}(\lambda) \subset \PO(n,1) 
\end{align}
be a lattice such that $M(\mr L_{\zeta_p}^n(\lambda), \alpha_0) \cong \Gamma^n_{\zeta_5}(\lambda) \setminus \RR H^n$, see Theorem \ref{theorem:introduction:uniformization}. 

\begin{theorem} \label{theorem:non-arithmeticity:p-arbitrary}
Let $p>3$ and $\lambda$ be as above. 
For $n \geq 2$, let $\Gamma^n_{\zeta_p}(\lambda) \subset \PO(n,1) $ be the lattice defined in equation \eqref{lattice:cyclotomic-5:non-arithmetic:new} above. Assume that each anti-unitary involution of $\ZZ[\zeta_p]^{2,1}_\lambda$ is $\Aut(\ZZ[\zeta_p]^{2,1}_\lambda)$-conjugate to an anti-unitary involution of the form $(x_0, x_1, x_2) \mapsto (\epsilon_0 \bar x_0, \epsilon_1 \bar x_1, \epsilon_2 \bar x_2)$ with $\epsilon_i \in \set{\pm 1}$. Assume moreover that $\Gamma^2_{\zeta_p}(\lambda) \subset \PO(2,1) $ is non-arithmetic. Then $\Gamma^n_{\zeta_p}(\lambda) \subset \PO(n,1) $ is non-arithmetic for all $n \geq 2$. 
\end{theorem}

We prove an analogous result for admissible hermitian lattices of the form $\mr L = (\QQ(\sqrt{-p}), (\OO_{\QQ(\sqrt{-p})})^{n,1})$, where $p \geq 3$ is an arbitrary prime number, see Theorem \ref{theorem:totally-geodesic-standard} in Section \ref{sec:nonarithmetic}. For $p=3$, this reduction to the case $n=2$ in the proof of the non-arithmeticity of $\Gamma^n_{\eis}$ is fact what allows us to prove Theorem \ref{theorem:introduction:non-arithmeticity} in Section \ref{section:proof-of-theorem}.  

In a follow-up paper, see \cite{gaay-hyperbolic}, 
we use the theory developed here to investigate the structure of the moduli space of stable real binary quintics, thereby extending results of Allcock, Carlson and Toledo (see Section \ref{section:uniformization} below). 
We will identify this moduli space with the space $M(\mr L_{\zeta_5}^2(\lambda))$, where $\lambda = \zeta_5 + \zeta_5^{-1}$, and use this description to prove that the conditions in Theorem \ref{theorem:non-arithmeticity:p-arbitrary} are verified in this case. In particular, by Theorem \ref{theorem:non-arithmeticity:p-arbitrary}, 
this will prove non-arithmeticity of $\Gamma^n_{\zeta_5}(\lambda) \subset \PO(n,1) $  for all $n \geq 2$. We will also show that the lattice $\Gamma^2_{\zeta_5}(\lambda) \subset \PO(2,1) $ is cocompact. 


\subsection{Volume} \label{section:volume-comp}
We return our attention to the Eisenstein hermitian lattices $\mr L_{\rm{eis}}^n = (\QQ(\zeta_3), \ZZ[\zeta_3]^{n,1})$ for $n \geq 2$. Let $\mr A_\eis^n$ be the set of anti-unitary involutions of $\ZZ[\zeta_3]^{n,1}$, let $C\mr A_\eis^n$ be a set of representatives for the action of $L^n \coloneqq \Aut(\ZZ[\zeta_3]^{n,1})/\langle -\zeta_3 \rangle$ on $P\mr A_\eis^n = \mr A_\eis^n / \langle -\zeta_3 \rangle$, 
and define $\RR H^n_\alpha$ and $L_\alpha^n \subset \rm{Isom}(\RR H^n_\alpha)$ as in equation \eqref{align:definition:RH-STAB}. 
For $i = 0, \dotsc, n$, consider the integral quadratic form $\Psi_i^n$ defined as 
\[
\Psi_i^n(x_0, \dotsc, x_n) = - x_0^2 + 3 x_1^2 + \cdots + 3x_i^2 + x_{i+1}^2 + \cdots + x_n^2.
\]
We use Theorem \ref{theorem:introduction:uniformization} to prove the following result. 



\begin{theorem} \label{theorem:volume}
Let $n \in \ZZ_{\geq 2}$. 
Let $\Gamma^n_{\rm{eis}}  \subset \PO(n,1)$ be the non-arithmetic lattice defined in equation \eqref{lattice:eisenstein:non-arithmetic} above (cf.\ Theorem \ref{theorem:introduction:non-arithmeticity}). Then there are a union of geodesic subspaces $\mr H_i \subset \RR H^n$ for $i \in \set{0, \dotsc, n}$ and open immersions
\[
\coprod_{i = 0}^n  \rm{PO}(\Psi_i^n,\ZZ) \setminus \left(\RR H^n - \mr H_i \right) \longhookrightarrow \Gamma^n_{\rm{eis}}   \setminus \RR H^n \longhookrightarrow M(\mr L_{\rm{eis}}^n).  
\]
In particular, 
\begin{align} \label{bound:1}
\sum_{i = 0}^n \rm{Vol}\left(  \rm{PO}(\Psi_i^n,\ZZ) \setminus \RR H^n \right) \leq &\rm{Vol}\left( \Gamma^n_{\rm{eis}}   \setminus \RR H^n   \right), \\
\label{bound:2}
&\rm{Vol}\left( \Gamma^n_{\rm{eis}}   \setminus \RR H^n  \right) \leq 
\sum_{\alpha \in C\mr A} \rm{Vol}(L^n_\alpha \setminus \RR H^n_\alpha). 
\end{align}

\end{theorem}

We note that there is a canonical bijection between $L^n \setminus P\mr A_\eis^n$ and the non-abelian cohomology group $\rm H^1(G, L^n)$, where $G = \ZZ/2 $ acts on the group $L^n =  \Aut(\ZZ[\zeta_3]^{n,1})/\langle -\zeta_3 \rangle$ via conjugation by the involution $\alpha_0 \colon \ZZ[\zeta_3]^{n,1} \to \ZZ[\zeta_3]^{n,1}$ defined as $\alpha_0(x) = \bar x$, see Proposition \ref{proposition:galois}. 
Moreover, by work of Prasad \cite{prasad}, there is a closed formula for the volume 
of the arithmetic orbifold $ \rm{PO}(\Psi_i^n,\ZZ) \setminus \RR H^n$; see also \cite{MR2124587, MR2974199}. For instance, consider the quadratic form $\Psi_0^{n}(x_0, \dotsc, x_n) = -x_0^2 + x_1^2 + \cdots + x_n^2$ for $n = 2r \geq 2$ even; by \cite[Theorem 6]{ratcliffe-volumes}, we have:
\begin{align*}
\rm{Vol}\left(  \rm{PO}(\Psi_0^{2r},\ZZ) \setminus \RR H^{2r} \right)  = 
\left(2^r + \varepsilon_r \right) \cdot \frac{\pi^r}{(2r)!}  \prod_{k = 1}^r \va{B_{2k}} 
\end{align*}
with the $B_{2k}$ Bernoulli numbers, and with $\varepsilon_r = 1$ if  $r \equiv 0,1 \bmod 4$ and $\varepsilon_r = -1$ if $r \equiv 2,3 \bmod 4$. 

We expect that \eqref{bound:1} is an equality whenever $n$ is even:

\begin{conjecture} \label{conjecture:volume}
Let $n \in 2\ZZ_{\geq 1}$. 
Then \eqref{bound:1} is an equality; in other words, the volume of the non-arithmetic hyperbolic orbifold $M(\mr L_{\rm{eis}}^n, \alpha_0) \cong \Gamma^n_{\rm{eis}} \setminus \RR H^n$ equals
\[
\rm{Vol}\left( \Gamma^n_{\rm{eis}}   \setminus \RR H^n  \right)= \sum_{i = 0}^n \rm{Vol}\left(  \rm{PO}(\Psi_i^n,\ZZ) \setminus \RR H^n \right). 
\]
\end{conjecture}

Work of Allcock, Carlson and Toledo \cite{realACTnonarithmetic, realACTsurfaces} implies Conjecture \ref{conjecture:volume} for $n \in \set{2,4}$. In fact, if the cardinality of $L^n \setminus P\mr A_\eis^n =  \rm H^1(G, L^n)$ equals $n+1$, then 
\eqref{bound:1} and \eqref{bound:2} are both equalities, thus Conjecture \ref{conjecture:volume} holds for such $n$. This happens for $n =4$ by \cite[Theorem 4.1]{realACTsurfaces}, and for $n=2$ by Theorem \ref{useful-theorem} in Section \ref{section:proof-of-theorem}. We do not know if the equality $\# \rm H^1(G, L^n) = n + 1$ 
is a coincidence for $n \in \set{2,4}$, or holds for all $n \in 2 \ZZ_{\geq 1}$.

For $n = 3$, one has 
$\# \rm{H}^1(G, L^n) = 5 = n+2$ (cf.\ \cite[Remark 6]{realACTbinarysextics}), which implies that \eqref{bound:1} or \eqref{bound:2} must be a strict inequality in this case. 

Conjecture \ref{conjecture:volume} would provide a way to explicitly calculate the volume of the non-arithmetic quotients $\GamEis^n \sm \RR H^n$ for $n \geq 2$ even. To the best of our knowledge, there is not much known about explicit volume computations in the other known families of non-arithmetic quotients in arbitrary dimension. For a lattice that comes from an interbreeding construction \cite{gromovshapiro} or generalizations (c.f.~\cite{raimbault, gelander-counting}), its covolume is a rational linear combination of the volumes of the arithmetic quotients that are glued, see \cite[Corollary 1.5]{emery-volumes} or \cite[Theorem 1.12]{emery-hyperbolic}. Note, however, that the rational coefficients are not provided, and that in general the arithmetic pieces used for the interbreeding construction are not explicated.  

\subsection{Uniformization of real moduli spaces} \label{section:uniformization}

Theorems \ref{theorem:introduction:uniformization} and \ref{theorem:introduction:non-arithmeticity} also have applications to the study of moduli spaces of real algebraic varieties. 
Write $\mr C_s^{\RR}$ for the space of non-zero
cubic forms with real coefficients in four variables which are stable in the sense of geometric invariant theory, and let $\mr C_0^\RR \subset \mr C_s^\RR$ be the open subspace of cubic forms whose associated cubic surface is smooth. The orbifold quotients
$
\mr M_s^{\RR} = \GL_4(\RR) \setminus \mr C_s^{\RR} \supset \GL_4(\RR) \setminus \mr C_0^\RR = \mr M_0^\RR
$
are the respective moduli spaces of stable and smooth real cubic surfaces. By \cite[Theorem 1.2]{realACTsurfaces}, there are a non-arithmetic lattice $P\Gamma^\RR \subset \PO(4,1)$, a homeomorphism 
\begin{align} \label{align:ACT-stable}
\mr M_s^\RR \cong P\Gamma^\RR \setminus \RR H^4, 
\end{align}
and a $P\Gamma^\RR$-invariant union $\mr H' \subset \RR H^4$ of two- and three-dimensional subspaces of $\RR H^4$ such that \eqref{align:ACT-stable} restricts to an orbifold isomorphism 
$
\mr M_0^\RR \cong P\Gamma^\RR \setminus \left( \RR H^4 - \mr H'\right)$. 

Our theory in dimension $n = 4$ relates to this result as follows. Consider the 
open immersion 
 $
\coprod_{i = 0}^4  \rm{PO}(\Psi_i^4,\ZZ) \setminus \left(\RR H^4 - \mr H_i \right) \hookrightarrow \Gamma^4_{\rm{eis}}   \setminus \RR H^4, 
$
cf.\ Theorem \ref{theorem:volume}. 

\begin{theorem} \label{theorem:reprove-ACT}
Let $\mr M_s^\RR \supset \mr M_0^\RR$ be the moduli spaces of stable and smooth real cubic surfaces. Then the following assertions are true. 
\begin{enumerate}
\item There are a union of geodesic subspaces $\mr H_i \subset \RR H^4$ for $i = 0, \dotsc, 4$ and a homeomorphism \begin{align} \label{uniformization-moduli-dgf}\mr M_s^{\RR} \cong \Gamma_{\rm{eis}}^4 \setminus \RR H^4 \end{align} restricting to an orbifold isomorphism $\mr M_0^{\RR} \cong \coprod_{i = 0}^4  \rm{PO}(\Psi_i^4,\ZZ) \setminus \left(\RR H^4 - \mr H_i \right)$.
\item When $\mr M_s^\RR$ is given the orbifold structure induced by \eqref{align:ACT-stable}, then \eqref{uniformization-moduli-dgf} is an isomorphism of orbifolds. In particular, the lattices $P\Gamma^\RR$ and $\Gamma_{\rm{eis}}^4$ are conjugate. 
\end{enumerate}
\end{theorem}
Our results have similar connections with \cite[Section 5]{realACTnonarithmetic}. 
Namely, if $\mr N_s^\RR(\infty)$ is the moduli space of stable real binary sextics that have a double root at $\infty$, then $\mr N_s^\RR(\infty) \cong \Gamma_{\rm{eis}}^2 \sm \RR H^2$ and the lattice $\GamEis^2 \subset \PO(2,1)$ 
is conjugate to lattice $\Gamma^\RR_\infty$ defined in \cite[Section 5]{realACTnonarithmetic}. 
We will prove this in Theorem \ref{theorem:connection-ACT} in Section \ref{section:binary}. 

\begin{remark} \label{remark:non-cocompact}
By Theorems \ref{theorem:reprove-ACT} and \ref{theorem:connection-ACT}, the lattice $\Gamma_{\rm{eis}}^n \subset \PO(n,1)$ is not cocompact for $n = 2,4$, because the moduli spaces $\mr N_s^\RR(\infty)$ and $\mr M_s^\RR$ are not compact. 
\end{remark}

\subsection{Orthogonality of the hyperplane arrangement} 

Let $\mr L = (K, \Lambda)$ be an admissible hermitian lattice of rank $n+1 \geq 3$. Define 
$
\mr H = \cup_{r \in \mr R} H_r \subset \CC H^n
$
as in Section \ref{section:outline}. 
A crucial ingredient in the proof of Theorem \ref{theorem:introduction:uniformization} is the fact that the hyperplane arrangement $\mr H \subset \CC H^n$ is an orthogonal arrangement in the sense of \cite{orthogonalarrangements}. In other words, for $r, t \in \Lambda$ such that $h(r,r) = h(t,t) = 1$, $H_r \neq H_t$ and $H_r \cap H_t \neq \emptyset$, one has $h(r,t) = 0$. 
See Theorem \ref{orthogonal} in Section \ref{sec:orthogonalhyperplane}.

In view of \cite[Theorem 1.2]{orthogonalarrangements}, the following is a corollary of Theorem \ref{orthogonal}. 

\begin{theorem}
Let $n \geq 2$. 
Let $\mr L = (K, \Lambda)$ be an admissible hermitian lattice of rank $n+1$. Define the lattice $L \subset \rm{Isom}(\CC H^n)$ and the hyperplane arrangement $\mr H \subset \CC H^n$ as in Section \ref{section:outline}. 
Then the orbifold fundamental group $\pi_1^\textnormal{orb} \left( L \sm \left( \CC H^n - \mr H \right) \right)$ is not a lattice in any Lie group with finitely many connected components. 
\end{theorem}

\section*{Acknowledgements}
This research has received funding from the European Union’s Horizon 2020 research and innovation programme under the Marie Skłodowska-Curie grant agreement N\textsuperscript{\underline{o}}754362 \img{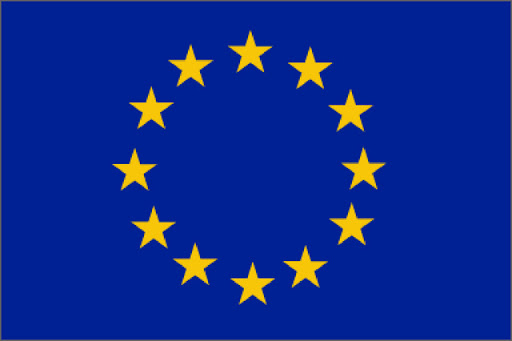}, from the European Research Council (ERC) under the European Union’s Horizon 2020 research and innovation programme under grant agreement N\textsuperscript{\underline{o}}948066 (ERC-StG RationAlgic), and from the ERC Consolidator Grant FourSurf N\textsuperscript{\underline{o}}101087365.

The author would like to thank Emiliano Ambrosi, Olivier Benoist and Jean Raimbault for helpful conversations concerning the content of this paper. 


\section{Admissible hermitian lattices} \label{sec:preliminaries}

The goal of this section is to introduce and study admissible hermitian lattices, who form the input of our gluing construction. 

\subsection{Admissible hermitian lattices} \label{set-up-1}  

\begin{notation} \label{notation:CMfield}
Throughout this paper, $K$ will denote a CM field with maximal totally real subfield $F \subset K$. 
We let $\sigma$ denote the unique involution $\sigma \colon K \to K$ such that for each embedding $\varphi \colon K \hookrightarrow \CC$, one has $\varphi \circ \sigma = \rm{conj} \circ \varphi$, where $\rm{conj} \colon \CC \to \CC$ denotes complex conjugation. We have $F = K^\sigma$ as subfields of $K$, and there is an element $a \in K$ with $K = F(a)$, $a^2 \in F$ and $\sigma(a) = -a$. We let $\OO_K \subset K$ and $\OO_F \subset F$ denote the respective rings of integers. 
\end{notation}

Let $K$ be a CM field and let $F \subset K$ and $\sigma \colon K \to K$ as in Notation \ref{notation:CMfield}. Let $n \geq 1$ be an integer and $\Lambda$ a free $\OO_K$-module of rank $n+1$. Let
\[
h \colon \Lambda \times \Lambda \longrightarrow \OO_K
\]
be a $\ZZ$-bilinear form which is $\OO_K$-linear in its first argument and which satisfies $h(y,x) = \sigma(h(x,y))$ for each $x,y \in \Lambda$. Note that for $\xi \in \OO_K$, we have $h(x, \xi y) = 
\sigma(\xi) h(x,y)$. In particular, $h$ is $\OO_F$-bilinear. For an embedding $\varphi \colon F \hookrightarrow \RR$, define
\begin{align*}
V_{\varphi} \coloneqq  \Lambda \otimes_{\OO_F, \varphi}  \RR.
\end{align*}
Then $V_\varphi$ is naturally a complex vector space of dimension $n+1$, equipped with a hermitian form 
$
h_{\varphi} \colon V_{\varphi} \times V_{\varphi} \to \CC
$
induced by $h$. Central in this paper is:

\begin{definition} \label{definition:introduction:admissible}
A \emph{hermitian lattice} is a pair $$\mr L = (K, \Lambda)$$ where $K$ is a CM field and where $\Lambda$ is a free $\OO_K$-module of finite rank equipped with a non-degenerate $\OO_F$-bilinear form $h \colon \Lambda \times \Lambda \to \OO_K$ which is $\OO_K$-linear in its first argument and which satisfies $h(y,x) = \sigma(h(x,y))$ for each $x,y \in \Lambda$. We say that a hermitian lattice $\mr L = (K, \Lambda)$ is \emph{admissible} if the following conditions hold:
\begin{enumerate}
\item \label{condition2.1}
The different ideal $\mf D_K \subset \OO_K$ (see e.g.\ \cite[Chapter III, \S 2]{Neukirch}) is generated by an element $\eta \in \OO_K$ that satisfies $\sigma(\eta) = -\eta$.
\item  \label{condition2.2} For one embedding $\tau \colon F \hookrightarrow \RR$, the hermitian form $h_\tau$ on the complex vector space $V_\tau = \Lambda \otimes_{\OO_F, \tau} \RR$ has signature $(n,1)$. For all other embeddings $\varphi \neq \tau$ of $F$ into $\RR$, the hermitian form $h_\varphi \colon V_\varphi \times V_{\varphi} \to \CC$ has signature $(n+1,0)$. 
\end{enumerate}
The \emph{rank} of $\mr L$ is the rank of $\Lambda$ as finite free $\OO_K$-module. 
\end{definition}

We will provide some examples of admissible hermitian lattices in Section \ref{section:examples} and in Section \ref{sec:standard}. 

\subsection{Unitary ball quotients} \label{set-up-2} 

Let $n$ be a positive integer. Let $\mr L = (K, \Lambda)$ be an admissible hermitian lattice of rank $n+1 \geq 2$. Let $F \subset K$ and  $\sigma \colon K \to K$ be as in Notation \ref{notation:CMfield}. 
For an element $x \in K$, we will sometimes use the notation $\overline{x} = \sigma(x)$. 
Let $\tau \colon F \hookrightarrow \RR$ be the embedding such that the hermitian form $h_\tau$ induced by $h$ has signature $(n,1)$, see Section \ref{set-up-1}. Define
\[
V \coloneqq V_\tau =  \Lambda \otimes_{\OO_F, \tau} \RR.\]
To simplify notation, we put $h = h_\tau$. Thus, 
\[
h \colon V \times V \longrightarrow \CC
\]
denotes the hermitian form of signature $(n,1)$ induced by the form $h \colon \Lambda \times \Lambda \to \OO_K$. 

Let $m$ be the largest positive integer for which the $m$-th cyclotomic field $\QQ(\zeta_m)$ can be embedded in $K$, where $\zeta_m = e^{2 \pi i /m} \in \CC$. Let $\zeta \in K$ be a primitive $m$-th root of unity in $K$, and define
\[
\mu_K = \langle \zeta \rangle \subset \ca O_K^\ast \subset \OO_K. 
\]
By the Dirichlet unit theorem, $\mu_K$ coincides with the torsion subgroup of $\OO_K^\ast$. 

Define $\Aut(\Lambda)$ as the unitary group of $\Lambda$ (i.e., the group of $\OO_K$-linear isomorphisms $\phi \colon \Lambda \xrightarrow{\sim} \Lambda$ that satisfy $h(\phi(x), \phi(y))$ for each $x,y \in \Lambda$) and let $L$ be the quotient of $\Aut(\Lambda)$ by $\mu_K$; thus, we have
\[
\Aut(\Lambda) = U(\Lambda)(\OO_K) \quad \tn{ and } \quad  L = \Aut(\Lambda)/\mu_K.
\]
An element $r \in \Lambda$ with $h(r,r) = 1$ is called a \textit{short root}. Let $\mr R \subset \Lambda$ be the set of short roots. For $r \in \mr R$, define isometries $\phi_r^i \colon V \to V$ as follows:
\begin{equation*}
\phi_r(x) = x-(1-\zeta)h(x,r) \cdot r, \quad \phi_r^i(x) = x-(1-\zeta^i) h(x,r) \cdot r, \quad i \in (\ZZ/m)^\ast. 
\end{equation*}
Note that $\phi^i_r \in \Aut(\Lambda)$ for $r \in \mr R$, and that $\phi_r^i = \phi_r \circ \cdots \circ \phi_r$ ($i$ times). In particular, $\phi_r^m = \id$. Let $\PP(V)$ be the space of complex lines in $V$, and let
\[
\CC H^n \coloneqq \set{ \ell = [v] \in \PP(V) \mid h(v,v) < 0} \subset \PP(V)
\] 
be the space of negative lines in $V$. Define
\begin{equation*}
H_r \coloneqq \{x \in \CC H^n: h(x,r) = 0 \} \; \tn{ for } \; r \in \mr R, \quad \tn{ and }\quad \mr H \coloneqq \bigcup_{r \in \mr R} H_r \white
 \subset \white \CC H^n. 
\end{equation*}
Here, for a point $x \in \CCH^n$, the condition $h(x,r) = 0$ in fact means that $h(x', r_\tau) = 0$ for the image $r_\tau \in V_\tau = V$ of $r$ under the canonical embedding $\Lambda \hookrightarrow \Lambda \otimes_{\OO_F,\tau} \RR = V_\tau = V$ and any non-zero representative $x' \in V$ of $x \in \CCH^n$.  

\begin{lemma} \label{lemma:beauville:finite}
The family of hyperplanes $(H_r)_{r \in \mr R}$ is locally finite. In particular, the hyperplane arrangement $\mr H \subset \CC H^n$ is a closed subset of $\CC H^n$. 
\end{lemma}
\begin{proof}
See \cite[Lemma 5.3]{beauvillecubicsurfaces} for the proof in case $\mr L = (\QQ(\zeta_3), \ZZ[\zeta_3]^{4,1})$. The general case, is similar; we will prove it for convenience of the reader. Let us first fix some notation. Let $\CC^{n,1} = (\CC^{n+1}, h_{n,1})$ be the complex vector space $\CC^{n+1}$ equipped with the hermitian form $h_{n,1}$ defined as $h_{n,1}(x,y) = -x_0\bar y_0 + x_1 \bar y_1 + \cdots + x_n \bar y_n$. We define a function $h_+ \colon \CC^{n+1} \times \CC^{n+1} \to \CC$ as  $h_+(x,y) = x_1\bar y_1 + \cdots + x_n \bar y_n$, so that $h_{n,1}(x,y) = h_+(x,y) - x_0 \bar y_0$. We put $\n{x} = \sqrt{h_+(x,x)}$ for $x \in \CC^{n+1}$. If we identify $\CC^n$ with the affine hyperplane $z_0 = 0$ in $\CC^{n,1}$, then $\n{\phantom{-}}$ induces the standard hermitian norm on $\CC^n \subset \CC^{n,1}$. 

Next, we fix an isomorphism $\varphi \colon V \cong \CC^{n,1}$ and let $M \subset \CC^{n,1}$ be the image of $\Lambda$ under $\varphi$. We define $\Delta_1 \subset M$ as $\Delta_1 = \varphi(\mr R)$, so that $\Delta_1$ is set of elements $\delta \in M$ such that $h_{n,1}(\delta,\delta) = 1$. Let $\bb H^n_\CC$ be the space of complex lines in $\CC^{n,1}$ which are negative for $h_{n,1}$. Then we have an isomorphism \begin{align*}&\psi \colon \bb H^n_\CC \xlongrightarrow{\sim} \BB^n(\CC) = \set{z \in \CC^n \colon \n{z}^2 = \va{z_1}^2 + \cdots + \va{z_n}^2 < 1},  \\
&\psi\left( [z_0 \colon \cdots \colon z_n]\right) = (z_1/z_0, \dotsc, z_n/z_0).\end{align*}
For $\delta \in M$, write $\delta = (\delta_0, \dotsc, \delta_n)$, and define
\[
H_{\delta} \coloneqq \set{(x_1, \dotsc, x_n) \in \BB^n(\CC) \mid - \bar \delta_0 + x_1 \bar \delta_1 + \cdots + x_n \bar \delta_n = 0}. 
\]
Under the composition of isomorphisms $\CC H^n \cong \HH^{n,1}_\CC \cong \BB^n(\CC)$, the hyperplane $H_r \subset \CC H^n$ for $r \in \mr R$ is identified with the hyperplane $H_\delta \subset \BB^n(\CC)$ for $\delta \in \Delta_1$. For $\delta \in \Delta_1$ and $x \in \BB^n(\CC)$, let $\wt x = (1, x_1, \dotsc, x_n) \in \CC^{n,1}$ and define
$
\va{h(x, \delta)} \coloneqq \va{h_{n,1}( \wt x, \delta)} = \va{
-\delta_0 + x_1 \bar \delta_1 + \cdots + x_n \bar \delta_n
}.
$
Observe that $H_\delta = \set{x \in \BB^n(\CC) \colon \va{h(x, \delta)}  = 0}$. 

We need to show that the family of hyperplanes $(H_\delta)_{\delta \in \Delta_1}$ is locally finite in $\BB^n(\CC)$. 
To prove this, let $x \in \BB^n(\CC)$. We want to show that for $\varepsilon$ small enough, the ball $B(x, \vep)$ meets only finitely many of the hyperplanes $H_\delta$. Since $x \in \BB^n(\CC)$, we have that $\n{x} < 1$. 
We choose $\vep \in \RR$ with $0 < \vep < 1$ so that $\n{x} < 1 - 2\vep$.

\begin{claim}\label{claim:beauville} There exists $M \in \RR_{>0}$ such that for each $\delta \in \Delta_1$ with $H_\delta \cap B(x,\vep) \neq \ES$, we have $\n{\delta} \leq M$. 
\end{claim}
\begin{proof}[Proof of Claim \ref{claim:beauville}]
Let $\delta \in \Delta_1$. 
Since $\n{\delta}^2 = \va{\delta_1}^2 + \cdots + \va{\delta_n}^2$ and $h_{n,1}(\delta,\delta) = - \va{\delta_0}^2 + \va{\delta_1}^2 + \cdots + \va{\delta_n}^2 = 1$, we get $\n{\delta}^2 =  1 + \va{\delta_0}^2$. Consequently,
\begin{align} \label{align:frac}
\frac{\va{\delta_0}}{\n{\delta}} = \sqrt{\frac{1}{\va{\delta_0}^{-2}+ 1}}.
\end{align}
If $\n{\delta}$ approaches infinity, then $\va{\delta_0} = \sqrt{\n{\delta}^2 - 1}$ approaches infinity and hence $\frac{\va{\delta_0}}{\n{\delta}} $ approaches $1$ in view of \eqref{align:frac}. 
In particular, there exists $M > 0 $ such that $$1 \geq \frac{\va{\delta_0}}{\n{\delta}} \geq 1-\vep \quad \text{ for all } \quad \delta \in \Delta_1 \quad\text{ such that } \quad\n{\delta} > M.$$ We fix this $M>0$. 

Next, let $\delta \in \Delta_1$ such that $H_\delta \cap B(x,\vep) \neq \ES$. We claim that $\n{\delta} \leq M$. 
To see this, let $x' \in H_\delta \cap B(x,\vep) $ be a point in the intersection. Define $\wt x = (1, x_1, \dotsc, x_n) \in \CC^{n,1}$ and $\wt x' = (1, x_1', \dotsc, x_n') \in \CC^{n,1}$. Then $h(x', \delta) = h_{n,1}(\wt x', \delta) = 0$ and $\n{x-x'}  
< \vep$. This gives:
\begin{align*}
\va{h(x, \delta)} = 
\va{h_{n,1}(\wt x - \wt x', \delta)}  = \va{h_+(\wt x-\wt x', \delta)}
 \leq \n{x - x'} \cdot \n{\delta}  < \vep \cdot \n{\delta}. 
\end{align*}
Moreover, we have $h(x,\delta) = h_{n,1}(\wt x, \delta) = h_+(\wt x, \delta) - \bar \delta_0$, and $$\va{h_+(\wt x, \delta)} = \va{x_1\bar \delta_1 + \cdots + x_n\bar \delta_n} \leq  \n{x} \cdot \n{\delta} < (1-2\vep) \cdot \n{\delta}.$$
For the sake of contradiction, assume that $\n{\delta} > M$. By our choice of $M$, we get $\va{\delta_0} \geq (1-\vep)\n{\delta}$. Then $$\va{h_+(\wt x, \delta)} \leq \n{x} \cdot \n{\delta} < 
(1 - 2 \vep) \cdot \n{\delta} \leq (1-2\vep) \cdot \frac{\va{\delta_0}}{(1-\vep)} < \va{\delta_0}.$$  Therefore, $ \va{h_+(\wt x, \delta)-\bar \delta_0} \geq \va{\delta_0} - \va{h_+(\wt x, \delta)}$ so that
\begin{align*}
\vep \cdot \n{\delta} &> \va{h(x, \delta)} = \va{h_+(\wt x, \delta)-\bar \delta_0} \\&\geq \va{\delta_0} - \va{h_+(\wt x, \delta)}
> (1-\vep) \cdot \n{\delta} -  (1-2\vep) \cdot \n{\delta} = \vep\cdot\n{\delta}.
\end{align*}
This yields the desired contradiction. Thus $\n{\delta} \leq M$. This proves Claim \ref{claim:beauville}.   
\end{proof}
By Claim \ref{claim:beauville}, for each $\delta \in \Delta_1$ such that $H_\delta \cap B(x,\vep) \neq \emptyset$, we have $\n{\delta} \leq M$. For such $\delta$, the equality $\n{\delta}^2 = 1 + \va{\delta_0}^2$ then implies $$\va{\delta_0} \leq \n{\delta} = \va{\delta_1}^2 + \cdots + \va{\delta_n}^2 \leq M.$$ 
Thus, the set of $\delta \in \Delta_1$ such that $H_\delta \cap B(x,\vep) \neq \emptyset$ is contained in the compact subset $K \coloneqq\set{(z_0,z_1, \dotsc, z_n) \in \CC^{n+1} \mid \va{z_0}^2 \leq \va{z_1}^2 + \cdots + \va{z_n}^2 \leq M} \subset \CC^{n+1}$. 
As the set of $\delta \in \Delta_1$ such that $H_\delta \cap B(x,\vep) \neq \emptyset$ is discrete, it is therefore finite. 
\end{proof}

\subsection{Orthogonal hyperplanes and complex reflections} \label{set-up-3} Let $\mr L = (K, \Lambda)$ be an admissible hermitian lattice of rank $n+1 \geq 2$. We continue with the notation of Sections \ref{set-up-1} and \ref{set-up-2}. 

The following result is a key technical ingredient in our gluing construction. 

\begin{theorem} \label{orthogonal}
Let $\mr L = (K, \Lambda)$ be an admissible hermitian lattice of rank $n+1 \geq 2$. 
Let $r, t \in \Lambda$ with $h(r,r) = h(t,t) = 1$, $H_r \neq H_t$ and $H_r \cap H_t \neq \emptyset$. Then $h(r,t) = 0$. 
\end{theorem}

The proof of Theorem \ref{orthogonal} shall be provided in Section \ref{sec:orthogonalhyperplane}. It depends on the results of Section \ref{unitaryshimura}, that describe the structure of the quotient $L \setminus \CC H^n$ as a certain moduli space of abelian varieties. We highlight that Sections \ref{unitaryshimura} and \ref{sec:orthogonalhyperplane} are independent of Sections \ref{sec:preliminaries} -- \ref{section:uniformization:realcubics}. 

Theorem \ref{orthogonal} implies that if $H_{r_1}, \dotsc, H_{r_k}$ for $r_i \in \mr R$ are mutually distinct, and if their common intersection is non-empty, then $\cap_{i = 1}^k H_{r_i} \subset \CC H^n$ is a totally geodesic subspace of codimension $k$. Moreover, for any $r \in \mr R$, the element $\phi_r \in \Aut(\Lambda)$ generates a finite subgroup $\langle \phi_r \rangle \subset \Aut(\Lambda)$ of order $m$, and that the restriction of the quotient map $\Aut(\Lambda) \to L$ to this subgroup $\langle \phi_r \rangle \subset \Aut(\Lambda)$ is injective. 

\begin{lemma} \label{finitereflectionorders}
Let $\phi \in \Aut(\Lambda)$. Assume $\phi \colon \CC H^n \to \CC H^n$ has order $m$ and restricts to the identity on $H_r \subset \CC H^n$ for some $r \in \mr R$. Then $\phi = a \cdot \phi_r^i$ for some $i \in \ZZ/m$ and $a \in \mu_K$. 
\end{lemma}
\begin{proof}
Let $\HH^n_\CC$ be the hyperbolic space attached to the standard hermitian space $\CC^{n,1}$ of dimension $n+1$. It is classical that
\begin{align} \label{align:classical-equation}
\text{Stab}_{U(n,1)}(\HH^{n-1}_\CC) = U(n-1,1) \times U(1),
\end{align}
where $U(1)=\{z \in \CC^\ast \colon \va{z}^2 = 1\}$. 
Define $W \coloneqq \langle r \rangle^\perp \subset V$, so that $V = W \oplus  \langle r \rangle$. For $\xi \in U(1)$, define $\phi_r^{\xi} \in \Isom(V)$ as $\phi_r^\xi(x) = x - (1 - \xi) h(x,r) r $, and consider the embedding 
$
\Isom(W) \times U(1) \hookrightarrow \Isom(V)$ defined as $ (\varphi, \xi) \mapsto \varphi \circ \phi_r^{\xi} = \phi_r^{\xi} \circ \varphi. 
$
By \eqref{align:classical-equation}, we get that
\begin{align} \label{align:classical-equation-2}
\Stab_{\Isom(V)}(H_r) = \Isom(W) \times U(1) \subset \Isom(V). 
\end{align}
Consider the element $\phi \in \Aut(\Lambda)$ in the statement of the lemma. 
Since $\phi$ fixes $H_r \subset \CC H^n$ pointwise, \eqref{align:classical-equation-2} implies that 
$
\phi = a \cdot \phi_r^{\xi}$ for some $ a \in \CC^\ast, \xi \in U(1)$. 
Since $\phi \in \Aut(\Lambda)$, we get $a \in \mu_K \subset \CC^\ast$, and as $\phi^m = \id \in \Isom(\CC H^n)$, we get $\xi \in \mu_K$. 
\end{proof}

\begin{proposition} \label{prop:remember}
Let $r, t \in \mr R$. The following are equivalent:
\begin{enumerate}%
\item One has $\phi_r = \phi_t \in \Aut(\Lambda)$. 
\item There exist $i,j \in \ZZ/m - \{0\}$ 
such that $\phi_r^i = \phi_t^j \in \Aut(\Lambda)$. 
\item There exist $a, b \in \OO_K - \{0\}$ with $\va{a}^2 = \va{b}^2$ such that $a \cdot r = b \cdot t \in \OO_K$. 
\item One has $H_r = H_t \subset \CC H^n$. 
\end{enumerate}
\end{proposition}
\begin{proof}
$[1 \implies 2]:$ This is clear. \\
$[2 \implies 3]:$ 
Assume $2$. 
Consider the equality
\[
\zeta^i \cdot r = \phi_r^i(r) = \phi_t^j(r) =  r -  (1 - \zeta^j)\cdot h(r,t) \cdot t \in \Lambda.
\]
Thus, there exist $a,b \in \OO_K - \{0\}$ with $a \cdot t = b \cdot r.$ We have $\va{a}^2 = h(a\cdot t, a\cdot t) = h(b\cdot r, b\cdot r) = \va{b}^2$. This implies $3$. 

$[3 \implies 1]:$ Assume $3$. Write $\xi = b/a \in K^\ast$; then $t = \xi \cdot r$ with $\va{\xi}^2 = 1$, so that 
\[
\zeta^i \cdot  r = \phi_r^i(r) = \phi_t^j(r) = \phi_{\xi  r}^j(r) = r - (1-\zeta^j)h(r, \xi r)  \xi r= 
\zeta^j \cdot r \in V,
\]
from which we conclude that $ i = j$. 
Moreover, for $x \in V$, one gets
\[
\phi_t(x) = \phi_{\xi \cdot r}(x) = x - (1 - \zeta) \cdot h(x, \xi \cdot r) \cdot \xi \cdot r = x - (1- \zeta) \cdot h(x,r) \cdot r = \phi_r(x). 
\]
Hence $\phi_r = \phi_t$, proving $1$. We conclude that $[1 \iff 2 \iff 3]$. 

$[3 \implies 4]:$ This is clear. 

$[4 \implies 2]:$ Assume that $H_r = H_t$. Then $\phi_t = \xi \cdot \phi_r^i$ for some $i \in \ZZ/m$ and $\xi \in \mu_K$, see Lemma \ref{finitereflectionorders}. Note that $i \neq 0$. Assume for a contradiction that $\xi \neq 1$; then the equality
\[
t = \phi_t(t) = \xi\cdot \phi_r^i(t) = \xi \cdot  t - \xi \cdot (1 - \zeta^i) h(t,r) \cdot r
\] shows that if we define $a = 1 - \xi$ and  $b = (\zeta^i-1) \cdot \xi \cdot h(t,r)$, we get
$
a\cdot  t = b\cdot r$ and $a,b \neq 0$ since $\xi \neq 1$. 
We have $\va{a}^2 = h(a\cdot t, a\cdot t) = h(b\cdot r, b\cdot r) = \va{b}^2$ hence $3$ holds, and hence $1$ by the above, so that $a \cdot \phi_r^i = \phi_t = \phi_r$ which implies $a =1 $ which is absurd. We conclude that $a = 1$, so that $\phi_t = \phi_r^i$ which proves $2$ since $i \in \ZZ/m-\set{0}$. This concludes the proof of the proposition.
\end{proof}

\begin{definition} \label{fullset}
Let $\ca H = \{H_r \mid r \in \mr R\}$. For $x \in \CC H^n$, define 
\begin{align*}
\ca H(x) &= \set{H \in \ca H \mid x \in H}, \\
G' (x) &= \langle \phi_r^i \in \Aut(\Lambda) \tn{ with } r \in \mr R, i \in \ZZ/m \mid  x \in H_r  \rangle, \\
G (x) &= \langle [\phi_r^i] \in \Aut(\Lambda)/\mu_K = L \tn{ with } r \in \mr R, i \in \ZZ/m \mid  x \in H_r  \rangle,
\end{align*}
The hyperplanes $H \in \ca H(x)$ are called the \emph{nodes} of $x$. We say that \emph{$x$ has $k$ nodes} if the cardinality of $\ca H(x)$ is $k$.  
\end{definition}

\begin{lemma} \label{G-G'}
Let $x \in \CC H^n$. The composition $G(x) \hookrightarrow \Aut(\Lambda) \to L$ is injective, hence defines an isomorphism $G'(x) \xrightarrow{\sim} G(x)$. 
\end{lemma}
\begin{proof}
Assume $\ca H(x) = \set{H_{r_1}, \dotsc, H_{r_k}}$, so that $G(x) = \langle \phi_{r_1}, \dotsc, \phi_{r_k} \rangle$ by Lemma \ref{Gxlemma}. Assume that for some $n_1, \dotsc, n_k \in \ZZ/m$, the element $\psi \coloneqq \phi_{r_1}^{n_1} \cdots \phi_{r_k}^{n_k}  \in G(x)$ maps to zero under the map $G(x) \to L$. Then $\psi  = \xi$ for some $\xi \in \mu_K \subset \Aut(\Lambda)$. Let $W \subset V$ be the orthogonal complement of $\langle r_1 \rangle \oplus \cdots \oplus \langle r_k \rangle$, so that  
\[
V = \langle r_1 \rangle \oplus \cdots \oplus \langle r_k \rangle \oplus W
\]
as complex hermitian vector spaces. Since the signature of $h \colon V \times V \to \CC$ is $(n,1)$, the subspace $W$ is non-zero. Moreover, $\psi$ is the identity on $W$, hence $\xi = 1 \in \mu_K$. It follows that $\psi = \id$, which shows that the map $G(x) \to L$ is injective as desired. 
\end{proof}

\begin{lemma} \label{Gxlemma}
Let $x \in \CCH^n$ be an element with $k$ nodes (see Definition \ref{fullset}). Let $\ca H(x) = \set{H_{r_1}, \dotsc, H_{r_k}}$ for some $r_i \in \mr R$. Then $G(x) = \langle [\phi_{r_1}], \dotsc, [\phi_{r_k}] \rangle \cong (\ZZ/m)^k.$
\end{lemma}
\begin{proof}
By Lemma \ref{G-G'}, it suffices to prove that  $G'(x) = \langle \phi_{r_1}, \dotsc, \phi_{r_k} \rangle \cong (\ZZ/m)^k.$ For this, let $r, t \in \mr R$. Then, for $z \in \Lambda$, one has 
\begin{align} \label{reflectionscommute}
\begin{split}
\phi_r^i(\phi_t^j(z))  
=z &- (1-\zeta^j)h(z,t)\cdot t  - (1 - \zeta^i) h(z, r) \cdot r \\
 &+ (1-\zeta^i)(1- \zeta^j) h(z,t)h(t,r) \cdot r.
\end{split}
\end{align} 
Suppose that $H_r, H_t \in \ca H(x)$, with $H_r \neq H_t$. By Theorem \ref{orthogonal}, we have $h(r,t) = 0$; by equation (\ref{reflectionscommute}), this implies that $\phi_r^i \circ \phi_t^j = \phi_t^j \circ \phi_r^i $ for each $i,j \in \ZZ/m$. We conclude that the group $G'(x)$ is abelian. 

Write $\ca H(x) = \set{H_{r_1}, \dotsc, H_{r_k}}$ for some $r_1, \dotsc, r_k \in \mr R$. By the above, we have $\langle \phi_{r_1}, \dotsc, \phi_{r_k} \rangle \cong (\ZZ/m)^k$. Thus, to prove the lemma, it suffices at this point to prove that $G'(x) = \langle \phi_{r_1}, \dotsc, \phi_{r_k} \rangle $. Let $\phi_r \in L$ with $r \in \mr R$ such that $h(x,r) = 0$. By definition, $G'(x)$ is generated by such elements. We claim that $\phi_r \in \langle \phi_{r_1}, \dotsc, \phi_{r_k} \rangle$. To see this, observe that $H_r \in \ca H(x)$. Thus $H_r = H_{r_i}$ for some $i \in \set{1, \dotsc, k}$. By Proposition \ref{prop:remember}, we have $\phi_r = \phi_{r_i}$ and this proves what we want. 
\end{proof}

\begin{lemma} \label{intersection:complexpart}
Let $\phi = \prod_{j  = 1}^k \phi_{r_j}^{i_j} \in \Aut(\Lambda)$ for some set $\{r_j\}$ of mutually orthogonal short roots $r_j \in \mr R$, where $i_j \not \equiv 0 \bmod m$ for each $j$. Then $\left(\CC H^n\right)^\phi = \cap_{j = 1}^k H_{r_j}$. 
\end{lemma}
\begin{proof}
Let $y \in V$ be representing an element in $\left(\CC H^n\right)^\phi$. Since the $r_i$ are orthogonal, and $\phi(y) = \lambda$ for some $\lambda \in \CC^\ast$, we have
$
\phi(y) = y - \sum_{j =  1}^k \left(1 - \zeta^{i_j}\right) h(y, r_j)r_j = \lambda y,
$
hence $(1-\lambda)y = \sum_{j = 1}^k \left(1 - \zeta^{i_j}\right) h(y, r_j)r_j  \in V$. But $y$ spans a negative definite subspace of $V$ while the $r_j$ span a positive definite subspace, so that we must have $1 - \lambda = 0 = \sum_{j = 1}^k \left(1 - \zeta^{i_j}\right) h(y, r_j)r_j$. Since the $r_j$ are mutually orthogonal, they are linearly independent; since $\zeta^{i_j} \neq 1$ we find $h(y, r_j) = 0$ for each $j$. Conversely, if $x \in \cap H_{r_j}$, then $\phi_{r_j}^{i_j}(x) = x$ for each $j$. The lemma follows. 
\end{proof}

\subsection{Examples} \label{section:examples}

Observe that the condition on the different ideal $\mf D_K \subset \OO_K$ in Definition \ref{definition:introduction:admissible} above is satisfied in several natural cases. 

\begin{example} \label{proposition:discr}
Suppose that $K/\QQ$ is an imaginary quadratic extension, or that $K = \QQ(\zeta_m)$ is a cyclotomic field for some integer $m \geq 3$. 
We claim that $\mf D_K = (\eta) \subset \OO_K$ for some element $\eta \in \OO_K$ such that $\sigma(\eta) = -\eta$. 
To see this, note first that if $K/\QQ$ is an imaginary quadratic field extension with discriminant $\Delta$, 
then $\mf D_K  
= (\sqrt \Delta)$ and the assertion is immediate. 
Thus, assume $K = \QQ(\zeta_m)$ with $\zeta_m = e^{2 \pi i/m} \in \CC$. 
Consider the subfield $F \coloneqq \QQ(\mu) \subset \QQ(\zeta_m) = K$, where $\mu \coloneqq \zeta_m + \zeta_m^{-1}$. 
Since $\OO_K = \ZZ[\zeta_m]$ 
we have $ \OO_K = \OO_F[\zeta_m]$. Let $f(x) \coloneqq x^2 - \mu x + 1 \in \OO_F[x]$ and notice that $f(x)$ is the minimal polynomial of $\zeta_m$ over $F$. We have $f'(\zeta_m) = 2 \zeta_m - \mu = \zeta_m - \zeta_m^{-1}$ hence $\mf D_{M/F} = \left( f'(\zeta_m) \right) = \left(\zeta_m - \zeta_m^{-1}\right)$, see \cite[Chapter III, Proposition 2.4]{Neukirch}. 
By \cite{Liang1976}, we know that $\OO_F = \ZZ[\mu]$. Hence, if $g(x) \in \ZZ[x]$ is the minimal polynomial of $\mu$ over $\QQ$, then $\mf D_{F/\QQ} = (g'(\mu))$ by \cite[Chapter III, Proposition 2.4]{Neukirch} again. Since $\mf D_{K/\QQ} = \mf D_{K/F} \mf D_{F/\QQ}$ (c.f.~\cite[Chapter III, Proposition 2.2]{Neukirch}), we obtain 
$$
\mf D_{K/\QQ} = \mf D_{K/F} \mf D_{F/\QQ} = 
\left( (\zeta_m - \zeta_m^{-1}) \cdot g'(\mu) \right)
$$
proving the claim. 
\end{example}

\begin{example}  \label{example:lambda}
Let $K$ be a CM field, 
and let $F \subset K$ be the maximal totally real subfield of $K$. 
Let $\lambda \in \OO_F$ such that $\tau(\lambda) > 0$ for some embedding $\tau \colon F \to \RR$ and $\psi(\lambda) < 0$ for all other $\psi \neq \tau \in \Hom(F,\RR)$ (such an element exists by Lemma \ref{rho} below). If we define $\Lambda \coloneqq (\OO_K^{n+1},h),$ where $h$ is the hermitian form given by the diagonal matrix $\tn{diag}(-\lambda, 1, \dotsc, 1)$, then $\Lambda$ satisfies Condition \ref{condition2.2} in Definition \ref{definition:introduction:admissible}. 
\end{example}

\begin{lemma} \label{rho}
Let $K$ be a CM field, 
and let $F \subset K$ be the maximal totally real subfield of $K$. 
There exists an element $\lambda \in \OO_F$ with $\tau(\lambda) > 0$ for some embedding $\tau \colon F \to \RR$ and $\psi(\lambda) < 0$ for all other $\psi \neq \tau \in \Hom(F,\RR)$. 
\end{lemma}
\begin{proof}
Let $\tilde \lambda \in \OO_F - \ZZ$. Note that we can write $\Hom(F,\RR) = \set{\tau, \psi_2, \dotsc, \psi_g}$ with $\tau(\tilde \lambda) < \psi_2(\tilde \lambda) < \cdots < \psi_g(\tilde \lambda)$. Let $r \in \QQ$ be a rational number such that $\tau(\tilde \lambda) < r < \psi_2(\tilde \lambda)$. Possibly after replacing $\tilde \lambda$ by $n \cdot \tilde \lambda$ for some $n \in \ZZ_{\geq 1}$, we may assume that $r \in \ZZ$. The element $\lambda =  r - \tilde \lambda \in \OO_F$ satisfies the requirements. 
\end{proof}

\section{Anti-unitary involutions} \label{section:anti} 

The goal of this section is to study anti-unitary involutions on $\Lambda$.  

\subsection{Anti-unitary involutions and copies of real hyperbolic space} 

Let $\mr L = (K, \Lambda)$ be an admissible hermitian lattice of rank $n+1 \geq 2$. We continue with the notation of Sections \ref{set-up-1} and \ref{set-up-2}. 

\begin{definition}
An $\ca O_F$-linear bijection $\alpha \colon \Lambda \xrightarrow{\sim} \Lambda$ is \textit{anti-unitary} if for all $x,y \in \Lambda$ and $b \in \OO_K$, one has 
$
\alpha(b \cdot x) = \sigma(b) \cdot \alpha(x)$ and $ h(\alpha(x), \alpha(y)) = \sigma(h(x,y)) \in \OO_K.
$ \end{definition}

Let $\mr A \subset \Aut(\Lambda)'$ be the set of anti-unitary involutions $\alpha \colon \Lambda \to \Lambda$. 
The group $\mu_K$ acts on $\mr A$ by multiplication; define 
\[
P\mr A \coloneqq \mu_K \setminus \mr A. 
\]

\begin{example} 
Consider Example \ref{example:lambda}. Thus, $K$ is a CM field, 
$F \subset K$ is the maximal totally real subfield of $K$, $\lambda \in \OO_F$ is an element such that $\tau(\lambda) > 0$ for some embedding $\tau \colon F \to \RR$ and $\psi(\lambda) < 0$ for all other $\psi \neq \tau \in \Hom(F,\RR)$ (cf.\ Lemma \ref{rho}), 
$\Lambda = (\OO_K^{n+1},h)$ where $h$ is the hermitian form given by $\tn{diag}(-\lambda, 1, \dotsc, 1)$. 
To obtain examples of anti-unitary involutions $\alpha \colon \Lambda \to \Lambda$, we proceed as follows. 
For each $\epsilon = (\epsilon_0, \dotsc, \epsilon_n) \in \set{\pm 1}^{n+1}$, define 
\[
\alpha_\epsilon \colon \Lambda \to \Lambda \quad \text{ as } \quad  \alpha_\epsilon \left( x_0, \dotsc, x_n \right)= \left( \epsilon_0 \bar x_0, \dotsc, \epsilon_n \bar x_n\right). 
\]
For $x,y\in \Lambda, b \in \OO_K$, one has $\alpha_\epsilon(b \cdot x) = \bar b \cdot \alpha_\epsilon(x)$ and 
$
h(\alpha_\epsilon(x), \alpha_\epsilon(y)) = 
\sigma(h(x,y) )
$ because $\sigma(\lambda) = \lambda$. Hence $\alpha_\epsilon$ is an anti-unitary involution for each $\epsilon \in \set{\pm 1}^{n+1}$. 
\end{example}

Let $\Aut(\Lambda)'$ be the group of unitary and anti-unitary $\OO_F$-linear bijections $\Lambda \xrightarrow{\sim} \Lambda$. 
Then 
\begin{align} \label{anti-iso-group}
\mu_K \subset \Aut(\Lambda) \subset \Aut(\Lambda)' \quad \tn{ and we define } \quad L' \coloneqq \Aut(\Lambda)'/\mu_K. 
\end{align}
Let $x \in K^\ast$. Observe that 
\begin{align}\label{crucialhypo}
\left(
x \in \OO_K^\ast \tn{ and }
x \cdot \sigma(x)  = 1\right) \quad \iff \quad \left(x \in \mu_K \right).
\end{align}
Indeed, for an embedding $\varphi \colon K \to \CC$, one has $
\va{\varphi(x)}^2 = \varphi(x \cdot \sigma(x))$;
if $x \in \OO_K^\ast$, then $\va{\varphi(x)} = 1$ for each $\varphi$ if and only if $x$ is a root of $1$ \cite[Corollary 5.6]{milneANT}. 
\begin{lemma} \label{Gammaembedding}
Let $\Isom(\CCH^n)$ be the group of isometries $f \colon \CC H^n \xrightarrow{\sim} \CC H^n$. The natural homomorphism $L' \to \Isom(\CCH^n)$ is injective.
\end{lemma}
\begin{proof}
This follows readily from (\ref{crucialhypo}). 
\end{proof}

Thus, any $\alpha \in P\mr A$ defines an anti-holomorphic involution 
\[
\alpha \colon \CC H^n \to \CC H^n
\]
that determines $\alpha \in P\mr A$ uniquely (whence the abuse of notation). 
We define 
\[ \RR H^n_{\alpha} \coloneqq (\CC H^n)^\alpha \subset \CC H^n.
\]
For each element $\alpha \in \mr A$, the hermitian form $h$ on $V$ induces a non-degenerate symmetric bilinear form
$$h|_{V^\alpha} \colon V^\alpha \times V^\alpha \longrightarrow \RR$$ 
on the real vector space 
\begin{align}\label{signature}
V^\alpha = \Lambda^\alpha\otimes_{\OO_F, \tau }\RR.
\end{align}
The form $h|_{V^\alpha} $ has signature $(n,1)$. The following lemma is readily proved:
 
\begin{lemma} \label{hyperbolic}
For $\alpha \in \mr A$, let $\PP(V^\alpha)$ be the space of real lines in $V^\alpha$, and let $\RR H(V^\alpha) \subset \PP(V^\alpha)$ be the subspace of lines in $V^\alpha$ negative for $h|_{V^\alpha}$. The canonical isomorphism $\PP(V^\alpha) \cong \PP(V)^\alpha$ restricts to an isomorphism $\RRH(V^\alpha) \cong \RRH^n_\alpha$. \hfill \qed
\end{lemma}

We conclude that for each $\alpha \in P\mr A$, the space $$\RR H^n_\alpha = (\CC H^n)^\alpha \subset \CC H^n$$ is isometric to the real hyperbolic space of dimension $n$. 

\subsection{Anti-unitary involutions and arithmetic lattices in $\PO(n,1)$}

Continue with the above notation. Define \[L_\alpha \coloneqq \textnormal{Stab}_{L}(\RR H^n_\alpha) \subset L\quad \quad \tn{(the stabilizer of $\RR H^n_\alpha$ in $L$).}\] 

\begin{lemma} \label{normalisator}
Let $\alpha \in P\mr A$. Then 
$
L_\alpha = N_{L}(\alpha) \coloneqq \set{\gamma \in L \mid \gamma \circ \alpha = \alpha \circ \gamma \in L'} \subset L.
$
\end{lemma}
\begin{proof}
The inclusion $N_L(\alpha) \subset L_\alpha$ is clear. For the other inclusion, let $g \in \Aut(\Lambda)$ with associated element $[g] \in \Aut(\Lambda)/\mu_K = L$. Assume that $[g] \in L_\alpha$ for some element $\alpha \in \mr A$ with class $[\alpha] \in P\mr A$. 
Notice that $V = V^{\alpha} \otimes_{\RR} \CC$. Hence each $z \in V$ can be written as $z = x + i \cdot y$ where $x,y \in V^\alpha$. Since $[g] \in L_\alpha$, we get $[g x] \in \RR H^n_\alpha$ for each $[x] \in \RR H^n_\alpha$. In particular, there exists an element $\xi \in \mu_K$ such that 
$
\alpha(gx) = \xi \cdot (gx) \in V
$
for each $x \in V^\alpha$. 
Consequently, for $[z] = [x+i\cdot y] \in \CC H^n$ with $x,y \in V^\alpha$, we have 
\[
[g(\alpha z)] = [g x-i \cdot g y] = [\xi \cdot gx - i \cdot \xi\cdot gy] = [\alpha(gx) - i \cdot \alpha(gy)]  = [\alpha(gz)]. 
\]
This shows that $[g] \circ [\alpha] = [\alpha] \circ [g]$ as elements of $\Isom(\CC H^n)$. Thus, in view of Lemma \ref{Gammaembedding}, we get $[g] \circ [\alpha] = [\alpha] \circ [g]$ as elements of $L'$, and we are done. 
\end{proof}

Let $\alpha \in \mr A$, consider the lattice $\Lambda^\alpha \subset V^\alpha$ (see \eqref{signature}) and the space of negative lines $\RR H(V^\alpha) \subset \PP(V^\alpha)$. We have $\tn{Isom}(\RR H(V^\alpha)) \cong \tn{Isom}(\RR H^n) = \tn{PO}(n,1)$ by Lemma \ref{Gammaembedding}. Let $\tn{O}(\Lambda^\alpha)(\OO_F)$ be the group of $\OO_F$-linear isometries $\gamma \colon \Lambda^\alpha \xrightarrow{\sim} \Lambda^\alpha$, and define $\tn{PO}(\Lambda^\alpha)(\OO_F) = \tn{O}(\Lambda^\alpha)(\OO_F)/\set{\pm 1}$. There are canonical embeddings 
\begin{align}\label{canonical-embeddings}
L_\alpha \longhookrightarrow \tn{Isom}(\RR H(V^\alpha)) = \tn{PO}(V^\alpha) \quad \tn{ and } \quad 
\tn{PO}(\Lambda^\alpha)(\OO_F)  \longhookrightarrow  \tn{PO}(V^\alpha). 
\end{align}

Explicitly, the embedding $L_\alpha \hookrightarrow \PO(V^\alpha)$ in \eqref{canonical-embeddings} is defined as follows. We consider $K$ as a subfield of $\CC$ via the embedding $\tau$. This identifies $\zeta \in K$ with a primitive $m$-th root of unity in $\CC$. Let $\zeta_{2m} = e^{2 \pi i / 2m} \in \CC$. Let $[g] \in L_\alpha$ with representative $g \in \Aut(\Lambda)$. Then $g \alpha g^{-1} = \zeta^i \cdot \alpha$ for some $i \in \ZZ_{\geq 0}$. This implies that $g(V^\alpha) = V^{\zeta^i  \alpha}$. Let $\xi \in \langle \zeta_{2m} \rangle \subset \CC^\ast$ such that $\xi^2 = \zeta^{-i}$. Then multiplication by $\xi$ defines an isomorphism of complex vector spaces $\xi \colon V \xrightarrow{\sim} V$ that identifies $V^{\zeta^i  \alpha}$ with $V^\alpha$. Sending $[g] \in L_\alpha$ to the class $[\xi \circ g] \in \PO(V^\alpha)$ of the composition $\xi \circ g$, which is an isometry $V^\alpha \xrightarrow{\sim} V^{\zeta^i  \alpha} \xrightarrow{\sim} V^\alpha$, gives a well-defined embedding $L_\alpha \hookrightarrow \PO(V^\alpha)$ which coincides with the first map appearing in equation \eqref{canonical-embeddings}. 

\begin{theorem} \label{th:crucialthm-finitevolume}
Let $\alpha \in \mr A$. The subgroups $L_\alpha \subset \PO(V^\alpha)$ and $\tn{PO}(\Lambda^\alpha)(\OO_F) \subset \PO(V^\alpha)$ are commensurable. In particular, $L_\alpha \subset \PO(V^\alpha)$ is an arithmetic lattice. 
\end{theorem}
%
\begin{proof}
We follow the ideas of the proof of \cite[Theorem 4.10]{heckman2016hyperbolic}. Let $g \in \Aut(\Lambda)$ such that $[g] \in L_\alpha$. Recall that $\zeta \in K$ is a primitive $m$-th root of unity in $K$ that generates the torsion subgroup $\mu_K \subset \OO_K^\ast$, see Section \ref{set-up-2}. 
Define 
\[
L_\alpha^I \coloneqq \set{[g] \in L_\alpha \mid h  \alpha h^{-1} = \zeta^i \cdot  \alpha \tn{ with $i \in 2\ZZ$, for some (equivalently, any) } h \in [g]}.
\]
Note that $L_\alpha^I \subset L_\alpha$ is a normal subgroup of $L_\alpha$. Also note that for each $[g] \in L_\alpha^I$ there is an element $h \in [g]$, unique up to sign, such that $h \alpha h^{-1} = \alpha$. 
Consequently, there is a natural embedding 
\begin{align}\label{naturalembedding}
f \colon L_\alpha^I \longhookrightarrow \tn{PO}(\Lambda^\alpha)(\OO_F).
\end{align}
In fact, one has $
L_\alpha \cap \tn{PO}(\Lambda^\alpha)(\OO_F) = L_\alpha^I \subset \tn{PO}(V^\alpha),
$
and this intersection has index at most two in $L_\alpha$. Indeed, either $L_\alpha^I = L_\alpha$ or there exists an element $[g] \in L_\alpha$ such that for some $h \in [g]$ one has $h \alpha h^{-1} = \zeta^i \cdot \alpha$ for some odd $i \in \ZZ$, and for any two such elements $[g_i] \in L_\alpha$ ($i = 1,2$), one has $[g_1 g_2] \in L_\alpha^I \subset L_\alpha$. 

It remains to show that $L_\alpha^I$ has finite index in $\tn{PO}(\Lambda^\alpha)(\OO_F)$. To prove this, let $i \colon \Lambda^\alpha \hookrightarrow \Lambda$ denote the inclusion. The induced map $i \colon \Lambda^\alpha \otimes_{\OO_F}\OO_K \hookrightarrow \Lambda$ has finite cokernel; let $d$ be its order. There is a map $\varphi \colon \Lambda \to \Lambda^{\alpha} \otimes_{\OO_F} \OO_K$ 
such that $i \circ \varphi = d$ and $\varphi \circ i = d$, 
and such that the induced map $\varphi_K \colon \Lambda \otimes_{\OO_K} K \to \Lambda^\alpha \otimes_{\OO_F}K$ is $d$ times the inverse $i^{-1}$ of $i \colon \Lambda^\alpha \otimes_{\OO_F}K \xrightarrow{\sim} \Lambda \otimes_{\OO_K}K$. For an element $M \in \rm{O}(\Lambda^\alpha)(\OO_F)$, let $M_{\OO_K}$ be the induced map $\Lambda^\alpha \otimes_{\OO_F}\OO_K \xrightarrow{\sim} \Lambda^\alpha \otimes_{\OO_F}\OO_K$. Define
\begin{align*}
\tn{PO}(\Lambda^\alpha)(\OO_F)^I  \coloneqq  \set{[M] \in \tn{PO}(\Lambda^\alpha)(\OO_F) \mid i \circ M_{\OO_K} \circ \varphi = d \cdot \overline M \tn{ for } \overline M \in \Aut(\Lambda)}. 
\end{align*}
We claim that the embedding $f \colon L_\alpha^I \hookrightarrow \tn{PO}(\Lambda^\alpha)(\OO_F)$ defined in \eqref{naturalembedding} factors as 
\begin{align} \label{align:first-inclusion-equality}
L_\alpha^I \longhookrightarrow \tn{PO}(\Lambda^\alpha)(\OO_F)^I  \longhookrightarrow \tn{PO}(\Lambda^\alpha)(\OO_F). 
\end{align}
To prove this, let $[g] \in L_\alpha^I$ with $g \in [g]$ such that $g \alpha = \alpha g$. Let $M_g = g|_{\Lambda^\alpha}$ be the restriction $\Lambda^\alpha \to \Lambda^\alpha$ of $g$ to $\Lambda^\alpha$. Then $f([g]) = [M_g] \in \tn{PO}(\Lambda^\alpha)(\OO_F)$. As we have $
i \circ (M_g)_{\OO_K} \circ \varphi = g \circ i \circ \varphi = d \cdot g$ as maps $ \Lambda \to \Lambda$, we get $[M_g] \in \tn{PO}(\Lambda^\alpha)(\OO_F)^I$. This proves the factorization \eqref{align:first-inclusion-equality}. 

Next, we claim that the first inclusion in \eqref{align:first-inclusion-equality} is an equality. To see this, let $[M] \in \tn{PO}(\Lambda^\alpha)(\OO_F)^I$. Let $\overline M \in \Aut(\Lambda)$ such that $\overline M(\Lambda^\alpha) = \Lambda^\alpha$ and $\overline M|_{\Lambda^\alpha} = M$. 
It suffices to prove that the class $[\overline M] \in L$ is contained in $L_\alpha^I \subset L_\alpha \subset L$. Since $[\overline M] \in L$ stabilizes $\RR H^n_\alpha$, we have $[\overline M] \in L_\alpha = N_L(\alpha)$ by Lemma \ref{normalisator}. Consequently, $\overline M \alpha (\overline M)^{-1} = \zeta^i \cdot \alpha$ for some $i \in \ZZ_{\geq 0}$. 
As $\overline M(\Lambda^\alpha) = \Lambda^\alpha$, we get $i \equiv 0 \bmod m$. Therefore, $[\overline M] \in L_\alpha^I$, proving the claim. We conclude that 
\begin{align} \label{last}
L_\alpha^I = \tn{PO}(\Lambda^\alpha)(\OO_F)^I  \subset \tn{PO}(\Lambda^\alpha)(\OO_F). 
\end{align}
To finish the proof, it remains to show $ \tn{PO}(\Lambda^\alpha)(\OO_F)^I \subset  \tn{PO}(\Lambda^\alpha)(\OO_F)$ is a subgroup of finite index. Consider the submodule $\varphi(\Lambda) \subset \Lambda^{\alpha} \otimes_{\OO_F}\OO_K$, and note that
\begin{align} \label{align:other-description-PO}
\tn{PO}(\Lambda^\alpha)(\OO_F)^I  =  \set{[M] \in \tn{PO}(\Lambda^\alpha)(\OO_F) \mid M_{\OO_K}(\varphi(\Lambda)) = \varphi(\Lambda) }. 
\end{align}
Let 
\[
H(\Lambda^\alpha) \coloneqq \Ker\left( \tn{PO}(\Lambda^\alpha)(\OO_F) \to \PGL(\Lambda^\alpha  / d\Lambda^\alpha) \right). 
\]
Thus, $H(\Lambda^\alpha)  \subset \tn{PO}(\Lambda^\alpha)(\OO_F)$ is the subgroup of $[M] \in \tn{PO}(\Lambda^\alpha)(\OO_F)$ such that there exists $\varepsilon \in \set{\pm 1}$, such that for each $x \in \Lambda^\alpha$, we have $M\cdot x = \varepsilon \cdot x + dy$ for some $y \in \Lambda^\alpha$. 
Thus, if $x \in \varphi(\Lambda)$, we get that $M_{\OO_K}\cdot x = \varepsilon \cdot x + dy  \in \varphi(\Lambda)$. Consequently, $M_{\OO_K}(\varphi(\Lambda)) = \varphi(\Lambda)$, so that $[M] \in \tn{PO}(\Lambda^\alpha)(\OO_F)^I  $ by \eqref{align:other-description-PO}. We conclude that 
\[
H(\Lambda^\alpha)  \subset \tn{PO}(\Lambda^\alpha)(\OO_F)^I \subset  \tn{PO}(\Lambda^\alpha)(\OO_F). 
\]
This implies that $ \tn{PO}(\Lambda^\alpha)(\OO_F)^I \subset  \tn{PO}(\Lambda^\alpha)(\OO_F)$ is a subgroup of finite index. 
\end{proof}



We will also need the following finiteness result. For a finite group $G$ and a $G$-group $A$, let $\rm H^1(G, A)$ denote the first cohomology group of $G$ with coefficients in $A$ in the sense of nonabelian group cohomology. See \cite[Chapitre I, \S 5.1]{serre-galoisienne}. 

\begin{proposition} \label{proposition:galois} Suppose that $\mr A \neq \emptyset$. Let $\alpha_0  \in \mr A$, and define $G = \Gal(K/F)$. 
\begin{enumerate}
\item \label{itemones}
The group $G$ acts on $\Aut(\Lambda)$, and there is a bijection $\Aut(\Lambda) \backslash \mr A \cong \rm H^1(G, \Aut(\Lambda))$. 
\item 
The sets $\Aut(\Lambda) \backslash \mr A$ and $C\mr A = L \backslash P\mr A$ are finite. 
\end{enumerate}
\end{proposition}

\begin{proof}
1. We can let $G$ act on $\Aut(\Lambda)$ via the involution $ \sigma \colon \gamma \mapsto \alpha_0 \circ \gamma \circ \alpha_0$. 
Consider 
\begin{align*}
\rm H^1(G, \Aut(\Lambda)) &= \rm Z^1(G, \Aut(\Lambda))/_{\sim} = \set{\tn{$1$-cocycles } G \to \Aut(\Lambda), f \colon s \mapsto a_s}/_{\sim}
\end{align*}
where $f \sim f'$ if there exists $b \in \Aut(\Lambda)$ such that $f'(s) = b^{-1}f(s) ^sb \in \Aut(\Lambda)$ for all $s \in G$. See \cite[Chapitre I, \S 5.1]{serre-galoisienne}. If $g$ denotes the generator of $G$, the map 
$
f \mapsto f(g)
$ defines a bijection between the set of $1$-cocycles $\rm Z^1(G, \Aut(\Lambda))$ and the set of elements $\gamma \in \Aut(\Lambda)$ that satisfy $\gamma \cdot \sigma(\gamma) = \id$. This yields a bijection
\[
\mr A \xrightarrow{\sim} \rm Z^1(G, \Aut(\Lambda)), \quad \alpha \mapsto \alpha \circ \alpha_0,
\]
with inverse given by $\gamma \mapsto \gamma \circ \alpha_0$. Let $\alpha , \beta \in \mr A$ and $g \in \Aut(\Lambda)$. Then $g^{-1} \beta g = \alpha$ if and only if $g^{-1} \left(\beta \circ \alpha_0\right) \sigma(g) = \alpha \circ \alpha_0$, hence item \ref{itemones} is proved. 

2. Since the canonical map $\Aut(\Lambda) \backslash \mr A \to L \setminus P\mr A$ is surjective, it suffices to prove that $\Aut(\Lambda) \backslash \mr A$ is finite. By item \ref{itemones}, we need to show that $\rm H^1(G, \Aut(\Lambda))$ is finite. Note that $\Aut(\Lambda) = \Aut_{\OO_K}(\Lambda) = U(\Lambda)(\OO_K)$ arises as the group of $\OO_K$-rational points of a linear algebraic group $U(\Lambda)_K \subset \GL(\Lambda \otimes_{\OO_K} K)$ over $K$. Moreover, the involution $\sigma \colon \Aut(\Lambda) \to \Aut(\Lambda)$ defined in item \ref{itemones} extends to an involution $\sigma \colon U(\Lambda)_K \to U(\Lambda)_K$, which implies that $\Aut(\Lambda)$ is an \emph{arithmetic $G$-group} in the sense of \cite[Section 3.4]{borel-serre}. In particular, the set $\rm H^1(G, \Aut(\Lambda))$ is finite by \cite[Proposition 3.8]{borel-serre}. The proof is finished. 
\end{proof}

\subsection{Anti-unitary involutions and the hyperplane arrangement}

Continue with the above notation. 

\begin{lemma}\label{anti-involution}
Let $\alpha \in \mr A$ and $r \in \mr R$. Then $\alpha \circ \phi_r^i = \phi_{\alpha(r)}^{-i} \circ \alpha$. 
\end{lemma}
\begin{proof}
Indeed, 
$
\alpha(\phi_r^i(x)) = \alpha( x - (1-\zeta^i)h(x,r)\cdot r ) = 
\phi_{\alpha(r)}^{-i}(\alpha(x))$ for $x \in \Lambda$. \end{proof}

\begin{lemma}\label{alphaswitch}
Let $x \in \RRH^n_\alpha$ and write $\ca H(x) = \set{H_{r_1}, \dotsc, H_{r_k}}$, $r_i \in \mr R$. 
\begin{enumerate}
\item
Let $i \in \{1, \dotsc, k\}$. Then $
\alpha(H_{r_i}) = H_{\alpha(r_i)} =  H_{r_j}$ for some  $j \in \{1, \dotsc, k\}$. 
\item Let $i,j \in \set{1, \dotsc, k}$ such that $H_{\alpha(r_i)} = H_{r_j}$. Then is some element $\lambda_i \in K^\ast$ with $\va{\lambda_i} = 1$, such that $\alpha(r_i) = \lambda_i \cdot r_j$. 
\end{enumerate}
\end{lemma}
\begin{proof}
1. For any $\beta \in \mr A$ and $r \in \mr R$, we have
$
\beta(H_r) = H_{\beta(r)}. 
$
Since $x \in H_{r_i}$, we have $x = \alpha(x) \in \alpha(H_{r_i}) = H_{\alpha(r_i)}$ for every $i$. In particular, we have $H_{\alpha(r_i)} \in \ca H(x)$ (see Definition \ref{fullset}), so that $H_{\alpha(r_i)} = H_{r_j}$ for some $j \in \set{1, \dotsc, k}$. 

2. This follows from Proposition \ref{prop:remember}. 
\end{proof}

\begin{definition} \label{alphainvolution}
Let $\alpha \in P\mr A$ and $x \in \RRH^n_\alpha$. Write $\ca H(x) = \set{H_{r_1}, \dotsc, H_{r_k}}$, see Definition \ref{fullset}. By Lemma \ref{alphaswitch}, the involution $\alpha$ induces an involution on the set $\ca H(x)$. Define $\alpha \colon I \to I$ as the resulting involution on the set $I = \set{1, \dotsc, k}$. 
\end{definition}

\begin{proposition}\label{preliminaryproposition}
Let $\alpha \in P\mr A$ and $x \in \RRH^n_\alpha$. Write $\ca H(x) = \set{H_{r_1}, \dotsc, H_{r_k}}$, and let $g = \phi_{r_1}^{i_1} \circ \cdots \circ \phi_{r_k}^{i_k} \in G(x)$ for some $i_j \in \ZZ/m$. The following are equivalent:
\begin{enumerate}
\item We have $g \circ \alpha \in P\mr A$. Equivalently: $g\circ \alpha \in L'$ is an involution.
\item For each $j \in \set{1, \dotsc, k}$, we have $i_{j} \equiv i_{\alpha(j)} \bmod m$. 
\end{enumerate}
\end{proposition}
\begin{proof}
By Lemma \ref{alphaswitch} and Definition \ref{alphainvolution}, for each $i \in \set{1, \dotsc, k}$, there exists $\xi_i \in K^\ast$ with $\va{\lambda_i} = 1$ such that $\alpha(r_i) = \lambda_i \cdot r_{\alpha(i)}$. In particular, $\phi_{\alpha(r_i)} = \phi_{r_{\alpha(i)}}$, see Proposition \ref{prop:remember}. Therefore, we have:
\begin{align*}
g \circ \alpha \circ g \circ \alpha &=  \phi_{r_1}^{i_1} \circ \cdots \circ \phi_{r_k}^{i_k} \circ \alpha \circ  \phi_{r_1}^{i_1} \circ \cdots \circ \phi_{r_k}^{i_k} \circ \alpha  \\
 & =  \phi_{r_1}^{i_1} \circ \cdots \circ \phi_{r_k}^{i_k} \circ 
 \phi_{r_1}^{-i_{\alpha(1)}} \circ \cdots \circ \phi_{r_k}^{-i_{\alpha(k)}}  \\
 & = \phi_{r_1}^{i_1-i_{\alpha(1)}} \circ \cdots \circ \phi_{r_k}^{i_k - i_{\alpha(k)}} \in G(x). 
 \end{align*}
 The proposition follows. 
 \end{proof}

\section{Attaching the hyperbolic pieces} \label{gluing} In this section, we construct a topological space by gluing together arithmetic hyperbolic pieces. Afterwards, in Section \ref{section-section-section}, we provide this topological space first with a path metric and then with a complete hyperbolic orbifold structure.

\subsection{Definition of the attachment}

Let $\mr L = (K, \Lambda)$ be an admissible hermitian lattice of rank $n+1$, for some $n \geq 1$. Continue with the notation of Sections \ref{sec:preliminaries} and \ref{section:anti}. Define 
\[
\widetilde Y \coloneqq \coprod_{\alpha \in P\mr A} \RR H^n_\alpha. 
\]
We will glue the different copies of hyperbolic space $\RR H^n_\alpha$ by defining an equivalence relation $\sim$ on $\widetilde Y$. 


\begin{definition} \label{def:conditions}
Define a relation $\sim$ on $\wt Y$ as follows. Let
 \[
(x, \alpha), (y,\beta) \in \wt Y
 \]
 with $\alpha, \beta \in P\mr A$,  $x \in \RR H^n_\alpha$ and $y \in \RR H^n_\beta$. Then $(x, \alpha) \sim (y,\beta)$ if the following conditions are satisfied:
\begin{enumerate} 
\item \label{item1}  We have $x = y \in \CC H^n$. 
\item \label{item2} If $\alpha \neq \beta$, then $x = y \in \CC H^n$ lies in $ \mr H \subset \CC H^n$, and $\beta \circ \alpha \in G(x) = G(y) \subset L$. 
\end{enumerate}
\end{definition}

\begin{remark} \label{rem:gluing}
We will see below (see Lemma \ref{intersection}) that the gluing rules can be rephrased as follows: we glue $\RR H^n_\alpha$ and $\RR H^n_\beta$ along their intersection, provided that 
for some (equivalently, any) $x \in \RR H^n_\alpha \cap \RR H^n_\beta$, the real structures $\alpha$ and $\beta$ differ by a product of reflections in hyperplanes $H_r \subset \mr H$ that pass through $x$. 
\end{remark}

\begin{lemma} \label{eqrel}
The relation $\sim$ on $\wt Y$ defined above is an equivalence relation. 
\end{lemma}

\begin{proof}
Let $(x,\alpha), (y,\beta), (z, \gamma) \in \wt Y$. Clearly $(x,\alpha) \sim (x,\alpha)$. If $(x,\alpha) \sim (y,\beta)$ and $(x,\alpha) \neq (y,\beta)$, then $ x = y \in \mr H$ and $\beta \circ \alpha \in G(x) = G(y)$. Thus, since $G(x)$ is a group, we get $\alpha \circ \beta = (\beta \circ \alpha)^{-1} \in G(x) = G(y)$, hence $(y,\beta) \sim (x,\alpha)$. Finally, assume $(x,\alpha) \sim (y,\beta)$ and $(y,\beta) \sim (z, \gamma)$. We claim that $(x,\alpha) \sim (z,\gamma)$. We may assume that $(x,\alpha), (y,\beta), (z, \gamma) \in \wt Y$ are pairwise distinct. This implies that $x = y = z$ and that $\beta \circ \alpha \in G(x) = G(y)$ and $\gamma \circ \beta \in G(y) = G(z)$. Thus, we get $\gamma \circ \beta \circ \beta \circ \alpha = \gamma \circ \alpha \in G(x) = G(z)$, proving the claim and hence the lemma.
\end{proof}

\begin{lemma} \label{lemma:pgammaaction}
The action of $L$ on $\CC H^n$ induces an action of $L$ on $\widetilde Y$ which is compatible with the equivalence relation $R$. Therefore, $L$ acts naturally on $Y$. Moreover, $$L \setminus \widetilde Y =  \coprod_{\alpha \in C\mr A} L_\alpha \setminus \RR H^n_\alpha$$
for any set $C \mr A \subset P\mr A$ of representatives for the quotient $L \setminus P\mr A$. 
\end{lemma}

\begin{proof}
If $\phi \in L$, then $
\phi \left( \RR H^n_\alpha \right) = \RR H^n_{\phi \alpha \phi^{-1}}$ hence $L$ acts on $\widetilde Y = \coprod_{\alpha \in P\mr A} \RR H^n_\alpha$, and
$$
L \setminus \widetilde Y = L \setminus \coprod_{\alpha \in P\mr A} \RR H^n_\alpha = \coprod_{\alpha \in C\mr A} L_\alpha \setminus \RR H^n_\alpha. 
$$
Now suppose that $(x,\alpha) \sim (y,\beta) \in \widetilde Y$ and let $\phi \in L$. Then $\phi\cdot (x,\alpha) = (\phi x, \phi \alpha \phi^{-1}) \in \RR H^n_{\phi \alpha \phi^{-1}}$ and similarly, $\phi\cdot (y,\beta ) \in \RRH^n_{\phi\beta \phi^{-1}}$. We claim that 
$
(\phi x, \phi \alpha \phi^{-1}) \sim (\phi y , \phi \beta \phi^{-1}). 
$
For this, we may and do assume that $(x,\alpha)\neq (y,\beta)$, hence $x = y \in \mr H$ and $\beta = g \circ \alpha$ for some $g \in G(x)$ as in Definition \ref{def:conditions}. We get
\[
\phi \beta \phi^{-1} \circ \phi \alpha \phi^{-1} = \phi \circ (\beta \circ \alpha) \circ \phi^{-1} \in \phi G(x)\phi^{-1} = G(\phi (x)). 
\]
This proves the lemma. 
\end{proof}

\begin{definition} \label{gluedspace}
Let $\mr L = (K, \Lambda)$ be an admissible hermitian lattice of rank $n+1 \geq 2$. 
Define $$Y(\mr L) = \wt Y/ \sim$$ as the quotient of $\widetilde Y$ by the equivalence relation $\sim$ introduced in Definition \ref{def:conditions}, and equip $Y(\mr L)$ with the quotient topology. Consider the natural action of the group $L$ on $Y(\mr L)$, cf.\ Lemma \ref{lemma:pgammaaction}. Define$$
M(\mr L) \coloneqq L \setminus Y(\mr L).$$ Let $C\mr A \subset P\mr A$ be a set of representatives for the action of $L$ on $P\mr A$ by conjugation. We call $M(\mr L)$ the \textit{attachment of the spaces $L_\alpha \sm \RR H^n_\alpha$ for $\alpha \in C\mr A$}. 
\end{definition}

\subsection{Basic properties of the attachment}

\begin{lemma} \label{intersection} The following assertions are true.
\begin{enumerate}
\item \label{dos}
Let $\alpha, \beta \in P\mr A$ and $x \in \RR H^n_\alpha \cap \RR H^n_\beta$ such that $(x, \alpha) \sim (x,\beta)$. Then $(y,\alpha) \sim (y,\beta)$ for every $y \in \RR H^n_\alpha \cap \RR H^n_\beta$. 
\item \label{tres} The natural map $\widetilde Y \to \CC H^n$ descends to a $L$-equivariant map $Y(\mr L) \to \CC H^n$. 
\end{enumerate}
\end{lemma}

\begin{proof}
Let us prove item \ref{dos}. Since $(x,\alpha)  \sim (x,\beta)$, there exists $g \in G(x)$ such that $\beta = g \circ \alpha$. Write $\ca H(x) = \set{H_{r_1}, \dotsc, H_{r_k}}$, so that $$G(x) = \langle \phi_{r_1}, \dotsc, \phi_{r_k} \rangle, \quad \quad g = \phi_{r_1}^{n_1} \circ \cdots \circ \phi_{r_k}^{n_k}, \quad \quad n_1, \dotsc, n_k \in \ZZ/m.$$
Let $I \subset \set{1, \dotsc, k}$ be the subset of $i \in \set{1, \dotsc, k}$ such that $n_i \not \equiv 0 \bmod m$. Then, since $\beta = g \circ \alpha$, we have
\begin{align} \label{align:intersection-betagamma}
\RR H^n_\alpha \cap \RR H^n_\beta = \RR H^n_\alpha  \cap (\CC H^n)^g = \RR H^n_\alpha \cap \bigcap_{i \in I} H_{r_i}
\end{align}
where the second equality holds by Lemma \ref{intersection:complexpart}. In particular, item \ref{dos} follows. 
Item \ref{tres} is clear: if $(x, \alpha) \sim (y,\beta) \in \wt Y$, then $x = y \in \CC H^n$. 
\end{proof}


\begin{lemma} \label{lemma:openembeddingtopological}
In the above notation, the following assertions are true. 
\begin{enumerate}
\item \label{item:above-not-1}For each $\alpha \in P\mr A$, the image of the natural continuous map 
\[
L_\alpha \setminus \RR H^n_\alpha \longrightarrow M(\mr L)
\]
is closed in $M(\mr L)$. 
\item \label{item:openembedding}
Let $C\mr A \subset P\mr A$ be a set of representatives for the action of $L$ on $P\mr A$ by conjugation.
The composition
\[
\coprod_{\alpha \in C\mr A} L_\alpha \setminus \left(\RR H^n_\alpha - \mr H\right) \xlongrightarrow{\sim}L \setminus \left( \coprod_{\alpha \in P\mr A} \RR H^n_\alpha - \mr H \right) \longrightarrow L  \setminus Y(\mr L) = M(\mr L)
\]
is an open embedding of topological spaces. 
\end{enumerate}
\end{lemma}
\begin{proof}
1. Let $\varphi \colon L_\alpha \setminus \RR H^n_\alpha \to M(\mr L)$ denote the natural map. 
Let $p \colon \wt Y \to Y(\mr L)$ and $q \colon Y(\mr L) \to L\setminus Y(\mr L) = M(\mr L)$ be the respective quotient maps; it suffices to show that $p^{-1}(q^{-1}(\varphi(L_\alpha \sm \RR H^n_\alpha)))$ is closed in $\wt Y$, and hence that \begin{align} \label{align:int-nonempty}p^{-1}(q^{-1}(\varphi(L_\alpha \sm \RR H^n_\alpha))) \cap \RR H^n_\beta\end{align} is closed in $\RR H^n_\beta$ for each $\beta \in P\mr A$. We first consider the case $\beta = \alpha$. 
Note that $$p^{-1}(q^{-1}(\varphi(L_\alpha \sm \RR H^n_\alpha))) \cap \RR H^n_\alpha = \RR H^n_\alpha$$
which is indeed closed in $\RR H^n_\alpha$. For $\beta \neq \alpha \in P\mr A$, assume that \eqref{align:int-nonempty} is non-empty, and let $z$ be a point in the intersection. Then there exists $g \in L$ such that for $y \coloneqq g \cdot x$ and $\gamma \coloneqq g \alpha g^{-1}$, we have $(y, \gamma) \sim (z, \beta)$. In particular, $y = z \in \RR H^n_\gamma \cap \RR H^n_\beta$, and by Lemma \ref{intersection} and its proof (see equation \eqref{align:intersection-betagamma}), we have 
$$p^{-1}(q^{-1}(\varphi(L_\alpha \sm \RR H^n_\alpha))) \cap \RR H^n_\beta = L_\beta \cdot \left(\RR H^n_\gamma \cap \RR H^n_\beta \right) = L_\beta \cdot  \left(\RR H^n_\beta \cap H_{r_1} \cap \cdots \cap H_{r_p} \right)
$$
for suitable $r_1, \dotsc, r_p \in \mr R$. Moreover, we have
\begin{align} \label{align:locallyfiniteRHbeta}
L_\beta \cdot  \left(\RR H^n_\beta \cap \bigcap_{i =1}^p H_{r_i} \right) = \bigcup_{g \in L_\beta}\left( \RR H^n_\beta \cap  H_{g r_1} \cap \cdots \cap H_{g r_p } \right).
\end{align}
Since the set $\set{H_r}_{r \in \mr R}$ of subsets of $\CC H^n$ is locally finite (see Lemma \ref{lemma:beauville:finite}), the set $$\set{\RR H^n_\beta \cap  H_{g r_1 } \cap \cdots \cap H_{g r_p }}_{g \in L_\beta}$$ of subsets of $\RR H^n_\beta$ is locally finite. In particular, the set on the right hand side of the equality in \eqref{align:locallyfiniteRHbeta} is closed in $\RR H^n_\beta$, and we are done. 

2. Let $\mr H_\RR \coloneqq \sqcup_{\alpha \in P\mr A} (\RR H^n_\alpha \cap \mr H) \subset \wt Y$ and let $\wt Y_0 \coloneqq \wt Y - \mr H_\RR$. 
 Let $p \colon \wt Y \to Y(\mr L)$ be the quotient map. We claim that the restriction $\psi \coloneqq p|_{\wt Y_0} $, which is a map 
\[
\psi \colon \wt Y_0 \longrightarrow Y(\mr L),
\]
is an open embedding. Note that $\psi$ is injective, because if $\psi(x,\alpha) = \psi(y,\beta)$ then $(x, \alpha) \sim (y,\beta)$ which implies that $(x,\alpha) = (y,\beta)$ as $x, y \not \in \mr H$. Moreover, for any open subset $U \subset \wt Y_0$, the subset $\psi(U)$ is open in $Y(\mr L)$ because $p^{-1}(\psi(U)) = U \subset \wt Y_0 \subset \wt Y$ and $ \wt Y_0$ is open in $\wt Y$ by Lemma \ref{lemma:beauville:finite}. We are done. 
\end{proof}

%

\section{Hyperbolic orbifold structure of the attachment} \label{section-section-section}


Section \ref{section-section-section} is devoted to the proof of Theorem \ref{theorem:introduction:uniformization}. 
The key is the following result. 

\begin{proposition} \label{glueingtheorem1} Let $\mr L = (K, \Lambda)$ be an admissible hermitian lattice of rank $n+1 \geq 2$. With the notation of Section \ref{gluing}, the following assertions are true. 
\begin{enumerate}%
\item \label{completepart}
Each path-connected component $Y(\mr L)_i$ of the space $Y(\mr L)$ admits a path metric such that the natural map $Y(\mr L)_i \to \CC H^n$ is a local isometric embedding. 
\item  \label{completepart:2} The above path metrics on the path-connected components of $Y(\mr L)$ induce complete path metrics on the path-connected components of $M(\mr L) = L \setminus Y(\mr L)$. 
\item \label{item:importanttheoremitem}
Each $x \in M(\mr L)$ admits an open neighbourhood $U \subset M(\mr L)$ which is isometric to the quotient of a connected open $W \subset \RR H^n$ by a finite group of isometries.
\end{enumerate}
\end{proposition}


\subsection{Preliminaries on path metrics} \label{section:prelim-path}

In this section, we provide some preliminary results regarding pull-back metrics and quotient metrics. These results are certainly well-known; we state and prove them for lack of suitable reference. 

In this paper, a continuous map of topological spaces $f \colon X \to Y$ is called \emph{proper} if the inverse image $f^{-1}(K)$ of any compact subset $K \subset Y$ is compact. 
\begin{lemma}\label{lemma:lee}
Let $X$ and $Y$ be topological spaces such that $Y$ is locally compact and Hausdorff. Let $f \colon X \to Y$ be a proper continuous map. Then $f$ is closed.  
\end{lemma}
\begin{proof}
This is well-known (see e.g.\ \cite[Theorem A.57]{MR2954043}).  
\end{proof}

\begin{lemma} \label{lemma:pathmetrics-pullback}
Let $f \colon X \to Y$ a local embedding of path-connected topological spaces. Assume $Y$ is given a path metric $d_Y$, and $X$ is Hausdorff. For a path $\gamma \colon [a,b] \to X$, define its length $\ell(\gamma)$ as $\ell(\gamma) = \ell(f \circ \gamma)$, where the latter denotes the length of the path $f_\ast(\gamma) \colon [a,b] \to Y$, $f_\ast(\gamma) = f \circ \gamma$. 
For $x,y \in X$, let $d_X(x,y)$ be the infimum of the lengths of paths which join $x$ and $y$. Then the following assertions are true.
\begin{enumerate}
\item $d_X$ defines a path metric on $X$ which is compatible with the topology of $X$. 
\item With respect to the path metric $d_X$, the map $f \colon X \to Y$ is a local isometric embedding. 
\item For each $x,y \in X$, we have $d_Y(f(x), f(y)) \leq d_X(x,y)$. In particular, the map $f$ is Lipschitz continuous. 
\end{enumerate}
\end{lemma}
\begin{proof}
1. It is clear that $d_X$ is a pseudometric. It is a path metric because $X$ is Hausdorff. We claim that the topology that $d_X$ induces on the set underlying $X$ agrees with the topology of $X$. To see this, let $x \in X$ and let $x \in U \subset X$ be an open neighbourhood such that the induced map $U \to f(U)$ is a homeomorphism. It suffices to show that the topology of $U$ agrees with the topology given by the metric restriction $d_U = d_X|_U$. This follows from the fact that the topology of $f(U) \subset Y$ agrees with the topology given by the metric restriction $d_{f(U)}= d_Y|_{f(U)}$. 

2. Let $x \in X$ and let $U \subset X$ be an open neighbourhood of $x$ such that the induced map $f \colon U \to f(U)$ is a homeomorphism. Then the paths in $U$ are in canonical bijection with the paths in $f(U)$, and also their lengths coincide by construction. Thus $f \colon U \to f(U)$ is an isometry. 

3. Let $x,y \in X$. It suffices to show that for any path $\gamma \colon [a,b] \to X$ from $x$ to $y$, we have $\ell(\gamma) \geq d_Y(f(x), f(y))$. Notice that $f_\ast(\gamma) \colon [a,b] \to Y$ is a path on $Y$ of length $\ell(f_\ast(\gamma)) = \ell(\gamma)$. Hence $\ell(\gamma) = \ell(f_\ast(\gamma))  \geq d_Y(f(x), f(y))$ as desired. 
\end{proof}

\begin{lemma} \label{lemma:uniform}
Let $X$ be a metric space and $Y$ a locally compact path metric space. 
Let $f \colon X \to Y$ be a uniformly continuous map. Assume that $f$ is proper and that $Y$ is complete. 
Then $X$ is complete. 
\end{lemma}
\begin{proof}
This is well-known; we provide a proof for convenience of the reader. Let $\set{x_n}_n$ be a Cauchy sequence in $X$. We must show that $\set{x_n}_n$ has a limit $x \in X$. 
Since $f$ is uniformly continuous, $\set{f(x_n)}_n$ is a Cauchy sequence in $Y$. In particular, for some $N>0$, the elements $f(x_n)$ for $n \geq N$ are all contained in the closed ball $\overline{B}_\epsilon(x_N) = \set{y \in Y \mid d(y,x_N) \leq \epsilon} \subset Y$ of radius $\epsilon$ for some $\epsilon>0$ (where $d$ is the metric on $Y$). Since $Y$ is a complete locally compact path metric space, the closed balls in $Y$ are compact (Hopf--Rinow theorem, see e.g.\ \cite[Section 1.9]{Gromov2007}). Thus, replacing $\set{x_n}_n$ by $\set{x_n}_{n \geq N}$, we may assume that $\set{f(x_n)}_n$ is contained in a compact subset $K \subset Y$. Since $f$ is proper, the subset $f^{-1}(K) \subset X$ is compact. Therefore, each $x_n \in X$ is contained in the compact subset $f^{-1}(K) \subset X$, hence the sequence $\set{x_n}_n$ of points $x_n \in f^{-1}(K)$ converges to some point $x \in f^{-1}(K)$. 
\end{proof}

\begin{lemma} \label{lemma:metricquotient}
Let $X$ be a path-connected topological space equipped with a metric $d_X$, and let $\Gamma$ be a group that acts isometrically on $X$ with closed orbits. Define a pseudometric $d$ on the quotient space $Y = \Gamma \sm X$ as follows: for orbits $[x_1], [x_2] \in Y$ of points $x_1, x_2 \in X$, put $d([x_1], [x_2]) = \inf_{\gamma \in \Gamma} d_X(x_1, \gamma x_2)$. The following holds. 
\begin{enumerate}
\item The pseudometric $d$ is a metric, and if $d_X$ is a path metric, then so is $d$. 
\item If the metric space $X$ is complete, then so is $Y$. 
\end{enumerate}
\end{lemma}
\begin{proof}
For the first assertion, see \cite[Section 1.16]{Gromov2007}. For the second assertion, assume that $X$ is complete. Let $\set{[x_n]}_n$ be a Cauchy sequence in $Y$. We aim to show that $\set{[x_n]}_n$ has a limit in $Y$. Up to replacing $\set{[x_n]}_n$ by a subsequence $\set{[x_{n_i}]}_{i}$  that satisfies $d([x_{n_i}], [x_{n_{i+1}}]) \leq 2^{-i}$, we may assume that $d([x_n], [x_{n+1}]) \leq 2^{-n}$. Let $p_1 \in [x_1]$ such that $d_X(p_1, [x_2]) = d([x_1], [x_2]) \leq \frac{1}{2}$. 
Recursively choose $p_n \in [x_n]$ for $n \geq 2$ such that $d_X(p_{n-1}, p_{n}) \leq d_X(p_{n-1}, [x_{n}]) + 2^{-(n-1)}$. This gives \begin{align*}d_X(p_{n-1}, p_{n}) &\leq d_X(p_{n-1}, [x_{n}]) + 2^{-(n-1)} \\
&= d_X([x_{n-1}], [x_{n}]) + 2^{-(n-1)} \leq 2^{-(n-1)} + 2^{-(n-1)} = 2^{-(n-2)}\end{align*} for each $n \geq 2$. Hence our sequence $\set{p_n}_n$ of elements in $X$ is Cauchy; since $X$ is complete, it has a limit $p \in X$. The orbit $[p] \in Y$ of $p \in X$ is the limit of the sequence $\set{[x_n]}_n$, and we are done.
\end{proof}

\subsection{Path metric of the attachment} \label{prooftheorem} 
Let 
$
p \colon \wt Y \to Y(\mr L) = \wt Y/\sim
$
be the projection map, and let 
$
\wt{\mr P} \colon \wt Y \to \CC H^n
$
be the map defined as $\wt{P}(x,\alpha) = x$. 
By Lemma \ref{intersection}, we obtain continuous maps \[
\mr P  \colon Y(\mr L) \longrightarrow \CC H^n \quad \tn{ and } \quad \overline{\mr P} \colon M(\mr L) = L \setminus Y(\mr L) \lr L \setminus \CC H^n
\] 
such that the following diagram commutes:
\begin{align} \label{align:diagram-P}
\begin{split}
\xymatrixcolsep{3pc}
\xymatrix{
\wt Y \ar[d]^-p \ar@/^/[dr]^-{\wt{\mr P}} & \\
Y(\mr L) \ar[r]^-{\mr P} \ar[d]& \CC H^n \ar[d] \\
M(\mr L) \ar[r]^-{\overline{\mr P}} & L \sm \CC H^n.
}
\end{split}
\end{align}


\begin{lemma} \label{lemma:compactfinite}
For each compact subset $Z \subset \CC H^n$ there are only finitely many $\alpha \in P\mr A$ with $Z \cap \RR H^n_\alpha \neq \emptyset$. 
\end{lemma}
\begin{proof}
The subgroup $L' \subset \textnormal{Isom}(\CC H^n)$ (see (\ref{anti-iso-group}) and Lemma \ref{Gammaembedding}) 
acts properly discontinuously on $\CC H^n$. So if $S$ is the set of $\alpha \in P\mr A$ such that $\alpha Z \cap Z \neq \emptyset$, then $S$ is finite. Moreover, if $\alpha \in P\mr A$ is such that $Z \cap \RR H^n_\alpha \neq \emptyset$, then there exists $z \in Z$ with $\alpha(z) = z$, so that 
$\alpha \in S$. 
In particular, there are only finitely many such $\alpha$. 
\end{proof}

\begin{lemma}
The map $\widetilde{ \mr P} \colon \widetilde Y \to \CC H^n$ is proper with finite fibres. In particular, the map $\widetilde{ \mr P}$ is closed. 
\end{lemma}
\begin{proof}
Let $Z \subset \CC H^n$ be a compact set. 
Then $Z$ meets only finitely many of the spaces $\RR H^n_\alpha$ for $\alpha \in P\mr A$, see Lemma \ref{lemma:compactfinite}. Each $Z \cap \RR H^n_\alpha$ is closed in $Z$ as $\RR H^n_\alpha$ is closed in $\CC H^n$, so $Z \cap \RR H^n_\alpha$ is compact for each $\alpha \in P\mr A$. We conclude that ${\widetilde{ \mr P}}^{-1}(Z) = \coprod_{\alpha \in P\mr A} (Z \cap \RR H^n_\alpha)$ is compact. 
In particular, $\widetilde{\mr P}$ is proper. Let $z \in \CCH^n$. Applying Lemma \ref{lemma:compactfinite} to the case where $Z = \set{z} \subset \CCH^n$ is a singleton shows that $z \in \RRH^n_\alpha$ for only finitely many $\alpha \in P\mr A$. Since $\wt{\mr P}^{-1}(z) = \{(x, \alpha) \in \wt Y \mid x = z\}$, this proves that $\wt{\mr P}^{-1}(z)$ is finite. As $\wt{\mr P}$ is proper, it is closed by Lemma \ref{lemma:lee}.
\end{proof}


\begin{corollary} \label{corollary:proper}
The map $\mr P \colon Y(\mr L) \to \CC H^n$ is proper with finite fibers, hence in particular, $\mr P$ is closed. Moreover, the map $\overline{\mr P} \colon M(\mr L) \to L \setminus \CC H^n$ has finite fibres. 
\end{corollary}
\begin{proof}
Consider diagram \eqref{align:diagram-P}. 
Since $p \colon \wt Y \to Y(\mr L)$ is surjective and $\wt{\mr P}$ is proper with finite fibres, the same holds for $\mr P$, and hence $\mr P$ is closed by Lemma \ref{lemma:lee}. 

Next, we show that the induced map $\overline{\mr P} \colon M(\mr L) \to L \backslash \CC H^n$ has finite fibres. To see this, let $x \in \CCH^n$ with image $[x] \in L \sm \CCH^n$. Let $\pi \colon \CCH^n \to L \sm \CCH^n$ be the quotient map. Then $\pi^{-1}([x]) = L\cdot x = \set{\gamma \cdot x}_{\gamma \in L}$ and hence $\mr P^{-1}(\pi^{-1}([x])) = \cup_{\gamma \in L} (\gamma \cdot \mr P^{-1}(x))$. Let $q \colon Y(\mr L) \to M(\mr L)$ be the quotient map. Then $\overline{\mr P}^{-1}([x]) = q(q^{-1}(\overline{\mr P}^{-1}([x]) )) = q\left(\mr P^{-1}(\pi^{-1}([x])) \right)  = q \left( \cup_\gamma \gamma \cdot \mr P^{-1}(x) \right) = q(\mr P^{-1}(x))$, and this is a finite set because the map $\mr P \colon Y(\mr L) \to \CCH^n$ has finite fibres. 
\end{proof} 

Our next goal is to prove that each point $x \in Y(\mr L)$ has a neighbourhood $V \subset Y(\mr L)$ that maps homeomorphically onto a finite union $\cup_{i = 1}^\ell \RR H^n_{\alpha_i} \subset \CC H^n$. Hence $x$ has an open neighourhood $x \in U \subset V$ that identifies with an open set in a union of copies of $\RR H^n$ in $\CC H^n$ under the map $\mr P$. This allows us to define a metric on $Y(\mr L)$ by pulling back the metric on $\CC H^n$, so that $Y(\mr L)$ is piecewise isometric to $\RRH^n$. 

Fix a point $f \in Y(\mr L)$ and a point $(x,\alpha) \in \widetilde Y$ such that $p(x,\alpha) = f$. Let $\alpha_1, \dotsc, \alpha_\ell$ be the elements in $P\mr A$ such that $(x, \alpha_i) \sim (x,\alpha)$ for each $i \in \set{1, \dotsc, \ell}$ (since the group $G(x)$ is finite by Lemma \ref{Gxlemma}, these are finite in number). Consider the quotient map $$p\colon \widetilde Y \longrightarrow Y(\mr L)$$ and define
\begin{equation} \label{Kf}
Y_f \coloneqq p\left(\coprod_{i = 1}^\ell \RR H^n_{\alpha_i}\right) \subset Y(\mr L). 
\end{equation}


\begin{lemma} \label{lemma:stabilizer-preserves}
The following assertions are true.
\begin{enumerate}
\item 
The stabilizer $L_f$ of $f \in Y(\mr L)$ acts on the set $\set{\alpha_1, \dotsc, \alpha_\ell}$ by conjugation. 
\item
The stabilizer $L_f$ of $f \in Y(\mr L)$ preserves the subset $Y_f \subset Y(\mr L)$.
\end{enumerate}
\end{lemma}

\begin{proof}
Let $\psi \in L_f$ and $i \in \set{1, \dotsc, \ell}$. Notice that 
$\psi\cdot f= f$ implies that $\psi \cdot (x, \alpha_i) \sim (x, \alpha)$ hence $\psi\alpha_i\psi^{-1} = \alpha_j$ for some $j \in \set{1, \dotsc, \ell}$. This proves the first assertion. 

To prove the second assertion, let $g = p(y, \alpha_i) \in Y_f$ for some $i \in \set{1, \dotsc, \ell}$ and $y \in \RRH^n_{\alpha_i}$. Then $\psi \cdot p(y, \alpha_i) = p(\psi(y), \psi \alpha_i \psi^{-1}) = p(\psi(y), \alpha_j)$ for some $j \in \set{1, \dotsc, \ell}$ by the above. Hence $\psi \cdot p(y,\alpha_i) \in Y_f$, proving what we want. 
\end{proof}

\begin{lemma} \label{localisometry}
Consider the quotient map $p \colon \wt Y \to Y(\mr L)$. 
Let $f \in Y(\mr L)$ and let $(x,\alpha) \in \widetilde Y$ such that $p(x,\alpha) = f$. Let $\alpha_1, \dotsc, \alpha_\ell$ be the elements in $P\mr A$ such that $(x, \alpha_i) \sim (x,\alpha)$ for each $i \in \set{1, \dotsc, \ell}$.  The following assertions are true.
\begin{enumerate}
\item \label{un}The set $Y_f$ defined in \eqref{Kf} is closed in $Y(\mr L)$. 
\item \label{deux}We have $\mr P\left( Y_f \right) = \cup_{i = 1}^\ell \RR H^n_{\alpha_i} \subset \CC H^n$, and the map 
\begin{equation} \label{align:Pfmap}
\mr P_f \colon Y_f \longrightarrow \bigcup_{i = 1}^\ell \RR H^n_{\alpha_i} 
\end{equation}
induced by $\mr P$ is a homeomorphism. 
    \item \label{trois}The set $Y_f \subset Y(\mr L)$ contains an $L_f$-equivariant open neighbourhood $U_f$ of $f$ in $Y(\mr L)$, where $L_f \subset L$ is the stabilizer of $f \in Y(\mr L)$, such that the natural map 
    \[
    L_f \sm U_f \longrightarrow L \sm Y(\mr L) = M(\mr L)
    \]
is injective. 
\end{enumerate}

\end{lemma}

\begin{proof}
1. This follows from the proof of Lemma \ref{lemma:openembeddingtopological}.\ref{item:above-not-1}. 

2. We have 
\begin{equation*}
\mr P_f(Y_f) =  \mr P\left( p\left(\coprod_{i = 1}^\ell \RR H^n_{\alpha_i}\right)\right) = \widetilde{ \mr P} \left( \coprod_{i = 1}^\ell \RR H^n_{\alpha_i}\right) = \bigcup_{i = 1}^\ell \RR H^n_{\alpha_i} \subset \CC H^n. 
\end{equation*}
To prove injectivity of \eqref{align:Pfmap}, let $(y, \alpha_i), (z,\alpha_j) \in \widetilde Y$ and suppose that $y = z \in \CC H^n$. Then indeed, $(y, \alpha_i) \sim (z,\alpha_j)$ because $\sim$ is an equivalence relation by Lemma \ref{eqrel}. 

Finally, we prove that $\mr P_f$ is closed. Let $Z \subset Y_f$ be a closed set. Then $Z$ is closed in $Y$ by Part 1, hence $p^{-1}(Z)$ is closed in $\widetilde Y$, hence $\widetilde{ \mr P} \left( p^{-1}(Z) \right) )$ is closed in $\CC H^n$, so that 
\begin{equation*}
\mr P_f(Z) = \mr P(Z) = \widetilde{ \mr P} \left( p^{-1} \left( Z \right) \right) = \left(\widetilde{ \mr P} \left( p^{-1} \left( Z \right) \right)\right) \cap \left(\cup_{i = 1}^\ell \RR H^n_{\alpha_i} \right)
\end{equation*}
is closed in $\cup_{i = 1}^\ell \RR H^n_{\alpha_i} $.  

3. Let $x = \mr P(f) \in \CC H^n$. Since $\CC H^n$ is locally compact, there exists a compact set $Z \subset \CC H^n$ and an open set $U \subset \CC H^n$ with $x \in U \subset Z$. Since $Z$ is compact, it meets only finitely many of the $\RR H^n_\beta \subset \CC H^n$ (Lemma \ref{lemma:compactfinite}). Consequently, the same holds for $U$; define $V = \mr P^{-1}(U) \subset Y(\mr L)$ and put
\[
\mr B = \{\beta \in P\mr A: U \cap \RR H^n_\beta \neq \emptyset \}.
\]
Also define, for $\beta \in P\mr A$, $Z_\beta \coloneqq p\left(\RR H^n_\beta \right) \subset Y(\mr L)$. Then 
\begin{equation*}
    f \in V \subset \bigcup_{\beta \in \mr B} Z_\beta = \bigcup_{\substack{\beta \in \mr B \\ \beta(x) = x}} Z_\beta \bigcup_{\substack{\beta \in \mr B \\ \beta(x) \neq x}} Z_\beta.  
\end{equation*}

Since each $Z_\beta$ is closed in $Y(\mr L)$ by the proof of Lemma \ref{lemma:openembeddingtopological}.\ref{item:above-not-1}, there is an open subset $V' \subset V$ with 
\begin{equation*}
f \in  V' \subset \bigcup_{\substack{\beta \in \mr B \\ \beta(x) = x}} Z_\beta = 
\bigcup_{\substack{\beta \in \mr B \\ \beta(x) = x \\ (x,\beta) \sim (x,\alpha)}} Z_\beta
\bigcup_{\substack{\beta \in \mr B \\ \beta(x) = x \\ (x,\beta) \not \sim (x,\alpha)}} Z_\beta.
\end{equation*}
Since each $Z_\beta$ is closed in $Y(\mr L)$, there exists an open subset $V'' \subset V'$ with
\begin{equation*}
    f \in  V'' \subset \bigcup_{\substack{\beta \in \mr B \\ \beta(x) = x \\ (x,\beta) \sim (x,\alpha)}} Z_\beta \subset 
    \bigcup_{\substack{\beta \in P \mr A \\ \beta(x) = x \\ (x,\beta) \sim (x,\alpha)}} Z_\beta = Y_f. 
\end{equation*}
We have $L_f \cdot Y_f = Y_f$ by Lemma \ref{lemma:stabilizer-preserves}, and the subset $U_f'' \coloneqq L_f \cdot V'' \subset L_f \cdot Y_f = Y_f$ is an $L_f$-equivariant open neighbourhood of $f$ in $Y(\mr L)$. 

We define $U_f \subset U_f''$ as follows. Let $x \in U_x \subset \CCH^n$ be an $L_x$-equivariant open neighbourhood such that $L_x \sm U_x \subset L \sm \CCH^n$. Then $U_x \cap \mr P(Y_f)$ is open in $\mr P(Y_f)$. We define $V_x \subset Y_f$ as the inverse image of $U_x \cap \mr P(Y_f)$ under the homeomorphism $Y_f \xrightarrow{\sim} \mr P(Y_f)$, and put 
\[
U_f \coloneqq U_f'' \cap V_x. 
\] 
Then $V_x$ is open in $Y_f$, and hence $U_f = U_f'' \cap V_x$ is open in $U_f'' \subset Y_f$, and therefore in $Y(\mr L)$. Moreover, since $U_x$ is $L_x$-equivariant and $L_f \subset L_x$, we see that $V_x$ is $L_f$-equivariant, and hence $U_f$ is $L_f$-equivariant as well. 

We claim that the natural map 
$
L_f \sm U_f \to L \sm Y(\mr L)
$
is injective. To prove this, let $(y, \alpha_i), (z, \alpha_j) \in \wt Y$ such that $p(y, \alpha_i), p(z, \alpha_j) \in U_f \subset Y_f$, and let $\psi \in L$ such that $\psi \cdot (y, \alpha_i) = (z, \alpha_j)$. It suffices to show that $\psi \in L_f$. Observe that $(\psi(y), \psi \alpha_i \psi^{-1}) \sim (z, \alpha_j)$. In particular, $\psi(y) = z$, and since $y,z \in U_x$ and the map $L_x \sm U_x \to L \sm \CCH^n$ is injective, we get that $\psi \in L_x$. Consequently, we have $x,z \in \RRH^n_{\psi \alpha_i \psi^{-1}} \cap \RRH^n_{\alpha_j}$, 
and the fact that $(z, \psi \alpha_i \psi^{-1}) \sim (z, \alpha_j)$ implies, in view of Lemma \ref{intersection}, that $(x, \psi \alpha_i \psi^{-1}) \sim (x, \alpha_j)$. But then $\psi \alpha_i \psi^{-1} = \alpha_k$ for some $k \in \set{1, \dotsc, \ell}$, which implies that $\psi \in L_f$, and we are done. 
\end{proof}
%


\begin{lemma} \label{lemma:hausdorff}
The topological space $Y(\mr L)$ is Hausdorff. 
\end{lemma}
\begin{proof}
Let $f , f' \in Y(\mr L)$ be elements such that $f \neq f'$. First suppose that $f \not \in Y_{f'}$. Since $Y_{f'}$ is closed in $Y(\mr L)$ by Lemma \ref{localisometry}, there is an open neighbourhood $U$ of $f$ such that $U \cap Y_{f'} = \emptyset$. Thus, $U \cap U_{f'}  \subset U \cap Y_{f'} = \emptyset$, where $U_{f'} \subset Y_{f'}$ is the open neighbourhood of $f'$ constructed in item 3 of Lemma \ref{localisometry}. 

Next, suppose that $f \in Y_{f'}$. Lift $f$ and $f'$ to elements $(x,\alpha),  (y,\beta) \in \widetilde Y$. Assume first that $x = y$. This means that $\mr P(f) = \mr P(f')$. Since $\mr P: Y_{f'} \to \CC H^n$ is injective, this implies that $f = f'$, contradiction. So we have $x \neq y \in \CC H^n$. But $\CC H^n$ is Hausdorff, so there are open subsets $\left(U \subset \CC H^n, V \subset \CC H^n \right)$ such that $x \in U$, $y \in V$ and $U \cap V = \emptyset$. Then $\mr P^{-1}(U) \cap \mr P^{-1}(V) = \emptyset$ and we are done. 
\end{proof}


\begin{definition}\label{metric}
Define a path metric $d_i$ on each path-connected component $Y(\mr L)_i$ of $Y(\mr L)$ by pulling back the metric on $\CCH^n$ via the map $\mr P \colon Y(\mr L)_i \to \CCH^n$ as in Lemma \ref{lemma:pathmetrics-pullback}; note that $\mr P$ is a local topological embedding by 
Lemma \ref{localisometry}. For each path-connected component $Y(\mr L)_i \subset Y(\mr L)$, the map $\mr P \colon Y(\mr L)_i \to \CCH^n$ is a Lipschitz continuous local isometric embedding, see Lemma \ref{lemma:pathmetrics-pullback}. 
\end{definition}

\begin{proposition} \label{prop:pathmetricquotient}
Consider the path metrics $d_i$ on the path-connected components of $Y(\mr L)$, see Definition \ref{metric}.  
Let $M(\mr L)_i$ be a path-connected component of $M(\mr L)$, and let $Y(\mr L)_i \subset Y(\mr L)$ be a path-connected component of $Y(\mr L)$ that surjects onto $M(\mr L)_i$ (so that $M(\mr L)_i = L_i \sm Y(\mr L)_i$, where $L_i \subset L$ denotes the stabilizer of $Y(\mr L)_i$ in $L$). For orbits $[x_1], [x_2] \in M(\mr L)_i$ of points $x_1, x_2 \in Y(\mr L)_i$, define $d_i([x_1], [x_2]) = \inf_{\gamma \in L}(d_i(x_1, \gamma x_2))$. Then this defines a path metric on $M(\mr L)_i$. 
\end{proposition}

\begin{proof}
By Lemma \ref{lemma:metricquotient}, it suffices to show that $L$ acts by isometries on $Y(\mr L)$ with closed orbits. To prove this, we first claim that $L$ acts isometrically on $Y(\mr L)$. 
For a curve $\gamma$ on $Y(\mr L)$ or on $\CCH^n$, let $\ell(\gamma)$ denote its length. By construction (cf.\ Lemma \ref{lemma:pathmetrics-pullback} and Definition \ref{metric}), we have $\ell(\gamma) = \ell(\mr P \circ \gamma)$ for each $\gamma \colon [a,b] \to Y(\mr L)$. Let $x,y \in Y(\mr L)_i$; it suffices to show that for each curve $\gamma \colon [a,b] \to Y(\mr L)$ with $\gamma(a) = x$ and $\gamma(b) = y$, the length of $\gamma$ equals the length of $g_\ast(\gamma) = g \circ \gamma$. This holds indeed, for the $L$-equivariance of the map $\mr P$ (cf.\ Lemma \ref{intersection}) implies that:
\[
\ell(g_\ast(\gamma)) = \ell(\mr P \circ g \circ \gamma) = \ell(g_\ast(\mr P \circ \gamma)) = \ell(\mr P \circ \gamma) = \ell(\gamma). 
\]
Next, we check that the $L$-orbits are closed in $Y(\mr L)$. Let $f \in Y(\mr L)$ with representative $(x,\alpha)\in \widetilde Y$. By equivariance of $p\colon \widetilde Y \to Y(\mr L)$, we have $p^{-1} \left( L \cdot f \right) = L \cdot \left(p^{-1}f \right)$, and since $p$ is a quotient map, it therefore suffices to show that $L \cdot \left(p^{-1}f \right)$ is closed in $\wt Y$. Moreover, $p^{-1}(f) = \{(x, \beta) \in \wt Y \colon (x,\beta) \sim (x, \alpha)\}$ is a finite set, so it suffices to show that the orbit $L \cdot (x,\beta)$ is closed in $\widetilde Y$ for each $\beta \in P\mr A$. 
Since $L$ is discrete, it suffices to show that $L$ acts properly on $\widetilde Y$. So let $Z \subset \widetilde Y$ be any compact set: we claim that $\{g \in L: g Z \cap Z \neq \emptyset \}$ is a finite set. Indeed, for each $g \in L$, one has $\widetilde{ \mr P} \left( g Z \cap Z  \right) \subset g \widetilde{\mr P}(Z) \cap \widetilde{\mr P}(Z)$, and the latter is non-empty for only finitely many $g \in L$, by properness of the action of $L$ on $\CC H^n$. 

Since the metric on $Y(\mr L)_i$ is a path metric, the same holds for the metric on $M(\mr L)_i$, see Lemma \ref{lemma:metricquotient}. This proves the proposition.
\end{proof}

\begin{proposition} \label{proposition:completeness}
The path metrics on the path-connected components of $Y(\mr L)$ and $M(\mr L)$ defined in Definition \ref{metric} and Proposition \ref{prop:pathmetricquotient} are complete. 
\end{proposition}

\begin{proof}
Let $Y(\mr L)_i \subset Y(\mr L)$ be a path-connected component with image $M(\mr L)_i \subset M(\mr L)$. 
Consider the map $\mr P \colon Y(\mr L)_i \to \CCH^n$. By Lemma \ref{lemma:pathmetrics-pullback}, the map $\mr P$ is Lipschitz continuous, and in particular uniformly continuous. Moreover, by Corollary \ref{corollary:proper}, the map $\mr P$ is proper. Consequently, the metric space $Y(\mr L)_i$ is complete because $\CCH^n$ is complete, see Lemma \ref{lemma:uniform}. By Lemma \ref{lemma:metricquotient}, the metric space $M(\mr L)_i$ is therefore complete as well, and we are done. 
\end{proof}

\subsection{Orbifold structure of the attachment} The next step is to prove that the space $M(\mr L) = L \setminus Y(\mr L)$ (see Definition \ref{gluedspace}) is locally isometric to quotients of open sets in $\RR H^n$ by finite groups of isometries. 
\begin{definition} \label{definition:nodes-ii}
Let $f \in Y(\mr L)$ with representative $ (x,\alpha) \in \widetilde Y$. Thus, $x$ is an element in $\CCH^n$, and $\alpha \in P\mr A$ is the class of an anti-unitary involution such that $\alpha(x) = x$. 
\begin{enumerate}
\item 
The \emph{nodes} of $f$ are by definition the nodes of $(x,\alpha)$ (see Definition \ref{fullset}). Thus, these are the hyperplanes $H \in \ca H(x)$, i.e.~the hyperplanes $H_r \in \ca H$ defined by short roots $r \in \mr R$ such that $x \in H_r$ (equivalently, such that $h(x,r) = 0$). 
\item 
The number of nodes of $f$ is the cardinality of $\ca H(x)$. 
\item 
The anti-unitary involution $\alpha \in P\mr A$ induces an involution on the set $\ca H(x)$ by Lemma \ref{alphaswitch}. 
Let $H \in \ca H(x)$ be a node of $x$. We call $H$ a \emph{real node} of $x$ if $\alpha(H) = H$. We call $(H, \alpha(H))$ a \emph{pair of complex conjugate nodes} of $x$ if $\alpha(H) \neq H$. 
\item 
If $k$ is the number of nodes of $f$, we generally write $k = 2a + b$, with $a$ the number of pairs of complex conjugate nodes of $x$, and $b$ the number of real nodes of $x$. We say that $f$ \emph{has $k$ nodes} (resp.\ \emph{$a$ pairs of complex conjugate nodes and $b$ real nodes}) if $x$ has $k$ nodes (resp.\ $a$ pairs of complex conjugate nodes and $b$ real nodes). 
\end{enumerate}
\end{definition}

Fix again a point $f \in Y(\mr L)$ and a point $(x,\alpha) \in \widetilde Y$ lying above $f$. Let $k = 2a + b$ be the number of nodes of $f$. Thus $x \in \RRH^n_\alpha$, and there exist $r_1, \dotsc, r_k \in \mr R$ such that 
\[
\ca H(x) = \set{H_{r_1}, \dotsc, H_{r_k}}, \quad G(x) = \langle \phi_{r_1}, \dotsc, \phi_{r_\ell} \rangle \cong (\ZZ/m)^k.
\]
For $\beta \in P\mr A$, recall that $(x,\beta) \sim (x,\alpha)$ if and only if $\alpha \circ \beta \in G(x)$. We relabel the $r_i$ so that they satisfy the following condition: 
\begin{align} \label{relabel}
\begin{split}
\alpha(H_{r_i}) = 
\begin{cases}
 H_{r_{i+1}} \tn{ for } i \tn{ odd and } i \leq 2a, \\
H_{r_{i-1}} \tn{ for } i \tn{ even and } i \leq 2a, \tn{ and } \\ 
H_{r_{i}} \tn{ for } i  \in \set{2a +1, \dotsc, k}.
\end{cases}
\end{split}
\end{align}

In other words, $H_{r_i}$ is a real node if and only if $i > 2a$, and $\left(H_{r_i}, H_{r_{i+1}}\right)$ is a pair of complex conjugate nodes if and only if $i < 2a$ is odd. 

\begin{lemma} \label{lemma:localcoordinates}
Continue with the notation from above.  
\begin{enumerate}
    \item \label{ein}Let $\beta \in P\mr A$ be such that $(x,\beta) \sim (x,\alpha)$. Then 
    \[
    \beta = \prod_{i = 1}^a \left( \phi_{r_{2i-1}} \circ \phi_{r_{2i}} \right)^{j_i} \circ \prod_{i = 2a+1}^k \phi_{r_i}^{j_i} \circ \alpha
    \]
    for some $j_1, \dotsc, j_a, j_{2a+1}, \dotsc, j_k \in \ZZ/m$. In particular, there are $m^{a + b}$ such $\beta$. 
    \item \label{zwei}There is an isometry $\CC H^n \xrightarrow{\sim} {\bb B}^n(\CC)$ identifying $x$ with the origin, $\phi_{r_i}$ with the map \begin{equation*}
        {\bb B}^n(\CC) \to {\bb B}^n(\CC), \white (t_1, \dotsc,t_i, \dotsc, t_n) \mapsto (t_1, \dotsc, \zeta t_i, \dotsc, t_n),
    \end{equation*}
    and $\alpha$ with the map defined by 
    \begin{equation} \label{eq:alpha}
        t_i \mapsto 
\begin{cases}
\bar t_{i+1} \white \textnormal{for $i$ odd and $i \leq 2a$}\\
\bar t_{i-1} \white \textnormal{for $i$ even and $i \leq 2a$}\\
\bar t_i \white\white \textnormal{for $i > 2a$}.
\end{cases}
    \end{equation}
\end{enumerate}
\end{lemma}
\begin{proof}
1. This follows readily from Proposition \ref{preliminaryproposition}. 

2. 
Define
\[
T =  \langle x \rangle \oplus \langle r_1 \rangle  \oplus \cdots \oplus \langle r_k \rangle  \subset V, 
\quad 
W = T^\perp = \{w \in V \mid h(w,t) = 0 \;\; \forall t \in T\}. 
\]
Observe that $\alpha(W) = W$. Since $W \subset \langle x \rangle^\perp$, the hermitian space $(W, h|_{W})$ is positive definite. Let $\{w_1, \dotsc, w_{n-k}\} \subset W$ be an orthonormal basis such that $\alpha(w_i) = w_i$ (such a basis exists because $h|_W$ restricts to a positive definite symmetric bilinear form on $W^\alpha$ that in turn induces a positive definite hermitian form on $W^\alpha \otimes_\RR \CC$ such that the canonical isomorphism $W^\alpha \otimes_\RR \CC = W$ preserves the hermitian forms). By equation \eqref{relabel} and Lemma \ref{alphaswitch}, there exist $\lambda_1, \dotsc, \lambda_k \in K^\ast$ with $\va{\lambda_i} = 1$, so that:
\begin{align} \label{relabel:2}
\begin{split}
\alpha(r_i) = 
\begin{cases}
 \lambda_i \cdot r_{i+1} \tn{ for } i \tn{ odd and } i \leq 2a, \\
\lambda_i \cdot r_{i-1} \tn{ for } i \tn{ even and } i \leq 2a, \tn{ and } \\ 
\lambda_i \cdot r_{i} \tn{ for } i  \in \set{2a +1, \dotsc, k}.
\end{cases}
\end{split}
\end{align}
Let $\rho_i \in \CC^\ast$ 
such that $\lambda_i = \rho_i^2$ and $\va{\rho_i} = 1$. 
We consider the following basis of $V$:
\begin{align} \label{align:basis-V}
V = \langle x', r_1', \dotsc, r_k', w_1, \dotsc, w_{n-k} \rangle, \quad \quad r_i' \coloneqq \rho_i \cdot r_i, \quad x' \coloneqq \frac{x}{\sqrt{-h(x,x)}}.
\end{align}
We have $\alpha(r_i') = r_{\alpha(i)}'$. 
Moreover, $h(x', x') = (-h(x,x))^{-1} \cdot h(x,x) = -1$ and $h(r_i', r_i') = \va{\rho_i}^2 \cdot h(r_i, r_i) = 1$. Thus, the basis \eqref{align:basis-V} induces an isomorphism of hermitian vector spaces
\begin{align} \label{isomorphism-hermitian-vector}
(V, h) \xlongrightarrow{\sim} (\CC^{n+1}, H),
\end{align}
where $H(x,y) = -x_0\bar y_0 + x_1 \bar y_1 + \cdots + x_n \bar y_n$, and the isometry $\CCH^n \xrightarrow{\sim} \BB^n(\CC)$ induced by \eqref{isomorphism-hermitian-vector} has all the required properties.
\end{proof}
\begin{definition}
\begin{enumerate}
\item
Define $L_f = \textnormal{Stab}_{L}(f)$ to be the subgroup of $L$ fixing $f \in Y(\mr L)$. This contains the group $G(x) \cong (\ZZ/m)^k$. 
\item 
Define $B_f$ as the subgroup of $G(x)$ generated by the order $m$ complex reflections associated to the real nodes of $f$. 
Hence $B_f  = \langle \phi_{r_i} \rangle_{i > 2a} \cong (\ZZ/m)^b$. 
\end{enumerate}
\end{definition}


We also need the following lemma. 

\begin{lemma} \label{lemma:T/G}
Write $m = 2^{e} \cdot d$ with $d \not \equiv 0 \bmod 2$. 
Let $T \coloneqq \{t \in \CC: t^m \in \RR\}$. Let $\langle \zeta_m \rangle$ act on $T$ by multiplication. 
Then each element in $\langle \zeta_m \rangle \sm T$ has a unique representative of the form $\zeta_{2^{e+1}}^\epsilon \cdot r \in T$ with $r  \in \RR_{\geq 0}$ and $\epsilon \in \{0,1\}$. 
\end{lemma}
\begin{proof}
Let $T_1 = \set{t \in T \colon \va{t} = 1}$. It suffices to show that each element in $\langle \zeta_m \rangle \sm T_1$ has a unique representative of the form $\zeta_{2^{e+1}}^{\epsilon}$ for $\epsilon \in \set{0,1}$. Now notice that $T_1 = \langle \zeta_{2m} \rangle$ and that $\zeta_{2m} = \zeta_{2^{e+1}d}$. In particular, we have an isomorphism
\begin{align*}
\langle \zeta_{2m} \rangle / \langle \zeta_{m} \rangle = \langle \zeta_{2^{e+1}} \rangle / \langle \zeta_{2^e} \rangle & \xlongrightarrow{\sim} \ZZ/2, \\
\zeta_{2^{e+1}} &\mapsto 1,
\end{align*}
and the result follows. 
\end{proof}

We obtain the key to Theorem \ref{glueingtheorem1}. Consider the quotient map $p\colon \widetilde Y \to Y(\mr L)$, and the subset $Y_f \subset M(\mr L)$, see (\ref{Kf}). 

\begin{proposition} \label{localmodel}
Consider the above notation; in particular, we consider $f \in Y(\mr L)$ and a point $(x,\alpha) \in \widetilde Y$ lying above $f$. 
The following assertions are true. 
\begin{enumerate}
    \item\label{caseone} If $f$ has no nodes, then $G(x) = B_f$ is trivial, and $Y_f = \RR H^n_\alpha \cong {\bb B}^n(\RR)$. 
    \item\label{casetwo} If $f$ has only real nodes, then $B_f \setminus Y_f$ is isometric to ${\bb B}^n(\RR)$. 
    \item\label{casethree} If $f$ has $a$ pairs of complex conjugate nodes ($k = 2a$), and no other nodes, then $B_f\setminus Y_f = Y_f$ is the union of $m^a$ copies of ${\bb B}^n(\RR)$, any two of which meet along a totally geodesic ${\bb B}^{2c}(\RR) \subset \BB^n(\RR)$ for some integer $c$ with $0 \leq c \leq a$.
    \item\label{casefour} More generally, if $f$ has $2a$ complex conjugate nodes and $b$ real nodes, then there is an isometry between $B_f \setminus Y_f$ and the union of $m^a$ copies of ${\bb B}^n(\RR)$ identified along common ${\bb B}^{2c}(\RR)'s$ (that is, the set $Y_f$ of case \ref{casethree} above) such that $[f] \in B_f \setminus Y_f$ is identified with the origin in any of the copies of  ${\bb B}^n(\RR)$. 
    \item\label{casefive} Assume $f$ has $2a$ complex conjugate nodes and $b$ real nodes, and consider the above isometry between $B_f \setminus Y_f$ and the union $\cup_i B_i$ of copies $B_i$ of ${\bb B}^n(\RR)$. Then $L_f$ acts transitively on these copies of ${\bb B}^n(\RR)$ (i.e., for each $i \neq j$ and each $p \in B_j$ there exists $\phi \in L_f$ such that $\phi \cdot p \in B_i$).  
   If ${\bb B}^n(\RR) = B_i$ is any one of them, and if $\Gamma_f = (L_f/B_f)_{{\bb B}^n(\RR)}$ is its stabilizer in the group $L_f/B_f$, then the natural map
    \begin{equation*}
        \Gamma_f \setminus {\bb B}^n(\RR) \to \left(L_f/B_f\right) \setminus \left(B_f \setminus Y_f \right) = L_f \setminus Y_f
    \end{equation*}
is an isometry of path metrics. Moreover, the action of $\Gamma_f$ on $\BB^n(\RR)$ fixes $0 \in \BB^n(\RR)$, and $0$ is sent to $[f] \in L_f \sm Y_f$ under the induced map $\BB^n(\RR) \to L_f \sm Y_f$. 
\end{enumerate}
\end{proposition}
\begin{proof}
1. This is clear. 

2. Suppose next that $f$ has $k$ real nodes, with $1 \leq k \leq n$. Then in the local coordinates $t_i$ of Lemma \ref{lemma:localcoordinates}.\ref{zwei}, we have that $\alpha: {\bb B}^n(\CC) \to {\bb B}^n(\CC)$ is defined by $\alpha(t_i) = \bar t_i$. Part \ref{ein} of the same lemma shows that any $\beta \in P\mr A$ fixing $x$ such that $(x,\alpha) \sim (x,\beta)$ is of the form 
\begin{equation*}
{\bb B}^n(\CC) \to {\bb B}^n(\CC), \white (t_1, \dotsc,t_i, \dotsc, t_n) \mapsto (\bar t_1 \zeta^{j_1}, \dotsc, \bar t_k \zeta^{j_k}, \bar t_{k+1}, \dotsc, \bar t_n). 
\end{equation*}
Since $f$ has $k$ real nodes and no complex conjugate nodes, we have (writing $j = (j_1, \dotsc, j_k)$ and $\alpha_j = \prod_{i = 1}^k \phi_{r_i}^{j_i} \circ \alpha$): 
\begin{equation*}
Y_f \cong  \bigcup_{j = (j_1, \dotsc, j_k) \in (\ZZ/m)^k} \RR H^n_{\alpha_j} \cong \left\{ (t_1, \dotsc, t_n) \in {\bb B}^n(\CC): t_1^m, \dotsc, t_k^m, t_{k+1}, \dotsc, t_n \in \RR \right\}. 
\end{equation*}
As before, write $m = 2^e \cdot d$ with $d$ odd. Each of the $2^k$ subsets 
\begin{align*}
    K_{f, \epsilon_1, \dotsc, \epsilon_k} \coloneqq 
    \left\{ (t_1, \dotsc, t_n) \in {\bb B}^n(\CC): \zeta_{2^{e+1}}^{-\epsilon_1}t_1, \dotsc, \zeta_{2^{e+1}}^{-\epsilon_k}t_k \in \RR_{\geq 0} \textnormal{ and } t_{k+1}, \dotsc, t_n \in \RR    \right\},
\end{align*}
indexed by $\epsilon_1, \dotsc, \epsilon_k \in \{0,1\}$, is isometric to the closed region in ${\bb B}^n(\RR)$ bounded by $k$ mutually orthogonal hyperplanes. By Lemma \ref{lemma:T/G}, their union $U$ is a fundamental domain for $B_f$, in the sense that it maps homeomorphically and piecewise-isometrically onto $B_f \setminus Y_f$. Under its path metric, $U = \cup K_{f, \epsilon_1, \dotsc, \epsilon_k} $ is isometric to ${\bb B}^n(\RR)$ by the following map: 
\begin{equation*}
    U \to {\bb B}^n(\RR), \white (t_1, \dotsc, t_k) \mapsto \left((-\zeta_{2^{e+1}})^{-\epsilon_1}t_1, \dotsc, (-\zeta_{2^{e+1}})^{-\epsilon_k}t_k, t_{k+1}, \dotsc, t_n\right). 
\end{equation*}
This identifies $B_f \setminus Y_f$ with the standard ${\bb B}^n(\RR) \subset {\bb B}^n(\CC)$. 

3. Now suppose $f$ has $a$ pairs of complex conjugate nodes 
$H_{r_1}, \dotsc, H_{r_{2a}}$ such that \eqref{relabel} holds, 
where $1 \leq a \leq \lfloor n/2 \rfloor$. There are now $m^{a}$ anti-isometric involutions $\alpha_{j}$ fixing $x$ and such that $(x, \alpha_{j}) \sim (x,\alpha)$: they are given in the coordinates $t_i$ as follows, taking $j = (j_1, \dotsc, j_a) \in (\ZZ/m)^a$:
\begin{equation} \label{alpha-imaginary}
    \alpha_j \colon (t_1, \dotsc, t_n) \mapsto (\bar t_2 \zeta^{j_1}, \bar t_1 \zeta^{j_1}, \dotsc, \bar t_{2a} \zeta^{j_{a}}, \bar t_{2a-1} \zeta^{j_a}, \bar t_{2a+1}, \dotsc, \bar t_n).
\end{equation}
Thus, any fixed-point set $\RR H^n_{\alpha_j}$ is identified with   the copy $ {\bb B}^n(\RR)_{\alpha_j} $ of $\BB^n(\RR)$ defined as 
\begin{align} \label{fix-piont-def}
    {\bb B}^n(\RR)_{\alpha_j} = \left\{ (t_1, \dotsc, t_n) \in {\bb B}^n(\CC) \colon t_{i} = \bar t_{i-1} \zeta^{j_i} \textnormal{ for $i \leq 2a $ even, } t_i \in \RR \textnormal{ for $i > 2a$} \right\},
\end{align}
and we have 
\begin{equation*}
Y_f \cong \bigcup_j \BB^n(\RR)_{\alpha_j} = 
\left\{ (t_1, \dotsc, t_n) \in {\bb B}^n(\CC) \colon
    t_i^m = \bar t_{i-1}^m \textnormal{if $i \leq 2a$ even, }
     t_i \in \RR \textnormal{ if } i > 2a    \right\}.
\end{equation*}
These $m^a$ copies of ${\bb B}^n(\RR)$ meet at the origin of ${\bb B}^n(\CC)$; in fact, for $j \neq j'$, the space ${\bb B}^n(\RR)_{\alpha_{j}}$ meets ${\bb B}^n(\RR)_{\alpha_{j'}}$ in a ${\bb B}^{2c}(\RR)$ if $c$ is the number of pairs $(j_i, j'_i)$ 
with $j_i = j'_i$. 

4. We proceed to treat the general case: we assume that $f$ has $k = 2a +b >0$ nodes, among which $a$ pairs of complex conjugate nodes and $b$ real nodes. In the local coordinates $t_i$ of Lemma \ref{lemma:localcoordinates}.\ref{zwei}, any anti-unitary involution fixing $x$ equivalent to $\alpha$ is of the form
\begin{align*}
\begin{split}
 \alpha_j \colon 
&(t_1, \dotsc, t_n) \mapsto \\
& (\bar t_2 \zeta^{j_1}, \bar t_1 \zeta^{j_1}, \dotsc, \bar t_{2a} \zeta^{j_{a}}, \bar t_{2a-1} \zeta^{j_a}, \bar t_{2a+1}\zeta^{j_{2a+1}}, \dotsc, \bar t_k \zeta^{j_{k}}, \bar t_{k+1}, \dotsc, \bar t_n),
\end{split}
\end{align*}where 
$$
j = (j_1, \dotsc, j_a, j_{2a+1}, \dotsc, j_{k}) \in (\ZZ/m)^{a+b}.
$$
 We now have $B_f \cong (\ZZ/m)^b$ acting by multiplying the $t_i$ for $2a+1 \leq i \leq k$ by powers of $\zeta$, and there are $m^{a+b}$ anti-unitary involutions $\alpha_j$. We have 
\begin{align*}
&Y_f \cong \bigcup_{j  \in (\ZZ/m)^{a+b}}^m \RR H^n_{\alpha_j} \cong \\ 
&\left\{ (t_1, \dotsc, t_n) \in {\bb B}^n(\CC) \mid t_2^{m} = \bar t_1^{m}, \dotsc, t_{2a}^{m} = \bar t_{2a-1}^{m}, t_{2a+1}^{m}, \dotsc, t_k^{m}, t_{k+1}, \dotsc, t_n \in \RR \right\}.
\end{align*}
We look at subsets $K_{f, \epsilon_1, \dotsc, \epsilon_k} \subset Y_f$ again, this time defined as 
\begin{align*}
 &K_{f, \epsilon_1, \dotsc, \epsilon_k}= \\
&    \left\{ t \in {\bb B}^n(\CC)\mid 
    t_i^m = \bar t_{i-1}^m \textnormal{ $i \leq 2a$ even, }
    \zeta_{2^{e+1}}^{-\epsilon_i}t_i \in \RR_{\geq 0} \textnormal{ $2a<i\leq k$, }
     t_i \in \RR, i > k    \right\}.
\end{align*}
As before, the natural map
$
U \coloneqq \bigcup_{\epsilon} K_{f, \epsilon} \to B_f \setminus Y_f
$ is an isometry. Define  
\begin{equation*}
\widetilde Y_f = 
\left\{ (t_1, \dotsc, t_n) \in {\bb B}^n(\CC):
    t_i^m = \bar t_{i-1}^m \textnormal{ for $i \leq 2a$ even, }
     t_i \in \RR, \textnormal{ for } i > 2a    \right\}.
\end{equation*}
Under its path metric, $U = \cup_\epsilon K_{f, \epsilon_1, \dotsc, \epsilon_k}$ is isometric to $\widetilde Y_f$ 
by the following map: 
\begin{align*}
    U \to \widetilde Y_f, \white 
&(t_1, \dotsc, t_k) \mapsto \\ 
&\left(t_1, \dotsc, t_{2a}, (-\zeta_{2^{e+1}})^{-\epsilon_1}t_{2a+1}, \dotsc, (-\zeta_{2^{e+1}})^{-\epsilon_k}t_k, t_{k+1}, \dotsc, t_n\right). 
\end{align*}
Hence $B_f \setminus Y_f \cong U \cong \widetilde Y_f$. Since $\widetilde Y_f$ is what $Y_f$ was in case \ref{casethree}, we are done. 

5. For $j = (j_1, \dotsc, j_a) \in (\ZZ/m)^a$, define an involution $\alpha_j \colon \BB^n(\CC) \to \BB^n(\CC)$ as in \eqref{alpha-imaginary}. Then the above shows that we have an isometry \begin{align}\label{coordinates:final}\begin{split}
B_f \sm Y_f &\cong \bigcup_{j  \in (\ZZ/m)^a} \BB^n(\RR)_{\alpha_j} \\
&= 
\left\{ (t_1, \dotsc, t_n) \in {\bb B}^n(\CC) \colon
    t_i^m = \bar t_{i-1}^m \textnormal{if $i \leq 2a$ even, }
     t_i \in \RR \textnormal{ if } i > 2a    \right\}.
\end{split}
\end{align} where $\BB^n(\RR)_{\alpha_j} \subset \BB^n(\CC)$ is the copy of $\BB^n(\RR)$ defined in \eqref{fix-piont-def}.  
The transitivity of the action of $L_f$ on these copies $\BB^n(\RR)_{\alpha_j} $ of ${\bb B}^n(\RR)$ follows from the fact that $G(x)\subset  L_f$ contains transformations multiplying $t_1, \dotsc, t_{2a}$ by powers of $\zeta$, hence the map $(t_i \mapsto \zeta^ut_i, t_{i-1} \mapsto t_{i-1})$ maps those $(t_{i-1}, t_i)$ with $t_i = \bar t_{i-1}\zeta^{j_i}$ to those $(t_{i-1}, t_i)$ with $t_i = \bar t_{i-1}\zeta^{j_1+u}$. 

Let $B$ be one of the copies of $\BB^n(\RR)$, and let $\Gamma_f$ be the stabilizer of $B$ in the group $L_f/B_f$. It remains to prove that the canonical map $\pi \colon \Gamma_f \setminus B\to L_f \setminus Y_f$ is an isometry. The surjectivity of $\pi$ follows from the transitivity of $L_f$ on the ${\bb B}^n(\RR)'s$. As $\pi$ is a piecewise isometry, we only need to prove injectivity. This will follow from the following elementary lemma. 
\begin{lemma} \label{lemma:elementary-lemma}
Let a group $G$ act on a set $X$, which is a union $X = \cup_{i \in I}Y_i$ of subsets $Y_i \subset X$. Assume that for each $i \in I$ and $g \in G$, we have $gY_i = Y_j$ for some $j \in I$. 
Suppose that for all $y \in X$, the stabilizer $\Stab_G(y)$ of $y$ in $G$ acts transitively on the sets $Y_i$ containing $y$, in the sense that for each $i$ and $j$ such that $y \in Y_i \cap Y_j$, there exists $g \in \rm{Stab}_G(y)$ such that $gY_i = Y_j$. Let $i_0 \in I$ and let $G_{i_0} \subset G$ be the stabilizer of $Y_{i_0}$. Then the natural map $G_{i_0} \setminus Y_{i_0} \to G \setminus X$ is injective.
\end{lemma} 
\begin{proof}[Proof of Lemma \ref{lemma:elementary-lemma}]
Let $x,y \in Y_{i_0}$ and $g \in G$ such that $g \cdot x = y$. Let $j \in I$ such that $gY_{i_0} = Y_j$. Then $y \in Y_{i_0} \cap Y_j$. Thus, there exists $h \in \text{Stab}_G(y)$ such that $hY_{j} = Y_{i_0}$, and $hg \cdot x = h \cdot y = y$. Define $f \coloneqq hg$. Then $f \in G_{i_0}$ and $f \cdot x = y$. 
\end{proof} 

We apply the lemma to the case $X = B_f \sm Y_f$ and $G = B_f \sm L_f$.  
Consider the isometry \eqref{coordinates:final}. 
Note that for each $j_1 \in \set{1, \dotsc, m^a}$ and each $\phi \in L_f$, we have $\phi(\BB^n(\RR)_{\alpha_{j_1}}) = \BB^n(\RR)_{\alpha_{j_2}}$ for some $j_2 \in \set{1, \dotsc, m^a}$. Let $y \in B_f \setminus Y_f$. To finish the proof of Proposition \ref{localmodel}, it suffices to prove that $\text{Stab}_{L_f/B_f}(y)$ acts transtivitely on the copies of ${\bb B}^n(\RR)$ containing $y$. 
In terms of the isometry \eqref{coordinates:final}, we have $$y \in B_f \sm Y_f \cong \bigcup_{j \in (\ZZ/m)^a} \BB^n(\RR)_{\alpha_j},$$
hence there exists $
    j' = (j_1', \dotsc, j_a') \in (\ZZ/m)^a$ such that  $$y = (t_1, \dotsc, t_n) \in \BB^n(\RR)_{\alpha_{j'}} \subset \BB^n(\CC).$$ In particular, $t_{i} = \bar t_{i-1} \zeta^{j_i'}$ for $i \leq 2a$ even, and 
$t_i \in \RR$ for $i > 2a$. 

If all $t_i$ for $i \leq 2a$ are non-zero, then  is contained in $\BB^n(\RR)_{\alpha_{j'}}$ but not in any $\BB^n(\RR)_{\alpha_{j}}$ with $j \neq j'$, so there is nothing to prove. Let us suppose that $t_1 = t_2 = 0$ and the other $t_i$ are non-zero. 
Then $y$ is contained in those $\BB^n(\RR)_{\alpha_{j}}$ with $j_i = j_i'$ for $i \geq 2$; there are $m$ of them. The group $\text{Stab}_{L_f/B_f}(y)$ contains transformations that multiply $t_1$ and $t_2$ by powers of $\zeta$ and leave the other $t_i$ invariant. Thus, $\text{Stab}_{L_f/B_f}(y)$ acts transitively on the set of $\BB^n(\RR)_{\alpha_{j}}$ containing $y$, because if $t_2 = \bar t_1 \zeta^{j_1}$, then $\zeta^{(j_1'-j_1)}t_2 = \bar t_1 \zeta^{j_1'}$. The general case is similar. 
\end{proof}

\begin{proof}[Proof of Proposition \ref{glueingtheorem1}]
\ref{completepart}. The path metrics on the path-connected components $Y(\mr L)_i$ of $Y(\mr L)$ were provided in Definition \ref{metric}; 
they were constructed in such a way that the map $\mr P \colon Y(\mr L)_i \to \CCH^n$ is a local isometric embedding. 

\ref{completepart:2}. The path metrics on the path-connected components of $M(\mr L)$ were defined in Proposition \ref{prop:pathmetricquotient}. Completeness follows from Proposition \ref{proposition:completeness}. 

\ref{item:importanttheoremitem}. Let $[f] \in M(\mr L)$ be the image of some element $f \in Y(\mr L)$. Then $[f]$ has a connected open neighborhood isometric to the quotient of a connected open set $W$ in $\RR H^n$ by a finite group of isometries $\Gamma_f$. To see this, define as before $L_f \subset L$ as the stabilizer of $f \in Y(\mr L)$, define $Y_f \subset Y(\mr L)$ as in equation (\ref{Kf}), and define an open neighbourhood $f \in U_f \subset Y(\mr L)$ with $U_f \subset Y_f$ as in Lemma \ref{localisometry}.\ref{trois}. Thus, by construction, the natural map $L_f \sm U_f \to L \sm Y(\mr L)$ is injective, hence an open embedding. 
By Proposition \ref{localmodel}.\ref{casefive}, we know that $L_f \setminus Y_f$ is isometric to $\Gamma_f \setminus \BB^n(\RR)$ for some finite group $\Gamma_f$ of isometries of $\BB^n(\RR)$ that fixes the origin $0 \in \BB^n(\RR)$. Thus, $L_f \setminus U_f \subset M(\mr L)$ is isometric to some open set $W'$ in $\Gamma_f \setminus \BB^n(\RR)$. Define $W'' \subset \BB^n(\RR)$ as the preimage of $W'$, and let $W \subset W''$ be the connected component of $W''$ that contains $0 \in W'' \subset \BB^n(\RR)$. Then $\Gamma_f$ preserves $W''$, and hence also $W$ since $\Gamma_f$ fixes $0 \in W'' \subset \BB^n(\RR)$. This concludes the proof of Proposition \ref{glueingtheorem1}. 
\end{proof}

\subsection{Uniformizing the attachment} \label{section:proof-first-main}


\begin{lemma} \label{lemma:pathmetric-orbifold}
Let $X$ be a path metric space such that for each $x \in X$ there exist an open neighbourhood $U$ of $x$ in $X$ and a connected hyperbolic manifold $M$ together with a finite group $G$ of isometries of $M$ such that $U$ and $G\sm M$ are isometric. Then $X$ admits a hyperbolic orbifold structure whose path metric is the given one. 
\end{lemma}
\begin{proof}
See \cite{lange-orbifolds-metric}. 
\end{proof}

\begin{proof}[Proof of Theorem \ref{theorem:introduction:uniformization}]
1. We want to prove that $L \setminus Y(\mr L)$ is naturally a real hyperbolic orbifold. In view of Proposition \ref{glueingtheorem1}, this follows from Lemma \ref{lemma:pathmetric-orbifold}. 

Consider the canonical open embedding of topological spaces
\begin{align} \label{align:openembedding-tops}
\coprod_{\alpha \in C\mr A}  L_\alpha \setminus \left(\RR H^n_\alpha - \mr H \right)  \longhookrightarrow M(\mr L) = L \setminus Y(\mr L),
\end{align}
see item \ref{item:openembedding} in Lemma \ref{lemma:openembeddingtopological}. Let us show that this map is in fact an open immersion of hyperbolic orbifolds. We begin with the following:

\emph{Claim:} Let $f = p(x, \alpha) \in Y(\mr L)$ be an element with no nodes (cf.\ Definition \ref{fullset}). Consider the stabilizer $L_f \subset L$ of $f \in Y(\mr L)$ in the group $L$, and the stabilizer $L_{\alpha,x} \subset L_{\alpha}$ of $x \in \RR H^n_\alpha$ in the group $L_\alpha$. Then $L_f = L_{\alpha,x} \subset L$. 

\emph{Proof of the Claim:} To prove that $L_f = L_{\alpha,x}$, we first observe that the quotient map $p \colon \widetilde Y \to Y(\mr L)$ induces an isomorphism between $L_{(x,\alpha)}$, the stabilizer of $(x,\alpha)\in \widetilde Y$ and $L_f$, the stabilizer of $f = [x, \alpha] \in Y(\mr L)$. So it suffices to show that $L_{(x,\alpha)} = L_{\alpha,x}$. By Lemma \ref{normalisator}, the subgroup $L_\alpha \subset L$ coincides with the normalizer $N_{L}(\alpha) \subset L$ of $\alpha$ in $L$, which implies that $    L_{\alpha,x} = L_{(x,\alpha)}$ because
\begin{align*}
L_{\alpha,x} = \left\{ g \in L_\alpha : gx = x \right\}  &= \left\{ g \in N_{L}(\alpha) : gx = x \right\} \\
&=       \left\{ g \in L: g\cdot (x,\alpha)= (g(x), g\alpha g^{-1}) = (x,\alpha)\right\} \\
& = L_{(x,\alpha)}. 
    \end{align*}
The claim is proved. Item \ref{oopensuborbifold} of the theorem can be deduced from it as follows. Let $f = p(x,\alpha) \in Y(\mr L)$ have no nodes. We have $Y_f = \RR H^n_\alpha$, hence 
\[
 L_f \setminus \RR H^n_\alpha= L_f \setminus Y_f = \Gamma_f \setminus \RR H^n \quad \tn{with} \quad \Gamma_f = L_f \setminus B_f =  L_f.
\] 
By construction, an orbifold chart of the space $L \setminus Y(\mr L)$ is given by $$W \to L_f \setminus W \subset L_\alpha \setminus (\RR H^n_\alpha - \mr H) \subset L \sm Y(\mr L)$$ for an invariant open subset $W$ of $\RR H^n_\alpha$ containing $x$. Because $L_f = L_{\alpha,x}$ by the above claim, this is also an orbifold chart for $\sqcup_{\alpha \in C\mr A}  L_\alpha \setminus \left(\RR H^n_\alpha - \mr H \right) $ at the point $(x,\alpha)$. 

Let $\mr Z \subset M(\mr L)$ be the complement of the image of the open immersion \eqref{align:openembedding-tops}. We claim that $\mr Z$ has measure zero. 
Let $[f] \in \mr Z$ be the image of some element $f \in Y(\mr L)$, and suppose that $f$ is the image of $(x,\alpha) \in \RR H^n_\alpha \subset \wt Y$. 
Define $Y_f \subset Y(\mr L)$ as in equation (\ref{Kf}), so that by Lemma \ref{localisometry}, we have $
Y_f \cong \bigcup_{i = 1}^\ell \RR H^n_{\alpha_i} 
$
where $\alpha_1, \dotsc, \alpha_\ell$ are the elements in $P\mr A$ such that $(x,\alpha_i) \sim (x, \alpha)$ for each $i \in \set{1, \dotsc, \ell}$. 
The fact that $\mr Z$ has measure zero is a local computation; 
by construction of the orbifold chart (see Proposition \ref{localmodel}), it thus follows from the fact that locally around $x \in \RR H^n_\alpha$, the space $\RR H^n_\alpha$ is glued to the spaces $\RR H^n_{\alpha_i}$ along a finite union of geodesic subspaces in $\RR H^n_\alpha$ each of which has measure zero as a subspace of $\RR H^n_\alpha$. 

2. The real hyperbolic orbifold $M(\mr L) = L \setminus Y(\mr L)$ is complete by item 1, so the uniformization of the connected components of $M(\mr L) $ follows from the uniformization theorem for complete $(G,X)$-orbifolds, see \cite[Proposition 13.3.2]{Thurston80}, \cite{matsumoto-montesinos} or \cite[Corollary 2.16, Chapter III.$\mathcal G$, page 603]{bridson-metric}. This proves item 2. 

3. We show that the complete real hyperbolic orbifold $M(\mr L)$ has finite volume. By the above, we know that the measure of $\mr Z \subset M(\mr L)$ is zero, where $\mr Z \subset M(\mr L)$ is the complement of the image of the open immersion \eqref{align:openembedding-tops}. Consequently, the volume of $M(\mr L)$ equals the volume of $M(\mr L) - \mr Z$. As we have an isomorphism of hyperbolic orbifolds $\sqcup_{\alpha \in C\mr A}  L_\alpha \setminus \left(\RR H^n_\alpha - \mr H \right)  \cong M(\mr L) - \mr Z$ by item 1, we get
\begin{align*}
\rm{vol}(M(\mr L)) &= \rm{vol}(M(\mr L)-\mr Z) = \rm{vol}\left(\coprod_{\alpha \in C\mr A}  L_\alpha \setminus \left(\RR H^n_\alpha - \mr H \right)\right) \\
&= \rm{vol}\left(\coprod_{\alpha \in C\mr A} L_\alpha \setminus \RR H^n_\alpha \right) 
= \sum_{\alpha \in C\mr A}\rm{vol}\left( L_\alpha \setminus \RR H^n_\alpha\right). 
\end{align*}
Since the set $C\mr A \cong L \sm P\mr A$ is finite by Proposition \ref{proposition:galois}, and since the volume of $L_\alpha \setminus \RR H^n_\alpha$ is finite by Theorem \ref{th:crucialthm-finitevolume}, we conclude that the volume of $M(\mr L)$ is finite. This concludes the proof of Theorem \ref{theorem:introduction:uniformization}. 
\end{proof}

\section{Preliminaries on totally geodesic immersions} \label{sec:totgeo} 

The goal of this section is to provide some preliminary results regarding totally geodesic immersions of complete connected hyperbolic orbifolds. 
The results in this section are well-known; we state and prove them for lack of suitable reference. 

Let $M$ and $\wt M$ be Riemannian manifolds, with respective Levi--Civita connections $\nabla$ and $\wt \nabla$. Let $\phi \colon M \to \wt M$ be an isometric immersion. 
Let $TM$ be the tangent bundle of $M$ and $NM$ the normal bundle of $\phi$. For a vector bundle $E \to M$, let $\Gamma(M, E)$ denote the space of global sections. The second fundamental form $\rm{II} \colon \Gamma(M, TM) \times \Gamma(M, TM) \to \Gamma(M, NM)$ is defined as follows (cf.\ \cite[Chapter 9]{lee-riemannian}). For a vector field $X$ on an open neighbourhood of the image of $\phi$, let $X|_M \in \Gamma(M, \phi^\ast(T\wt M))$ be the induced section of $\phi^\ast(T\wt M)$ over $M$, and let $X|_M^\perp \in \Gamma(M, NM)$ be the projection of the latter onto the normal bundle. Let $X,Y \in \Gamma(M,TM) \subset \Gamma(M, \phi^\ast(T\wt M))$ be vector fields on $M$, and let $X_W, Y_W \in \Gamma(W, T \wt M)$ be arbitrary extensions of $X,Y$ over an open neighbourhood $\phi(M) \subset W \subset \wt M$. Consider $\wt \nabla^W_{X_W}(Y_W) \in \Gamma(W, T \wt M)$, where $\wt \nabla^W$ is the restriction of $\wt \nabla$ to $W \subset \wt M$. We define $\rm{II}(X,Y) \coloneqq \wt \nabla^W_{X_W}(Y_W)|_M^\perp \in \Gamma(M, NM)$. 

Let $\phi \colon M \to \wt M$ be an isometric immersion. Recall that $\phi$ is a \emph{totally geodesic immersion} if $\phi$ carries the geodesics on $M$ to geodesics on $\wt M$. By \cite[Proposition 8.12]{lee-riemannian}, $\phi$ is a totally geodesic immersion if and only if the second fundamental form $\rm{II} \colon \Gamma(M, TM) \times \Gamma(M, TM) \to \Gamma(M, NM)$ of $\phi$ vanishes identically. 

\begin{lemma} \label{lemma:second-fundamental-lemma}
Let $\phi \colon M \to \wt M$ be an isometric immersion of Riemannian manifolds. Assume that for each $p \in M$, there exist open neighbourhoods $p \in U \subset M$ and $\phi(p) \in V \subset \wt M$ such that $\phi(U) \subset V$ and such that the induced map $\phi_{U,V}  \colon U \to V$ is a totally geodesic immersion. Then $\phi$ is a totally geodesic immersion. 
\end{lemma}
\begin{proof}
By the above, we need to prove that $\rm{II}(X,Y) =0 \in \Gamma(M, NM)$ for each $X,Y \in \Gamma(M, TM)$. Fix vector fields $X$ and $Y$ on $M$, and let $p \in M$. 
Let $p \in U \subset M$ and $\phi(p) \in V \subset \wt M$ be open neighbourhoods such that $\phi(U) \subset V$ and such that the induced map $\phi_{U,V} \colon U \to V$ is a totally geodesic immersion. Let $X_U$ and $Y_U \in \Gamma(U, TU)$ be the restrictions of $X$ and $Y$. 
The second fundamental form $\rm{II}_U \colon \Gamma(U, TU) \times \Gamma(U, TU) \to \Gamma(U, NU)$ of the map $\phi_{U,V} \colon U \to V$ vanishes because $\phi_{U,V}$ is a totally geodesic immersion. Since $\rm{II}_U(X_U, Y_U) = \rm{II}(X,Y)|_U \in \Gamma(U, NU)$, we get $\rm{II}(X,Y)|_U = 0$, and the lemma follows. 
\end{proof}

\begin{lemma} \label{lemma:intermediate}
Let $M = \RRH^{n_1}$ and $N = \RRH^{n_2}$ for some $n_2\geq n_1 \geq 1$. Consider a smooth map $\phi \colon M \to N$. Assume that for each $x \in M$, there exist open neighbourhoods $x \in U \subset M$ and $y \coloneqq \phi(x) \in V \subset N$ such that $\phi(U) \subset V$ and such that the induced map $\phi_{U,V} \colon U \to V$ is a totally geodesic isometric immersion.  Then $\phi$ is a totally geodesic closed isometric embedding. 
\end{lemma}

\begin{proof}
It is clear that $\phi$ is an isometric immersion. By Lemma \ref{lemma:second-fundamental-lemma}, the map $\phi$ is therefore a totally geodesic immersion.
We prove that $\phi$ is topologically a closed embedding. Let $x \in \RR H^{n_1}$. It suffices to show that $\phi$ commutes with the differential $d\phi_x \colon T_x\RRH^{n_1} \to T_{\phi(x)}\RRH^{n_1}$ and the two exponentials $\exp_x \colon T_x \RR H^{n_1}\xrightarrow{\sim} \RR H^{n_1}$ and $\exp_{\phi(x)} \colon T_{\phi(x)} \RR H^{n_2} \xrightarrow{\sim} \RR H^{n_2}$. Indeed, in that case, $\phi = \exp_{\phi(x)} \circ d\phi_x \circ \exp_x^{-1}$, which is clearly topologically a closed embedding. 

To commute that $\phi$ commutes with $d\phi_x$ and the two exponentials $\exp_x$ and $\exp_{\phi(x)}$, let $v \in T_x \RR H^{n_1}$. Then $\exp_x(v) = \gamma_v(1)$, where $\gamma_v \colon \RR \to \RR H^{n_1}$ is the unique geodesic defined on $\RR$ with $\gamma_v(0) = x$ and $\gamma_v'(0) = v$, where $\gamma'_v(0) = (d\gamma_v)_0(d/dt) \in T_x\RR H^{n_1}$ is the velocity of $\gamma_v$ at $0$. We must show that 
$
\phi \left( \exp_x(v) \right) = \exp_{\phi(x)}(d\phi_x(v)). 
$
Let $w = d \phi_x(v) \in T_{\phi(x)} \RR H^{n_2}$ and let $\mu_w \colon \RR \to \RR H^{n_2}$ be the unique geodesic defined on $\RR$ such that $\mu_w(0) = \phi(x)$ and $\mu'_w(0) = w$. We get $\exp_{\phi(x)}(w) = \mu_w(1)$, hence it suffices to show that
$
\phi(\gamma_v(1)) = \mu_w(1). 
$ Define a map $\tilde \mu \colon \RR \to \RR H^{n_2}$ as $\mu' \coloneqq \phi \circ \gamma_v$. Then $\tilde \mu$ is a geodesic, since $\phi$ is a totally geodesic immersion. Moreover, $\tilde \mu$ is defined on $\RR$, and we have
\[
\tilde \mu(0) = \phi(\gamma_v(0) = \phi(x), \quad \quad (\tilde \mu)'(0) = (\phi \circ \gamma_v)'(0) = d\phi_x(\gamma_v'(0)) = d\phi_x(v) = w. 
\]
By the uniqueness of $\mu_w$ as a geodesic defined on $\RR$ with $\mu_w(0) = \phi(x)$ and $\mu_w'(0) = w$, we get $\tilde \mu = \mu_w$. Consequently,
$
\phi(\exp_x(v)) = \phi(\gamma_v(1))  = \mu_w(1) = \exp_{\phi(x)}(d\phi_x(v))
$
as wanted. This shows that $\phi(\gamma_v(1)) = \mu_w(1)$, finishing the proof of the lemma. 
\end{proof}

\begin{definition}[cf.\ \S 2.4.1 in \cite{subspacestabilizers}] \label{def:interestingly}
Let $\Gamma_1 \subset \PO(n_1,1)$ and $\Gamma_2 \subset \PO(n_2,1)$ be lattices, where $n_2 \geq n_1 \geq 1$ are integers. Let $$i \colon \Gamma_1\setminus \RR H^{n_1} \longrightarrow \Gamma_2 \setminus \RR H^{n_2}$$ be a morphism of smooth orbifolds. We say that $i$ is a \emph{totally geodesic immersion} if for any of the lifts $\widetilde i \colon \RR H^{n_1} \to \RR H^{n_2}$ of $i$, there exists a totally geodesic subspace $U \subset \RR H^{n_2}$ such that $\widetilde i$ factors through an isometry $\RR H^{n_1} \xrightarrow{\sim} U \subset \RR H^{n_2}$. If the map $i$ is moreover an embedding, then we call it a \emph{totally geodesic embedding}.  
\end{definition}

\begin{lemma} \label{direct}
Let $n_2 \geq n_1 \geq 1$ be integers. 
Let $\Gamma_1 \subset \tn{PO}(n,1)$ and $\Gamma_2 \subset \PO(n_2,1)$ be lattices. Let 
\[
i \colon \Gamma_1 \setminus \RR H^{n_1} \longrightarrow \Gamma_2 \setminus \RR H^{n_2}
\]
be a morphism of hyperbolic orbifolds. Assume that for each $\bar x \in \Gamma_1 \setminus \RR H^{n_1}$, there exist connected open neighbourhoods $\bar x \in V_1 \subset \Gamma_1 \setminus \RR H^{n_1}$ and $\bar y = i(\bar x) \in V_2  \subset \Gamma_2 \setminus \RR H^{n_2}$ with $i(V_1) \subset V_2$, such that $V_i \cong U_i / \Gamma_{i,x}$ 
for a connected open subset $U_i \subset \RR H^{n_i}$ and a lift $x_i \in U_i$ of the point $\bar x_i$ $(i = 1,2)$, and such that the induced map $i \colon V_1 \to V_2$ lifts to a totally geodesic immersion $j \colon U_1 \to U_2$ which is equivariant for a homomorphism 
$$
\PO(n_1,1)\supset \Gamma_{1,x} \longrightarrow \Gamma_{2,y} \subset \PO(n_2,1). 
$$
Then any lift $\wt i\colon  \RR H^{n_1} \to \RR H^{n_2}$ of $i$ is a totally geodesic closed immersion. In particular, the map $i$ 
is a totally geodesic immersion of hyperbolic orbifolds. 
\end{lemma}
\begin{proof}
Let $M_1 = \Gamma_1 \setminus \RR H^{n_1}$ and $M_2 = \Gamma_2 \setminus \RR H^{n_2}$.  
Let $\wt i\colon  \RR H^{n_1} \to \RR H^{n_2}$ be a lift of $i \colon M_1 \to M_2$. We need to show that $\wt i$ is a totally geodesic closed isometric embedding. To do this, let $x_1 \in \RRH^{n_1}$ and $x_2 = \iota(x_1) \in \RRH^{n_2}$. Let $x_1 \in U_1 \subset \RR H^{n_1}$ be a connected $\Gamma_{1,x_1}$-stable open neighbourhood with  $V_1 = U_1/\Gamma_{1,x} \subset M_1$, let $x_2 \in U_2 \subset \RRH^{n_2}$ be a connected $\Gamma_{2,x_2}$-stable open neighbourhood with $V_2 = U_2/\Gamma_{2,x} \subset M_2$, such that $\wt i(U_1) \subset U_2$. By Lemma \ref{lemma:intermediate}, it suffices to show that the map $\wt i_{U_1,U_2} \colon U_1 \to U_2$ induced by $\wt i$ is a totally geodesic immersion in a neighbourhood of $x_1 \in U_1$. 

Let $\rho \colon \Gamma_1 \to \Gamma_2$ be the homomorphism attached to the morphism of orbifolds $i \colon M_1 \to M_2$ and the point $x_1 \in \RRH^{n_1}$. Note that $\wt i \colon \RRH^{n_1} \to \RRH^{n_1}$ is equivariant with respect to $\rho \colon \Gamma_1 \to \Gamma_2$, and that the induced map $M_1 \to M_2$ coincides with the map $i$. Moreover, $\rho(\Gamma_{1,x_1}) \subset \Gamma_{2,x_2}$, the map $\wt i_{U_1, U_2} \colon U_1 \to U_2$ is equivariant for the map $\rho \colon \Gamma_{1,x_1} \to \Gamma_{2,x_2}$, and the map $V_1 \to V_2$ induced by $\wt i_{U_1, U_2} \colon U_1 \to U_2$ and $\rho \colon \Gamma_{1,x_1} \to \Gamma_{2,x_2}$ coincides with the restriction of the map $i \colon M_1 \to M_2$ to $V_1$. 

Let $\bar x_1 \in M_1$ and $\bar x_2 \in M_2$ be the images of $x_1$ and $x_2$. By assumption, there exist connected open neighbourhoods $\bar x_1 \in V_1'  \subset M_1$ and $\bar x_2 = i(\bar x_1) \in V_2' \subset M_2$ with $i(V_1') \subset V_2'$, such that $V_i' \cong U_i' / \Gamma_{i,x}$ for connected opens $U_i' \subset \RR H^{n_i}$ and lifts $x_i' \in U_i'$ of $\bar x_i$, and such that $i \colon V_1' \to V_2'$ lifts to a totally geodesic immersion $j \colon U_1' \to U_2'$ with $j(x_1') = x_2'$ which is equivariant for a homomorphism 
$
\Gamma_{1,x_1} \to \Gamma_{2,x_2}$. 

Let $W_i \subset M_i$ be a connected open neighbourhood of $\bar x_i$ such that $W_i \subset V_i \cap V_i'$ and $i(W_1) \subset W_2$. Up to replacing $V_i$ and $V_i'$ by $W_i$, and $U_i$ (resp.\ $U_i'$) by the connected component of the inverse image of $W_i$ containing $x_i$ (resp.\ $x_i'$), we may assume that $V_i = V_i'$ for $i = 1,2$. Then, for $i = 1,2$, we let 
\[
x_i \in U_{\epsilon_i}(x_i) \subset U_i, \quad \quad x_i' \in U_{\epsilon_i}(x_i') \subset U_i'
\]
be an open geodesic ball of radius $\epsilon_i>0$ 
such that we have $j(U_{\epsilon_1}(x_1')) \subset U_{\epsilon_2}(x_2')$ and $\wt i_{U_1, U_2}(U_{\epsilon_1}(x_1)) \subset U_{\epsilon_2}(x_2))$. In particular, by the definition of geodesic ball (see \cite[Chapter 6, page 158]{lee-riemannian}), $U_{\epsilon_i}(x_i)$ and $U_{\epsilon_i}(x_i') $ are the respective images of open balls $B_{\epsilon_i}(0) \subset T_{x_i}\RRH^{n_i}$ and $B_{\epsilon_i}'(0) \subset T_{x_i'}\RRH^{n_i}$ under the exponential maps $T_{x_i}\RRH^{n_i} \xrightarrow{\sim} \RRH^{n_i}$ and $T_{x_i'}\RRH^{n_i} \xrightarrow{\sim} \RRH^{n_i}$. 
Now $B_{\epsilon_i}(0) \subset T_{x_i}\RRH^{n_i}$ and $B_{\epsilon_i}'(0) \subset T_{x_i'}\RRH^{n_i}$ are open balls in a normed space, and therefore simply connected. In particular, the open geodesic balls $U_{\epsilon_i}(x_i) \subset U_i$ and $U_{\epsilon_i}(x_i') \subset U_i'$ are simply connected for each $i \in \set{1,2}$. 

Since $U_{\epsilon_i}(x_i)$ and $U_{\epsilon_i}(x_i')$ are geodesic balls, they are metric balls of the same radius, see \cite[Corollary 6.13]{lee-riemannian}. Therefore, since $\Gamma_i$ and $\Gamma_i'$ act by isometries and fix $x_i$ and $x_i'$ respectively, these groups preserve these balls. Consequently, up to replacing $U_i$ and $U_i'$ by $U_\epsilon(x_i)$ and $U_\epsilon(x_i')$, we may assume that $U_i = U_\epsilon(x_i)$ and $U_i' = U_\epsilon(x_i')$. In particular, the $U_i$ are simply connected, and by \cite[Lemma 2.2]{lange-orbifolds-metric}, for $i = 1,2$, there exists an isometry 
\begin{align*}
g_i \colon U_i \xrightarrow{\sim} U_i'   \quad &\text{such that} \quad   \Gamma_{i, x_i'} = g_i \circ \Gamma_{i,x_i} \circ g_i^{-1}  \subset g_i \circ \rm{Isom}(U_i) \circ g_i^{-1} = \rm{Isom}(U_i') \\
& \text{and }\quad\quad\quad \;g_i(x_i) = x_i'. 
\end{align*}
The isometry $g_i \colon U_i \xrightarrow{\sim} U_i'$ is the restriction of an isometry $g_i \colon \RRH^{n_i} \xrightarrow{\sim} \RRH^{n_i}$, i.e., $g_i$ is the restriction of an element $g_i \in \PO(n_i,1) = \rm{Isom}(\RR H^{n_i})$ to $U_i \subset \RRH^{n_i}$. 
To prove that $\wt i_{U_1,U_2} \colon U_1 \to U_2$ is a totally geodesic immersion in a neighbourhood of $x_1 \in U_1$, we may thus assume that $U_i = U_i'$ and $x_i = x_i'$ for $i = 1,2$. 
Then $\wt i_{U_1, U_2}$ and $j$ define two lifts 
$
U_1 \rightrightarrows U_2
$
of the map $i \colon V_1 \to V_2$. Since $\wt i_{U_1, U_2}(x_1) = x_2 = j(x_1)$, we have $\wt i_{U_1, U_2} = j$ because $U_1$ and $U_2$ are simply connected. The map $j \colon U_1 \to U_2$ is a totally geodesic immersion, hence the same holds for $\wt i_{U_1, U_2}$. 
\end{proof}

To prove that $\Gamma_{\eis}^n \subset \PO(n,1)$ is non-arithmetic for each $n \geq 2$, we would like to reduce to the case $n =2$. A key ingredient in this reduction step is the following theorem. 

\begin{theorem}[Bergeron--Clozel] \label{th:bergeron}
Let $n_2 \geq n_1 \geq 2$ be integers. Let $\Gamma_i \subset \PO(n_i,1)$ be a lattice for $i = 1,2$, and let $$\Gamma_1 \setminus \RR H^{n_1} \longrightarrow \Gamma_2 \setminus \RR H^{n_2}$$ be a totally geodesic immersion (see Definition \ref{def:interestingly}). If the lattice $\Gamma_2 \subset \rm{PO}(n_2,1)$ is arithmetic, then the lattice $\Gamma_1 \subset \rm{PO}(n_1,1)$ is arithmetic as well. 
\end{theorem}
\begin{proof}
See \cite[Proposition 15.2.2]{bergeron-clozel} (compare also \cite[Theorem 1.7]{subspacestabilizers}). 
\end{proof}

\section{Standard hermitian lattices} \label{sec:standard} 

The goal of this section is to study particular hermitian lattices $\mr L = (K, \Lambda)$, where the underlying CM field $K$ is of the form $\QQ(\sqrt{-p})$ if $p>2 $ is prime, or of the form $\QQ(\zeta_p)$ if $p>3$ is a prime number such that the class number of $\QQ(\zeta_p)$ is odd. 
This analysis will allow us in the next two sections to prove Theorems \ref{theorem:introduction:non-arithmeticity} and \ref{theorem:non-arithmeticity:p-arbitrary}.

\begin{lemma} \label{lemma:narkiewicz}
Let $p>2$ be a prime number. Assume that the class number of $\QQ(\zeta_p)$ is odd. Then, there exists a unit $\lambda$ in the ring of integers of the totally real number field $ \QQ(\zeta_p+\zeta_p^{-1})$ that satisfies the following conditions:
\begin{enumerate}
\item \label{item:narkiewicz:1}One has that $\tau(\lambda) > 0$ for some $\tau \in \Hom(\QQ(\zeta_p+\zeta_p^{-1}),\RR)$ and $\varphi(\lambda)<0$ for each $\varphi \neq \tau \in \Hom(\QQ(\zeta_p+\zeta_p^{-1}),\RR)$. 
\item \label{item:narkiewicz:2} One has $\lambda \not \equiv -1 \bmod (1-\zeta)$. 
\end{enumerate}
Moreover, if $p=3$ one may take $\lambda = -1$, and if $p = 5$ one may take $\lambda = \zeta_5+\zeta_5^{-1}$. 
\end{lemma}
\begin{proof}
Let $K = \QQ(\zeta_p)$ and $F = \QQ(\zeta_p+\zeta_p^{-1}) \subset K$. Assume first that $p = 3$. Then $F = \QQ$ and $\OO_F = \ZZ$, and one may put $\lambda = -1 \in \OO_F^\ast$. 

Next, assume that $p\geq 5$. 
For an integer $a$ with $1 < a < \frac{1}{2} \cdot p$, define
\begin{align} \label{lambda-a}
\lambda_a \coloneqq \zeta_p^{(1-a)/2} \cdot \frac{\zeta_p^a-1}{\zeta_p-1} \in \OO_F^\ast.  
\end{align}
Let $C^+_p \subset \OO_F^\ast$ be the group of units in $\OO_F$ generated by $-1$ and the elements $\lambda_a$ for $1 < a< \frac{1}{2} \cdot p$. 
By \cite[Lemma 8.1 and Theorem 2]{washington}, the index of $C^+_p $ in $\OO_F^\ast$ is finite, and we have $h^+ = [\OO_F^\ast \colon C^+_p]$, where $h^+$ denotes the class number of $F$. Moreover, by \cite[Theorem 4.10]{washington}, we have that $h^+$ divides the class number $h$ of $K = \QQ(\zeta_p)$. Since $h$ is odd by assumption, we get that $h^+$ is odd. 

Let $\varphi_1, \dotsc, \varphi_g$ be the embeddings of $F$ into $\RR$. Since the class number $h$ of $\QQ(\zeta_p)$ is odd, the map $$\rm{sgn} \colon \OO_F^\ast \longrightarrow \prod_{i = 1}^g  \set{\pm 1}, \quad \quad u \mapsto (\rm{sign}(\varphi_1(u)), \dotsc, \rm{sign}(\varphi_g(u)) ) $$ is surjective by a classical result of Kummer, see e.g.\ \cite[page 126]{narkiewicz} or \cite[Corollary 2.2]{edgar}. Hence, for each $\varepsilon = (\varepsilon_1, \dotsc, \varepsilon_g) \in \set{\pm 1}^g$ there exists $u \in \OO_F^\ast$ such that $\rm{sgn}(u) = \varepsilon$. Then $u^{h^+} \in C_p^+$ since $h^+ = [\OO_F^\ast \colon C^+_p]$, and $\rm{sgn}(u^{h^+}) = \rm{sgn}(u)^{h^+} = \rm{sgn}(u)$ since $h^+$ is odd. This shows that the homomorphism $\rm{sgn} \colon C_p^+ \to \set{\pm 1}^g$ is also surjective. In particular, if we let $T \subset \OO_F^\ast$ be the subgroup of totally positive elements, and $T^+_p = T \cap C^+_p$, then
$
C^+_p/T^+_p \xrightarrow{\sim} \OO_F^\ast/T \xrightarrow{\sim} \prod_{i = 1}^g \set{\pm 1}. 
$

Define generators $\mu_1, \dotsc, \mu_g \in C_p^+$ with $\mu_1 = -1$, such that for each $i \in \set{2, \dotsc, g-1}$, one has $\rm{sgn}(\mu_i) = (\varepsilon_1, \dotsc, \varepsilon_g) \in \set{\pm 1}^g$ with $\varepsilon_i = 1$ and $\varepsilon_j = -1$ for each $j \neq i$. We claim that there exists $i \in \set{2, \dotsc, g}$ such that $\mu_i \not \equiv -1 \bmod (1-\zeta_p)$. Indeed, assume for a contradiction that $\mu_i \equiv -1 \bmod (1-\zeta_p)$ for each $i \in \set{2, \dotsc, g}$. Then the reduction homomorphism
\[
C_p^+ \longrightarrow \left(\OO_K/(1-\zeta_p)\OO_K\right)^\ast \xrightarrow{\sim} (\bb F_p)^\ast
\]
has image equal to $\set{\pm 1} \subset \bb F_p^\ast$. But for each $a$ with $1 < a < \frac{p}{2}$, we have
\[
\lambda_a \equiv a \bmod (\zeta_p-1),
\]
and since $p \geq 5$, we have $2 \not \equiv -1 \bmod p$. Thus, the element $\lambda_2 \in C_p^+$ satisfies $\lambda_2 \equiv 2 \not \equiv -1 \bmod (1-\zeta_p)$, which is a contradiction as wanted. 

Let $i \in \set{2, \dotsc, g}$ such that $\mu_i \not \equiv -1 \bmod (1-\zeta)$, and put $\lambda \coloneqq \mu_i$. Then $\lambda$ satisfies conditions 1 and 2 in the statement of the lemma. 
If $p = 5$, then $\lambda = \zeta_5+\zeta_5^{-1}$ satisfies conditions 1 and 2 in the statement of the lemma, and we are done. 
\end{proof}

\begin{definition} \label{def:standard}
Let $n \geq 1$ be an integer. 
\begin{enumerate}
\item For a prime number $p\geq 3$, define the hermitian lattice  $\mr L_{\rm{qd}}^n(p)$ as 
 $$\mr L_{\rm{qd}}^n(p) = (\QQ(\sqrt{-p}),\Lambda_{\rm{qd}}^n(p) ), \quad \quad \Lambda_{\rm{qd}}^n(p) = \left((\OO_{\QQ(\sqrt{-p})})^{n+1}, h\right),$$
where $h$ is the hermitian form with $h(x,y) = -x_0 \bar y_0 + \cdots + x_n \bar y_n$. 
\item 
For a prime number $p \geq 5$ such that the class number of $\QQ(\zeta_p)$ is odd, 
and $\lambda$ a unit in the ring of integers of $\QQ(\zeta_p+\zeta_p^{-1})$ that satisfies Conditions \ref{item:narkiewicz:1} and \ref{item:narkiewicz:2} in Lemma \ref{lemma:narkiewicz}, define 
$$
\mr L_{\zeta_p}^n(\lambda) = (\QQ(\zeta_{p}), \Lambda_{\zeta_p}^n(\lambda)),\quad \quad \Lambda_{\zeta_p}^n(\lambda)= (\OO_K^{n+1}, h_\lambda), 
$$
with $h_\lambda$ the hermitian form defined as $h_\lambda(x,y) = - \lambda \cdot x_0 \bar y_0 + \cdots + x_n \bar y_n$. 
\end{enumerate}
Any hermitian lattice $\mr L$ of the form $\mr L_{\rm{qd}}^n(p)$ or $\mr L_{\zeta_p}^n(\lambda) $ for some $p, \lambda,n$ is called \emph{standard}. 
\end{definition}
\begin{lemma}
Let $\mr L$ be a standard hermitian lattice. Then $\mr L$ is admissible. 
\end{lemma}
\begin{proof}
See Examples \ref{proposition:discr} and \ref{example:lambda}. 
\end{proof}

Any attachment $M(\mr L)$ (cf.\ Definition \ref{gluedspace}) with $\mr L$ a standard hermitian lattice is called a \emph{standard attachment}. 

Let $n \geq 1$ and let $\mr L = (K, \Lambda)$ be a standard hermitian lattice of rank $n+1$, 
see Definition \ref{def:standard}. 
Define anti-unitary involutions $\alpha_{n,i} \colon \Lambda \to \Lambda$ for $i = 0, 1, \dotsc, n$ as follows:
\begin{align} \label{def:beta}
\begin{split}
&\alpha_{n,i} \colon \Lambda \longrightarrow \Lambda, \\ &\alpha_{n,i}(x_0,  \dotsc, x_n) = \left( \bar x_0, -\bar x_1, \dotsc, 
- \bar x_i, \bar x_{i + 1}, \dotsc, \bar x_n \right). 
\end{split}
\end{align}
As usual, we let $\mu_K \subset \OO_K^\ast$ denote the group of finite units of $\OO_K$. 

\begin{lemma} \label{lemma:conjinv:intro}
Let $n \geq 1$ and let $\mr L = (K, \Lambda)$ be a standard hermitian lattice of rank $n+1$. 
Then the anti-unitary involutions  
$\pm \alpha_{n,i} \colon \Lambda \to \Lambda$ for $i = 0, \dotsc, n$ defined in \eqref{def:beta} are pairwise non $\Aut(\Lambda)$-conjugate. Furthermore, the $\mu_K$-equivalence classes 
$[\alpha_{n,i}]$ $(i = 0, \dotsc, n)$ are pairwise non $\Aut(\Lambda)/\mu_K$-conjugate. 
\end{lemma}
\begin{proof}
We first introduce some notation. 
Define $\mf p$ as the prime ideal generated by $\sqrt{-p} \in \OO_K$ if $K = \QQ(\sqrt{-p})$ and by $1 - \zeta_p$ if $K = \QQ(\zeta_p)$. 
Define $W \coloneqq \Lambda / \mf p \Lambda \cong (\bb F_p)^{n+1}$. 
Consider the bilinear form 
$
q \colon W \times W \to \OO_K/\mf p = \bb F_p$ defined as  $q(x,y) = h(x,y) \bmod \mf p$. 
Define $\ol \lambda \in \OO_K/\mf p = \bb F_p$ as follows: we put $\bar \lambda =1$ if $K = \QQ(\sqrt{-p})$, and define $\bar \lambda$ as the image of $\lambda \in \OO_K$ under the map $\OO_K \to \OO_K/\mf p$ if $\mr L = \mr L_{\zeta_p}^n(\lambda)$. 

For an anti-unitary involution $\alpha \colon \Lambda \to \Lambda$, define $d(\alpha) \coloneqq \dim(W^\alpha)$ and $t(\alpha) \coloneqq \det(q|_{W^\alpha})$. Then $d(\alpha)$ and $t(\alpha)$ are clearly $\Aut(\Lambda)$-conjugacy invariants. Let $i \in \set{0, \dotsc, n}$. For $x = (x_0,0, \dotsc, 0, x_{i + 1}, \dotsc, x_{n}) \in W^{\alpha_i}$, one has $q(x,x) = -\ol \lambda \cdot x_0^2 + x_{i+1}^2 + \cdots + x_{n}^2.$ Similarly, for $x = (0,x_1, \dotsc, x_i, 0, \dotsc, 0) \in W^{-\alpha_i}$, one has $
q(x,x) = x_1^2 + \cdots + x_i^2.$ 
Thus, for $ i = 0, 1, \dotsc, n$, one obtains: 
\begin{align*}\left( d(\alpha_i), t(\alpha_i) \right) = \left( n + 1 -i, -\ol \lambda \right), \quad \quad \left( d(-\alpha_i), t(-\alpha_i) \right) = \left( i, 1 \right).\end{align*} 
Since $\bar \lambda \neq -1$, see Condition \ref{item:narkiewicz:2} in Lemma \ref{lemma:narkiewicz}, this proves the first assertion. 

To prove the second assertion, assume that $[\alpha_{i,n}]$ and $[\alpha_{j,n}]$ are conjugate. Then there exists $g \in \Aut(\Lambda)$ such that $g \alpha_{i,n} g^{-1} = \zeta^b \cdot \alpha_{j,n}$ for some root of unity $\zeta \in \OO_K^\ast$ and some $b \in \ZZ$. If $K = \QQ(\sqrt{-p})$ with $p \geq 5$ then $\OO_K^\ast = \set{\pm 1}$ and we obtain $i=j$ by what has already been proven. If $K = \QQ(\sqrt{-3})$, then $\OO_K^\ast = \langle - \zeta_3 \rangle$ hence again, $i = j$. If $K = \QQ(\zeta_p)$, then $\OO_K^\ast = \langle - \zeta_p \rangle$ so that $i=j$ and the result follows. 
\end{proof} 

\begin{proposition} \label{proposition:identification-of-arithmetic-part}
Let $\mr L = (K, \Lambda)$ be a standard hermitian lattice of rank $n+1\geq 2$. Define $L = \Aut(\Lambda)/\mu_K$, let $\alpha_0, \dotsc, \alpha_n$ be as in \eqref{def:beta}, and for $i \in \set{0, \dotsc, n}$, define $L_i = \rm{Stab}_L(\RR H^n_{\alpha_i})$ and $\Lambda_i = \Lambda^{\alpha_i}$. The inclusions in the diagram
\begin{align} 
\label{diagram:new-try-works:zero}
\xymatrix{
L_i& L_i^I\ar@{_{(}->}[l] \ar@{=}[r] & \PO(\Lambda_i)(\OO_F)^I \ar@{^{(}->}[r] & \PO(\Lambda_i)(\OO_F),
}
\end{align}
see the proof of Theorem \ref{th:crucialthm-finitevolume}, are equalities. 
\end{proposition}
\begin{proof}
Let us first prove that $L_i^I =L_i$. For this, recall that $L_i = N_L(\alpha_i)$ (see Lemma \ref{normalisator}). Thus, to prove that $L_i^I =L_i$, it suffices to show that there are no $g \in \Aut(\Lambda)$ such that $g \circ \alpha_i = - \alpha_i \circ g$. This is true because of Lemma \ref{lemma:conjinv:intro}, which implies that the elements $\alpha_i$ and $-\alpha_i \in \mr A$ are not $\Aut(\Lambda)$-conjugate. 

It remains to prove that $\PO(\Lambda_i)(\OO_F)^I = \PO(\Lambda_i)(\OO_F)$. This comes down to showing that for each $i \in \set{0, \dotsc, n}$, every isometry of the $\OO_F$-lattice $\Lambda_i$ is induced by an isometry of $\Lambda$. To prove this, let $\Lambda_i^\vee \subset \Lambda\otimes_{\OO_K}K$ be the $\OO_F$-lattice dual to $\Lambda_i$. Let $\theta = \sqrt{-p}$ if $K = \QQ(\sqrt{-p})$ and $\theta = \zeta_p-\zeta_p^{-1}$ if $K = \QQ(\zeta_p)$. Then note that, as sublattices of $\Lambda \otimes_{\OO_K}K$, we have: \begin{align*}
\Lambda_i &= \OO_F \oplus \left( \theta \OO_F \right)^{\oplus i} \oplus \OO_F^{\oplus (n-i)}, \\
\Lambda_i^\vee &= \OO_F \oplus \left( \theta^{-1} \OO_F \right)^{\oplus i} \oplus \OO_F^{\oplus (n-i)}, \\
\Lambda_i \cap \theta \Lambda &= \va{\theta}^2\OO_F \oplus 
 \left( \theta \OO_F \right)^{\oplus i} \oplus (\va{\theta}^2\OO_F)^{\oplus (n-i)}  = \va{\theta}^2 \Lambda_i^\vee. 
\end{align*}
Since $\theta^{-1}\left( \Lambda_i \cap \theta \Lambda \right)  = \theta \OO_F \oplus \OO_F^{\oplus i} \oplus (\theta \OO_F)^{\oplus (n-i)}$, one deduces that
\begin{align} \label{align:preserves}
\begin{split}
\Lambda = 
\OO_K \cdot \Lambda_i + \OO_K \cdot  \theta^{-1}\left( \Lambda_i \cap \theta \Lambda \right)   = 
\OO_K \cdot \Lambda_i + \OO_K \cdot  \theta \Lambda_i^\vee. 
\end{split}
\end{align}
Now let $\phi \in \Aut(\Lambda_i)$ be an $\OO_F$-linear isometry of $\Lambda_i$. Then $\phi$ extends to a $K$-linear isometry $\psi$ of $\Lambda \otimes_{\OO_K} K$ which preserves $\OO_K \cdot \Lambda_i$ as well as $\OO_K \cdot \theta \Lambda_i^\vee$. By \eqref{align:preserves}, it follows that $\psi(\Lambda) = \Lambda$, and hence $\psi$ defines an isometry $\psi \colon \Lambda \xrightarrow{\sim} \Lambda$ such that $\psi(\Lambda_i) = \Lambda_i$. This concludes the proof of the proposition. 
\end{proof}


\section{Totally geodesic immersions of standard attachments} \label{sec:nonarithmetic}

We fix an integer $n \geq 3$, and let
\begin{align} \label{align:m-l}
\mr M = (K, M) \quad \quad \text{and} \quad \quad \mr L = (K, \Lambda)
\end{align}
be standard hermitian lattices such that $\mr M$ has rank two and $\mr L$ has rank $n$. In particular, we have either that $M = \Lambda_{\rm{qd}}^2(p)$ and $\Lambda = \Lambda_{\rm{qd}}^n(p)$, or that $M = \Lambda_{\zeta_p}^2(\lambda)$ and $\Lambda = \Lambda_{\zeta_p}^n(\lambda)$ for suitable $p$ and $\lambda$. 
We define anti-unitary involutions $\beta_i \colon M \to M$ for $i = 0,1,2$ and $\alpha_i \colon \Lambda \to \Lambda$ by putting $\beta_i = \alpha_{2,i}$ and $\alpha_i = \alpha_{n,i}$, see \eqref{def:beta}. We let $M(\mr M, \beta_0) \subset M(\mr M)$ (resp.\  $M(\mr L, \alpha_0) \subset M(\mr L)$) be the connected component containing the image of the natural map $\RRH^2_{\beta_0} \to M(\mr M)$ (resp.\ $\RRH^n_{\alpha_0} \to M(\mr L)$). 

The goal of this section is to prove the following theorem, and deduce Theorem \ref{theorem:non-arithmeticity:p-arbitrary} from it. 

\begin{theorem} \label{theorem:totally-geodesic-standard}
Consider the above notation. Assume that each anti-unitary involution of $M$ is $\Aut(M)$-conjugate to one of the $\pm \beta_i$ $(i = 0,1,2)$. Then the following assertions are true. 
\begin{enumerate}
\item \label{a}The orbifold $M(\mr M)$ is connected, hence equal to $M(\mr M, \beta_0)$. 
\item \label{b}There exists a canonical totally geodesic immersion of complete connected hyperbolic orbifolds
$$
\iota \colon M(\mr M) = M(\mr M, \beta_0)\longrightarrow M(\mr L, \alpha_0). 
$$
\item \label{c}
In particular, if the hyperbolic orbifold $M(\mr M)$ is non-arithmetic, then the hyperbolic orbifold $M(\mr L, \alpha_0)$ is non-arithmetic as well. 
\end{enumerate}
\end{theorem}

\begin{remark}
For $M$ as above (i.e.\ equal to $\Lambda_{\rm{qd}}^2(p)$ or $ \Lambda_{\zeta_p}^2(\lambda)$ for suitable $p, \lambda$), one can ask whether 
it is always true that each anti-unitary involution of 
$M$ is $\rm{Aut}(M)$-conjugate to exactly one of the $\pm \beta_i$. This is in any case true if $M = \Lambda_{\rm{qd}}^2(3) = \ZZ[\zeta_3]^{2,1}$, as we will prove below (cf.\ Theorem \ref{useful-theorem}). In a subsequent paper, see \cite{gaay-hyperbolic}, we will prove that this also holds if $M = \Lambda_{\zeta_5}^2(\zeta_5 + \zeta_5^{-1})$. 
\end{remark}

\subsection{Continuous map between standard attachments} 

Continue with the above notation. In particular, we consider the standard hermitian lattices $\mr M$ and $\mr L$, see equation \eqref{align:m-l} above. We put
\begin{align*}
L= \Aut(\Lambda)/\mu_K  \quad \quad \text{ and } \quad \quad 
\Gamma= \Aut(M)/\mu_K,
\end{align*}
where, as usual, $\mu_K \subset \OO_K^\ast$ denotes the torsion-subgroup of $\OO_K^\ast$. 
For $i  \in \set{0,1, \dotsc, n}$, define $\Lambda_i  \coloneqq \Lambda^{\alpha_i}$, and put
\begin{align*}
L_i \coloneqq L_{\alpha_i} = \rm{Stab}_L(\RRH^n_{\alpha_i}) \subset L. 
\end{align*}
For $i \in \set{0, 1,2}$, define $M_i \coloneqq M^{\beta_i}$, and put
\begin{align*}
\Gamma_i \coloneqq \Gamma_{\beta_i} = \rm{Stab}_\Gamma(\RRH^2_{\beta_i}) \subset \Gamma. 
\end{align*}
Consider the canonical embeddings $\Aut(M) \subset \GL_3(\OO_K)$ and $\Aut(\Lambda) \subset \GL_{n+1}(\OO_K)$, and the embedding
\begin{align}\label{jn}
j \colon \Aut(M) \longhookrightarrow \Aut(\Lambda), \quad M \mapsto (M, \Id). 
\end{align}
Observe that $\CCH^n$ is the space of lines in $\CC^{n+1}$ which are negative for the hermitian form $h(x,y) = - \lambda \cdot x_0 \bar y_0 + \cdots + x_n \bar y_n$. For an element $x = (x_0, \dotsc, x_n) \in \CC^{n+1}$ such that $h(x,x) < 0$, we define $[x_0 \colon \dotsc \colon x_n] \in \CCH^n$ as the associated negative line. Consider then the totally geodesic embedding \begin{align*}\nu \colon \CC H^2 &\longhookrightarrow \CC H^n, \\
[x_0\colon x_1\colon x_2] &\mapsto [x_0 \colon x_1 \colon x_2 \colon 0 \colon \dotsc \colon 0]. 
\end{align*}
Observe that $\nu$ commutes with each pair $(\beta_i, \alpha_i)$ for $i = 0,1,2$, hence induces totally geodesic embeddings
\begin{align} \label{align:nu-i}
\nu_i \colon \RRH^2_{\beta_i}  \longhookrightarrow \RRH^n_{\alpha_i}, \quad \quad i = 0, 1,2. 
\end{align}


\begin{lemma} \label{lemma:ol-nu}
For each $i \in \set{0,1,2}$, there is a natural injective homomorphism 
\begin{align} \label{homomorphism:nat}
\Gamma_i\longhookrightarrow L_i
\end{align}
that makes the map \eqref{align:nu-i} equivariant. In particular, for each $i \in \set{0,1,2}$, there is a canonical totally geodesic immersion of hyperbolic orbifolds $\ol{\nu_i}  \colon \Gamma_i\sm \RRH^{2}_{\beta_i} \to L_i \sm \RRH^n_{\alpha_i}$ such that the following diagram commutes:
\[
\xymatrix{
\RRH^{2}_{\beta_i}  \ar@{^{(}->}[r]^{\nu_i} \ar[d] & \RRH^n_{\alpha_i} \ar[d] \\
\Gamma_i\sm \RRH^{2}_{\beta_i} \ar[r]^{\ol{\nu_i}} & L_i \sm \RRH^n_{\alpha_i}.
}
\]
\end{lemma}

\begin{proof}
Let $W = M \otimes_{\OO_K}\CC$ and $V = \Lambda \otimes_{\OO_K} \CC$, where the tensor products are taken with respect to the embedding $\tau \colon K = \QQ(\zeta_p) \hookrightarrow \CC$ with $\tau(\zeta_p) = \zeta_p = e^{2 \pi i/p} \in \CC$. For $i = 0,1,2$, define $W_i = W^{\beta_i}$ and $V_i = V^{\alpha_i}$. 
Consider the natural isomorphisms $\RRH^2_{\beta_i} = \RRH(W_i)$ and $\RRH^n_{\alpha_i} = \RRH(V_i)$, see Lemma \ref{hyperbolic}. With respect to these identifications, the map $\nu_i$ is given by the totally geodesic immersion $\nu_i \colon \RRH(W_i) \hookrightarrow \RRH(V_i)$ induced by the canonical embedding $W_i \hookrightarrow V_i$. As $\dim(W_i) = 3$, the composition $\SO(W_i) \to \rm{O}(W_i) \to \PO(W_i)$ is an isomorphism. Moreover, the composition
\begin{align} \label{align:embedding-phi}
\phi \colon \PO(W_i) = \SO(W_i) \longhookrightarrow \SO(V_i) \longrightarrow \PO(V_i)
\end{align}
is an embedding. We have $L_i = \PO(\Lambda_i)(\OO_F)$ and $\Gamma_i = \PO(M_i)(\OO_F)$ by Lemma \ref{proposition:identification-of-arithmetic-part}, and the map \eqref{align:embedding-phi} restricts to an embedding
\[
\phi \colon \Gamma_i= \PO(M_i)(\OO_F) \longhookrightarrow \PO(\Lambda_i)(\OO_F) = L_i. 
\]
that makes the map $\nu_i \colon \RRH(W_i) \hookrightarrow \RRH(V_i)$ equivariant. The lemma follows. 
\end{proof}

We now continue by assuming the following:

\begin{condition}\label{condition:connected}
Each anti-unitary involution of $M$ is $\Aut(M)$-conjugate to one of the $\pm \beta_i$ $(i = 0,1,2)$. 
\end{condition}

\begin{lemma} \label{connected}
Assume Condition \ref{condition:connected}. Then $M(\mr M)$ is connected. 
\end{lemma}
\begin{proof}
By Condition \ref{condition:connected}, the natural map $\RRH^2_{\beta_0} \sqcup \RRH^2_{\beta_1} \sqcup \RRH^2_{\beta_2} \to M(\mr M)$ is surjective. Moreover, if $\sim$ denotes the equivalence relation on $\RRH^2_{\beta_0} \sqcup \RRH^2_{\beta_1} \sqcup \RRH^2_{\beta_2}$ induced by its inclusion into $\coprod_{\beta \in P\mr B} \RRH^2_{\beta}$, where $P\mr B = \langle - \zeta_p \rangle \sm \mr B$, we have 
\[
\left( \RRH^2_{\beta_0} \coprod \RRH^2_{\beta_1} \coprod \RRH^2_{\beta_2} \right)/_\sim = \RRH^2_{\beta_0} \bigcup \RRH^2_{\beta_1} \bigcup \RRH^2_{\beta_2} \subset \CCH^2,
\]
which follows from Lemma \ref{intersection}. Consequently, we have a natural surjective map $\cup_{i = 0}^2 \RRH^2_{\beta_i} \to M(\mr M)$; since $\cup_{i = 0}^2 \RRH^2_{\beta_i}$ is connected, the lemma follows.  
\end{proof}

Let $\mr H_2 \subset \CCH^2$ (resp.\ $\mr H_n \subset \CCH^n$) be the hyperplane arrangement defined by the norm one vectors in $M$ (resp.\ $\Lambda$). Recall that there are natural open embeddings $\sqcup_{i = 0}^2 \Gamma_i \sm (\RRH^2_{\beta_i} - \mr H_2 ) \hookrightarrow M(\mr M)$ and $\sqcup_{i = 0}^n L_i \sm (\RRH^n_{\alpha_i} - \mr H_n ) \hookrightarrow M(\mr L)$, see Theorem \ref{theorem:introduction:uniformization} (and compare Lemma \ref{lemma:openembeddingtopological}). 

\begin{proposition} \label{proposition:canonical-map}
Assume Condition \ref{condition:connected}. 
There is a canonical map of topological spaces $\iota \colon M(\mr M) = M(\mr M, \beta_0) \to M(\mr L, \alpha_0)$ that makes the diagram
\begin{align} 
\label{diagram:top-com}
\begin{split}
\xymatrix{
M(\mr M, \beta_0) \ar[r]^\iota & M(\mr L, \alpha_0) \\
\coprod_{i = 0}^2 \Gamma_i \sm \left(\RRH^2_{\beta_i} - \mr H_2 \right) \ar[r]^{\ol \nu} \ar@{^{(}->}[u] & \coprod_{i = 0}^n L_i \sm\left( \RRH^n_{\alpha_i} - \mr H_n \right) \ar@{^{(}->}[u]
}
\end{split}
\end{align}
commute, where $\ol \nu$ is the coproduct of the maps $\ol{\nu_i}$ of Lemma \ref{lemma:ol-nu}. 
\end{proposition}
\begin{proof}
We have $M(\mr M) = M(\mr M, \beta_0)$ by Lemma \ref{connected}. 
Let $\ol f \in M(\mr M)$. We can lift $\ol f$ to an element $(x, \beta_i) \in \RRH^2_{\beta_0} \sqcup \RRH^2_{\beta_1} \sqcup \RRH^2_{\beta_2}$ for some $i \in \set{0,1,2}$. This yields an element $(\nu_i(x), \alpha_i) \in \RRH^2_{\alpha_0} \sqcup \RRH^2_{\alpha_1} \sqcup \RRH^2_{\alpha_2}$. We define $\iota(\ol f)$ as the image $[(\nu_i(x), \alpha_i)] \in M(\mr L)$ of the element $(\nu_i(x), \alpha_i)$ under the composition of maps $\RRH^2_{\alpha_0} \sqcup \RRH^2_{\alpha_1} \sqcup \RRH^2_{\alpha_2} \to Y(\mr L) \to M(\mr L)$. Notice that $[(\nu_i(x), \alpha_i)] \in M(\mr L, \alpha_0)$. 

We check that this gives a well-defined defined map. To prove this, let  $(y, \beta_j) \in \RRH^2_{\beta_0} \sqcup \RRH^2_{\beta_1} \sqcup \RRH^2_{\beta_2}$ be another lift of $\ol f$. Assume first that $i = j$. In that case, there exists an element $g \in \Gamma_i$ such that $g \cdot x = y$. Consider the homomorphism $\phi \colon \Gamma_i \to L_i$ defined in \eqref{homomorphism:nat}. By Lemma \ref{lemma:ol-nu}, we get $\phi(g) \cdot \nu_i(x) = \nu_i(g \cdot x) = \nu_i(y)$. Hence $[(\nu_i(x), \alpha_i)]  = [(\nu_i(y), \alpha_i)] \in M(\mr L)$, proving what we want. Next, assume that $i \neq j$. We claim that $[(\nu_i(x), \alpha_i)]  = [(\nu_i(y), \alpha_j)]  \in M(\mr L)$. Observe that there exists $g \in \Gamma_i$ such that $(g \cdot x, \beta_i) \sim (y, \beta_j)$. Since $[(\nu_i(x), \alpha_i)]  = [(\phi(x) \cdot \nu_i(x), \alpha_i)]= [(\nu_i(g \cdot x), \alpha_i)]$, we may assume that $(x, \beta_i) \sim (y, \beta_j)$. This implies that $x = y \in \RRH^2_{\beta_i} \cap \RRH^2_{\beta_j}$ and hence that 
\[
\nu_i(x) = \nu_i(y) \in \nu \left( \RRH^2_{\beta_i} \cap \RRH^2_{\beta_j}\right) = \nu \left( \RRH^2_{\beta_i}\right) \cap \nu \left( \RRH^2_{\beta_j}\right) \subset \RRH^n_{\alpha_i} \cap \RRH^n_{\alpha_j}. 
\]
Consequently, $(\nu_i(x), \alpha_i) \sim (\nu_j(y), \alpha_j)$, hence $[(\nu_i(x), \alpha_i)]  = [(\nu_i(y), \alpha_j)]  \in M(\mr L)$ as wanted. Thus $\iota \colon M(\mr M) \to M(\mr L, \alpha_0)$ is a well-defined map. Moreover, $\iota$ is continuous since $\iota$ fits inside the commutative diagram 
\[
\xymatrix{
\coprod_{i = 0}^2\RRH^2_{\beta_i}  \ar@{->>}[d] \ar[r] & \coprod_{\alpha \in P\mr A} \RRH^n_{\alpha} \ar@{->>}[d] \\
 M(\mr M) \ar[r]^\iota & M(\mr L)
}
\]
in which the vertical arrows are quotient maps and the upper horizontal arrow is continuous. Diagram \eqref{diagram:top-com} clearly commutes and the proof is finished. 
\end{proof}

\subsection{Orbifold map between standard attachments} \label{sec:orbifoldmap} 
\begin{proposition} \label{prop:fund-funda-fund}
Assume Condition \ref{condition:connected}. 
The map $\iota \colon M(\mr M) \to M(\mr L, \alpha_0)$ defined in Proposition \ref{proposition:canonical-map} is a morphism of real analytic orbifolds. 
\end{proposition}

\begin{proof}
 Let $\bar f \in M(\mr L_\eis^2)$. Our goal is to show that $\iota \colon M(\mr M) \to M(\mr L, \alpha_0)$ is a morphism of real analytic orbifolds at the point $\bar f$. For this, we lift $\bar f$ to an element $(x, \beta_{i_0}) \in \RRH^2_{\beta_0} \sqcup \RRH^2_{\beta_1} \sqcup \RRH^2_{\beta_2}$ for some $i_0 \in \set{0,1,2}$, which induces an element $(\nu_{i_0}(x), \alpha_{i_0}) \in \RRH^n_{\alpha_0} \sqcup \RRH^n_{\alpha_1} \sqcup \RRH^n_{\alpha_2}$ with image $[(\nu_{i_0}(x), \alpha_{i_0})] = \iota(\ol f)\in M(\mr L, \alpha_0)$. 
 
We have $\nu(x) \in \RRH^n_{\alpha_{i_0}}$. Let $f \in Y(\mr M)$ be the image of $x$ under the natural map $\RRH^2_{\alpha_{i_0}} \to Y(\mr L)$. Similarly, let $g \in Y(\mr L)$ and $\bar g \in M(\mr L)$ be the images of $\nu(x) \in \RRH^n_{\alpha_{i_0}}$ under the natural maps $\RRH^n_{\alpha_{i_0}} \to Y(\mr L)$ and $\RRH^n_{\alpha_{i_0}} \to M(\mr L)$. 

\begin{lemma} \label{lemma:canonical-hom}
Let $A_f \coloneqq \Stab_{\Gamma}(f) \subset \Gamma = \Aut(M)/\mu_K$ and $A_g \coloneqq \Stab_L(g)  \subset L = \Aut(\Lambda)/\mu_K$. Let $B_f \subset A_f$ and $B_g \subset A_g$ be the subgroups generated by reflections in the real roots of $f$ and $g$ respectively. Then there is a canonical group homomorphism
\begin{align} \label{well-defined-embedding}
\pi \colon A_f/B_f = \Stab_{\Gamma}(f)/B_f \longrightarrow \Stab_{L}(g)/B_{g} = A_{g} / B_{g}. 
\end{align}
\end{lemma}
\begin{proof}[Proof of Lemma \ref{lemma:canonical-hom}]
To prove this, it suffices to show that there is a canonical group homomorphism $\pi' \colon A_f \to A_g/B_g$ that sends $B_f \subset A_f$ to zero. Observe that $\Stab_{\Gamma}(f) = \Stab_{\Aut(M)}(f)/\mu_K$ and $\Stab_{L}(g) = \Stab_{\Aut(\Lambda)}(g)/\mu_K$. Let $\phi \in A_f$ and choose a representative $\tilde \phi \in \Stab_{\Aut(M)}(f)$. Let $M^\perp \subset \Lambda$ be the orthogonal complement of $M \subset \Lambda$, and observe that $\Lambda = M \oplus M^\perp$ as hermitian $\OO_K$-modules. 
The element $(\tilde \phi, \Id) \in \Aut(M) \oplus \Aut(M^\perp) \subset  \Aut(\Lambda)$ satisfies $(\tilde \phi, \Id) \in \Stab_{\Aut(\Lambda)}(g) \subset \Aut(\Lambda)$,  
and we define
$$
\pi'(\phi) \coloneqq [(\tilde \phi, \Id) ] \in B_g \sm \Stab_{\Aut(\Lambda)}(g) / \mu_K = B_g \sm A_g. 
$$
For another choice of lift $\zeta^i \cdot \tilde \phi \in \Stab_{\Aut(M)}(f)$ of $\phi$, we get 
\[
[(\zeta^i \cdot \tilde \phi, \Id)] = [(\tilde \phi, \zeta^{-i})] = [(\tilde \phi, \Id) ] \quad \quad \text{since} \quad \quad (\Id, \zeta^{-i}) \in B_g \subset \Aut(\Lambda).
\]
Thus, $\pi'(\phi)$ does not depend on the chosen lift of $\phi$, and we get a group homomorphism
$
\pi' \colon A_f \to A_g/B_g$; 
the latter sends $B_g$ to zero, which allows us to define \eqref{well-defined-embedding} as the canonical map $A_f/B_f \to A_g/B_g$ through which $\pi'$ factors. 
\end{proof}
 
 \begin{lemma} \label{lemma:localcoordinates:infamily:new}
Consider the points $x \in \CC H^2$ and $\nu(x) \in \CC H^n$. 
    There are isometries $
\kappa_2 \colon    \CC H^2 \xrightarrow{\sim} {\bb B}^2(\CC)$ and $\kappa_n\colon \CC H^n \xrightarrow{\sim} {\bb B}^n(\CC)$ identifying $x$ and $\nu(x)$ with the respective origins in $\BB^2(\CC)$ and $\BB^n(\CC)$, and that have the following properties: 
\begin{enumerate}
\item 
The isometries $\kappa_2$ and $\kappa_n$ commute with $\nu$ and the embedding $\BB^2(\CC) \hookrightarrow \BB^n(\CC)$ that sends $(x_1, x_2)$ to $(x_1, x_2, 0, \dotsc, 0)$. 
\item Assume that $x$ has $k$ real nodes $H_{r_i}$ $(k \in \set{1,2})$ defined by elements $r_i \in M$ of norm one. Let $t_i \in \Lambda$ be the image of $r_i$. For $i \geq 3$, define $t_{i} = (0, 0, \dotsc, 0, 1, 0, \dotsc, 0) \in \Lambda$, where the $1$ is placed on the $(i+1)$-th coordinate. Then $H_{t_1},  H_{t_{3-k}}, \dotsc, H_{t_n}$ are the nodes of $\nu(x)$. Moreover, $\kappa_2$ (resp.\ $
\kappa_n$) identifies $\phi_{r_1}$ (resp.\ $\phi_{t_1}, \dotsc, \phi_{t_n}$) with the map $(t_1, t_2) \mapsto (\zeta \cdot t_1, t_2)$ (resp.\ maps $(t_1, \dotsc, t_n) \mapsto (t_1,  \dotsc, t_{i-1}, \zeta  \cdot t_{i } ,t_{i+1}, \dotsc,  t_n)$), 
    and $\beta_{j_0}$ (resp.\ $\alpha_{j_0}$) with the map $(t_1, t_2) \mapsto (\bar t_1, \bar t_2)$ (resp.\ $(t_1, \dotsc, t_n) \mapsto (\bar t_1, \dotsc, \bar t_n)$). 
    \item 
    Assume that $x$ has one pair of complex conjugate nodes $H_{r_1}, H_{r_2}$. Let $t_i \in \Lambda$ be the image of $r_i$. For $i \geq 2$, define $t_{i} = (0, 0, \dotsc, 0, 1, 0, \dotsc, 0) \in \Lambda$, where the $1$ is placed on the $(i+2)$-th coordinate. Then $H_{t_1},  H_{t_2}, \dotsc, H_{t_n}$ are the nodes of $\nu(x)$. Moreover, the map $\kappa_2$ (resp.\ $
\kappa_n$) identifies $\phi_{r_1}$ (resp.\ $\phi_{r_2}$, resp.\ $\phi_{t_1}, \dotsc, \phi_{t_n}$) with the map $(t_1, t_2) \mapsto (\zeta \cdot t_1, t_2)$ (resp.\ $(t_1, t_2) \mapsto (t_1, \zeta t_2)$, resp.\ maps $(t_1, \dotsc, t_n) \mapsto (t_1,  \dotsc, t_{i-1}, \zeta  \cdot t_{i } ,t_{i+1}, \dotsc,  t_n)$), and $\beta_{j_0}$ (resp.\ $\alpha_{j_0}$) with the map $(t_1, t_2) \mapsto (\bar t_2, \bar t_1)$ (resp.\ $(t_1, \dotsc, t_n) \mapsto (\bar t_2, \bar t_1, \bar t_3, \bar t_4, \dotsc, \bar t_n)$). 
    \end{enumerate}
\end{lemma}
\begin{proof}[Proof of Lemma \ref{lemma:localcoordinates:infamily:new}]
See the proof of Lemma \ref{lemma:localcoordinates}.\ref{zwei}.
\end{proof}
 
To show $\iota \colon M(\mr M) \to M(\mr L, \alpha_0)$ is a morphism of real analytic orbifolds at the point $\bar f$, there are four cases to consider: $x \in \CCH^2$ either has zero, one or two real nodes, or $x$ has one pair of complex conjugate nodes (see Definitions \ref{fullset} and \ref{definition:nodes-ii}). 
\begin{enumerate}
\item[Case I:] \emph{The element $x \in \RRH^2_{\beta_{i_0}}$ has zero nodes.} \end{enumerate}
This case follows from Proposition \ref{proposition:canonical-map} (cf.\ Lemma \ref{lemma:openembeddingtopological}.\ref{item:openembedding}). 
\begin{enumerate}
\item[Case II:] \emph{The element $x \in \RRH^2_{\beta_{i_0}}$ has one real node and no other nodes.} 
\end{enumerate}
Consider the notation in item 2 of Lemma \ref{lemma:localcoordinates:infamily:new}. 
Under the coordinates $\CC H^2 \xrightarrow{\sim} \BB^2(\CC)$ of Lemma \ref{lemma:localcoordinates:infamily:new}, any $\xi \in P\mr B$ with $(x, \beta_{i_0}) \sim (x, \xi)$ to an involution of the form $(t_1, t_2) \mapsto (\zeta^i \cdot \bar t_1, \bar t_2)$. Thus if the $\xi_i$ for $i = 1, \dotsc, m$ are the $m$ anti-unitary involutions such that $(x, \beta_{i_0}) \sim (x, \xi_i)$, then for the set $Y_f$ defined in (\ref{Kf}), one has
$$
Y_f \cong \bigcup_{i = 1}^m \RR H^2_{\xi_i} \cong \set{(t_1, t_2) \in \BB^2(\CC) \mid t_1^m, t_2 \in \RR}. 
$$
Define $\chi_i = \phi_{r_1}^{i_1} \circ \cdots \circ \phi_{r_{n-1}}^{i_{n-1}} \circ \beta_{j_0}$ for $i = (i_1, \dotsc, i_{n-1}) \in \ZZ^m$ with $1 \leq i_j \leq m$. In the coordinates of Lemma \ref{lemma:localcoordinates:infamily:new}, we have 
$$
Y_g \cong \bigcup_{i_1, \dotsc, i_{n-1} = 1}^m \RR H^n_{\chi_i} \cong \set{(t_1, \dotsc, t_n) \in \BB^n(\CC) \mid t_1^m,t_2,  t_3^m, \dotsc, t_n^m \in \RR}.
$$
Next, define $K_{f, \epsilon} = \set{(t_1, t_2) \in \BB^2(\CC) \mid i^{-\epsilon} \cdot t_1 \in \RR_{\geq 0}, t_2 \in \RR}$ and  
\[
K_{g, \epsilon_1, \dotsc, \epsilon_{n-1}} = \set{
(t_1, \dotsc, t_n) \in \BB^n(\CC) \mid i^{\epsilon_1} t_1, i^{-\epsilon_2} t_3, \dotsc, i^{-\epsilon_{n-1}} t_n \in \RR_{\geq 0}, t_2 \in \RR}.
\]
With respect to the natural embedding $\nu \colon \CC H^2 \to \CC H^n$, one has for each $i = 1, \dotsc, m$, that
$$
\nu\left( \RR H^n_{\xi_i} \right) \subset \RR H^n_{\chi_{(i,m, \dotsc, m)}}.
$$
Therefore, the maps defined above fit together in the following commutative diagram:
\begin{align} \label{D-DIAGRM}
\begin{split}
\xymatrix{
B_f \setminus Y_f \ar[d]^{\wr} \ar@{^{(}->}[r] &B_g \setminus Y_g  \ar[d]^{\wr}  \\
B_f \setminus  \bigcup_{i = 1}^m \RR H^n_{\xi_i} \ar@{^{(}->}[r] \ar[d]^{\wr} & B_g \setminus \bigcup_{i_1, \dotsc, i_{n-1} = 1}^m \RR H^2_{\chi_i}  \ar[d]^{\wr}  \\
 \langle \zeta \rangle \setminus \set{(t_1, t_2) \in \BB^2(\CC) \mid t_1^m, t_2 \in \RR} \ar[d]^{\wr} \ar@{^{(}->}[r] & 
\langle \zeta \rangle^{n-1} \setminus \set{(t_1, \dotsc, t_n) \mid t_1^m,t_2,  t_3^m, \dotsc, t_n^m \in \RR} \ar[d]^{\wr} \\
  \bigcup_{\epsilon \in \set{0,1}} K_{f,\epsilon}  \ar[d]^{\wr} \ar@{^{(}->}[r] &   \bigcup_{\epsilon_1, \dotsc, \epsilon_{n-1} = 0}^1 K_{g, \epsilon_1, \dotsc, \epsilon_{n-1}} \ar[d]^{\wr}   \\
  \BB^2(\RR) \ar@{^{(}->}[r] & \BB^n(\RR).
}
\end{split}
\end{align}
Notice that embedding $B_f \setminus Y_f \hookrightarrow B_g \setminus Y_g$ on top of diagram \eqref{D-DIAGRM} is equivariant with respect to the homomorphism $\pi$ defined in (\ref{well-defined-embedding}). As in the proof of Proposition \ref{glueingtheorem1}.\ref{item:importanttheoremitem}, see also Lemma \ref{localisometry}.\ref{trois}, 
we choose an $A_f$ (resp.\ $A_g$)-equivariant open neighbourhood $f \in U_f \subset Y(\mr M)$ with $U_f \subset Y_f$ (resp.\ $g \in U_g \subset Y(\mr L)$ with $U_g \subset Y_g$) such that $A_f \setminus U_f  \subset \Gamma \setminus Y(\mr M)$ (resp.\ such that $A_g \setminus U_g \subset L \setminus Y(\mr L)$). 

Define an open subset $U_f' \subset U_f \subset Y_f$ as $U_f' \coloneqq U_f \cap \nu^{-1}(U_g)$, where we view $\nu$ as the canonical map $Y_f \hookrightarrow Y_g$ induced by $\nu \colon \CCH^2 \hookrightarrow \CCH^n$. Notice that $U_f' \subset Y(\mr M)$ is again an open neighbourhood of $f \in Y(\mr M)$. We claim that $U_f'$ is $A_f$-equivariant. To prove this, it suffices to show that $\nu^{-1}(U_g)$ is $A_f$-equivariant. To see this, let $\psi \in A_f$ and $h \in \iota^{-1}(U_g)$; we need to show that $\psi \cdot h \in \iota^{-1}(U_g)$. 
Lift $\psi$ to an element $\psi' \in \rm{Stab}_{\Aut(M)}(f)$. Via the canonical embedding $j \colon \Aut(M) \hookrightarrow \Aut(\Lambda)$ defined in \eqref{jn}, we obtain an element $j(\psi') \in \Stab_{\Aut(\Lambda)}(g)$ and we let $\overline{j(\psi')} \in A_g$ denote its image in $A_g$ under the quotient map $\Stab_{\Aut(\Lambda)}(g) \to A_g$. Then
$$
\nu(\psi \cdot h) = \nu(\psi' \cdot h) = j(\psi') \cdot \nu(h) = \overline{j(\psi')} \cdot \nu(h) \in U_g,
$$
where $\overline{j(\psi')} \cdot \nu(h) \in U_g$ because $U_g$ is $A_g$-equivariant. Hence $\psi \cdot h \in \nu^{-1}(U_g)$, so that $U_f'$ is $A_f$-equivariant, as claimed. 

Possibly up to replacing $U_f$ by $U_f'$, we may assume that $\nu \colon Y_f \hookrightarrow Y_g$ satisfies $\nu(U_f) \subset U_g$, and in particular, that $\iota(A_f \setminus U_f) \subset A_g \setminus U_g$. If we define $V_f \coloneqq B_f \sm U_f$ and $V_g \coloneqq B_g \sm U_g$, then, in view of diagram \eqref{D-DIAGRM}, we obtain an isometric embedding
\[
V_f \longhookrightarrow V_g
\]
which is equivariant for the homomorphism $\pi \colon A_f/B_f \to A_g/B_g$ defined in \eqref{well-defined-embedding} above, and which fits into the following commutative diagram:
\[
\xymatrix{
&\BB^2(\RR) \ar[r]^-{\sim} & B_f \sm Y_f  \ar@{^{(}->}[r] & B_g\sm Y_g  \ar[r]^-\sim & \BB^n(\RR)& \\
&& V_f\ar@{^{(}->}[u] \ar[d] \ar@{^{(}->}[r]  & V_g\ar@{^{(}->}[u] \ar[d] && \\
M(\mr M)&A_f \sm U_f \ar@{=}[r] \ar@{_{(}->}[l]& (A_f/B_f) \sm V_f \ar@{^{(}->}[r]^-{\iota}  & (A_g/B_g) \sm V_g \ar@{=}[r]& A_g \sm U_g \ar@{^{(}->}[r]& M(\mr L).
}
\]
Thus, the map $\iota \colon M(\mr M) \to M(\mr L, \alpha_0)$ is a morphism of real analytic orbifolds at the point $\bar f \in M(\mr M)$, which is what we wanted to prove.  
\begin{enumerate}
\item[Case III:] \emph{The element $x \in \RRH^2_{\beta_{i_0}}$ has two real nodes and no other nodes.} 
\end{enumerate}
This case is similar to Case II and we leave the proof to the reader. 
\begin{enumerate}
\item[Case IV:]\emph{The element $x \in \RRH^2_{\beta_{i_0}}$ has one pair of complex conjugate nodes and no other nodes.} 
\end{enumerate}
Consider the coordinates $\CC H^2 \xrightarrow{\sim} \BB^2(\CC)$ in item 3 of Lemma \ref{lemma:localcoordinates:infamily:new}. 
As in the proof of Proposition \ref{localmodel}.\ref{casethree} (see also Case II above), we have that $B_f$ is trivial and 
\begin{equation*}
Y_f = B_f \setminus Y_f \cong  
\left\{ (t_1, t_2) \in {\bb B}^2(\CC):
    t_2^m = \bar t_{1}^m \right\} = \bigcup_{i = 1}^m\BB^2(\RR)_{\xi_i},
\end{equation*}
where $\xi_i \colon \BB^2(\CC) \to \BB^2(\CC)$ is defined as $\xi_i(t_1, t_2) = (\bar t_2\zeta^i, \bar t_1\zeta^i)$ for $i \in \ZZ/m$. 
Similarly, in view of the proof of Proposition \ref{localmodel}.\ref{casefour}, we can describe $B_g \setminus Y_g$ as follows: 
\begin{align*}
Y_g &\cong \left\{ (t_1, \dotsc, t_n) \in {\bb B}^n(\CC) \mid t_2^{m} = \bar t_1^{m}, t_{3}^{m}, \dotsc, t_n^{m} \in \RR \right\}, \quad \quad  \tn{ and } \\
B_g \setminus Y_g &\cong  
\left\{ (t_1, \dotsc, t_n) \in {\bb B}^n(\CC) \mid 
    t_2^{m} = \bar t_1^{m}, t_{3}, \dotsc, t_n \in \RR  \right\} = \bigcup_{i = 1}^m\BB^n(\RR)_{\xi_i},
\end{align*}
where $\xi_i \colon \BB^n(\CC) \to \BB^n(\CC)$ is defined as $\xi_i(t_1, \dotsc, t_n) = (\bar t_2 \zeta^i, \bar t_1 \zeta^i, \bar t_3, \dotsc, \bar t_n)$ for $i \in \ZZ/m$. By Proposition \ref{localmodel}.\ref{casefive}, the group $A_f = \Stab_{\Gamma}(f)$ (resp.~$A_g = \Stab_L(g)$) acts transitively on the copies of $\BB^2(\RR)$ (resp.\ $\BB^n(\RR)$). Define 
$
S_f \coloneqq \tn{Stab}_{A_f / B_f}({\BB^2(\RR)_{\xi_1}})$ and $S_g \coloneqq \tn{Stab}_{A_g / B_g}({\BB^n(\RR)_{\xi_1}})$ as the respective stabilizers of $\BB^2(\RR)_{\xi_1}$ and $\BB^n(\RR)_{\xi_1}$. The map $Y_f = B_f \sm Y_f \to B_g \sm Y_g$ induced by $\nu \colon Y_f \hookrightarrow Y_g$ is equivariant for the map $\pi \colon A_f/B_f \to A_g/B_g$ constructed in Lemma \ref{lemma:canonical-hom}, hence we obtain a map
$$
A_f \sm Y_f = (B_f \sm A_f) \sm (B_f \sm Y_f) \longrightarrow (B_g \sm A_g) \sm (B_g \sm Y_g) = A_g \sm Y_g
$$
that fits into the following commutative diagram:
\[
\xymatrix{
&\bigcup_{i=1}^m \BB^2(\RR)_{\xi_i} \ar[r]^-\sim&Y_f\ar@{^{(}->}[r] \ar[dd]&B_g \sm Y_g\ar[r]^-\sim \ar[dd]& \bigcup_{i=1}^m \BB^n(\RR)_{\xi_i} \\
\BB^2(\RR)_{\xi_1}\ar@{^{(}->}[ur] \ar[rrrr] \ar[dd] &&&& \BB^n(\RR)_{\xi_1}\ar@{^{(}->}[u] \ar[dd] \\
&&A_f\sm Y_f \ar[r]& A_g\sm Y_g \ar[dr]^\sim& \\
S_f \setminus \BB^2(\RR)_{\xi_1} \ar[urr]^-\sim \ar[rrrr] &&&& S_g \setminus \BB^n(\RR)_{\xi_1}. 
}
\]
 As in the proof of Proposition \ref{glueingtheorem1}.\ref{item:importanttheoremitem}, see also Lemma \ref{localisometry}.\ref{trois}, 
we choose an $A_f$ (resp.\ $A_g$)-equivariant open neighbourhood $f \in U_f \subset Y(\mr M)$ with $U_f \subset Y_f$ (resp.\ $g \in U_g \subset Y(\mr L)$ with $U_g \subset Y_g$) such that $A_f \setminus U_f  \subset M(\mr M)$ (resp.\ such that $A_g \setminus U_g \subset M(\mr L)$). As in Case II, we may assume that the canonical map $\nu \colon Y_f \hookrightarrow Y_g$ satisfies $\nu(U_f) \subset U_g$, and in particular, that $\iota(A_f \setminus U_f) \subset A_g \setminus U_g$. 

Finally, define $V_f \coloneqq U_f \cap \BB^2(\RR)_{\xi_1}$ and $V_g  \coloneqq (B_g \sm U_g) \cap \BB^n(\RR)_{\xi_1}$. Then $V_f$ and $V_g$ are open subsets of $\BB^2(\RR)_{\xi_1}$ and $\BB^n(\RR)_{\xi_1}$ respectively, and we have 
a commutative diagram of the following form:
\[
\xymatrix{
\BB^2(\RR)_{\xi_1}&V_f \ar@{^{(}->}[rrr]\ar@{_{(}->}[l] \ar[d] &&& V_g \ar[d]\ar@{^{(}->}[r] & \BB^n(\RR)_{\xi_1} \\
&S_f \sm V_f \ar[r]^-\sim&A_f\sm U_f\ar@{^{(}->}[r]^-\iota& A_g\sm U_g\ar[r]^-\sim & S_g \sm V_g&
}
\]
Consequently, $\iota \colon M(\mr M) \to M(\mr L, \alpha_0)$ is a morphism of real analytic orbifolds at $\bar f \in M(\mr M)$. This finishes Case IV, and thereby the proof of Proposition \ref{prop:fund-funda-fund}.
\end{proof}


The following is a corollary of the proof of Proposition \ref{prop:fund-funda-fund}. 

\begin{corollary} \label{corollary:totallygeodesic}
Assume Condition \ref{condition:connected}. 
The map $\iota \colon M(\mr M) \to M(\mr L, \alpha_0)$ defined in Proposition \ref{proposition:canonical-map} is a totally geodesic immersion of complete connected real hyperbolic orbifolds, see Definition \ref{def:interestingly}. 
\end{corollary}
\begin{proof}
Let $i_0 \in \set{0,1,2}$. 
Let $x \in \RRH^2_{\beta_{i_0}}$ with image $\bar f \in M(\mr M)$. Consider $\nu(x) \in \RRH^n_{\beta_{i_0}}$ and let $\bar g \in M(\mr L, \alpha_0)$ be its image in $M(\mr L, \alpha_0)$. By the proof of Proposition \ref{prop:fund-funda-fund}, there exist open neighbourhoods $W_f \subset M(\mr M)$ and $W_g \subset M(\mr L, \alpha_0)$ of $\bar f$ and $\bar g$ respectively, such that $W_f = S_f \sm V_f$ for an open neighbourhood $V_f \subset \BB^2(\RR)$ of the origin $0 \in \BB^2(\RR)$ and a finite group of isometries $S_f$ that preserves $V_f$, and $W_g = S_g \sm V_g$ for an open neighbourhood $V_g \subset \BB^n(\RR)$ of the origin $0 \in \BB^n(\RR)$ and a finite group of isometries $S_g$ that preserves $V_g$, such that $\iota(W_f)  \subset W_g$ and such that the map $\iota \colon W_f \to W_g$ lifts to a totally geodesic immersion $V_f \hookrightarrow V_g$ which is equivariant for a group homomorphism $\pi \colon S_f \to S_g$. Replacing $V_f$ by its connected component that contains $0 \in \BB^2(\RR)$, we may assume that $V_f$ is connected. Replacing $V_g$ by its connected component containing the image of $V_f$, we may assume that $V_g$ is connected as well. The corollary follows then from Lemma \ref{direct}. 
\end{proof}

\begin{proof}[Proof of Theorem \ref{theorem:totally-geodesic-standard}]
For item \ref{a}, see Lemma \ref{connected}. For item \ref{b}, see Corollary \ref{corollary:totallygeodesic}. Item \ref{c} follows from item \ref{b} and Theorem \ref{th:bergeron}. 
\end{proof}

\begin{proof}[Proof of Theorem \ref{theorem:non-arithmeticity:p-arbitrary}]
Note that the $\Aut(\ZZ[\zeta_p]^{2,1}_\lambda)$-conjugacy classes of the anti-unitary involutions $(x_0, x_1, x_2) \mapsto (\epsilon_0\bar x_0, \epsilon_1 \bar x_1, \epsilon_2 \bar x_0)$ for $\epsilon_i \in \set{\pm 1}$ are the same as the $\Aut(\ZZ[\zeta_p]^{2,1}_\lambda)$-conjugacy classes of the elements $\pm \alpha_{2,i}$, where the latter are defined in \eqref{def:beta}. Thus, the result is a direct consequence of Theorem \ref{theorem:totally-geodesic-standard}.\ref{c}. 
\end{proof}

\section{Uniformization of real moduli spaces} \label{section:uniformization:realcubics} 

In this section, we first prove Theorem \ref{theorem:reprove-ACT}, and then prove Theorem \ref{theorem:introduction:non-arithmeticity}. 

\subsection{Cubic surfaces} \label{sec:cubic}

Write $\mr C_s^{\RR}$ for the space of non-zero
cubic forms with real coefficients in four variables which are stable in the sense of geometric invariant theory, and let $\mr C_0^\RR \subset \mr C_s^\RR$ be the open subspace of cubic forms whose associated cubic surface is smooth. The orbifold quotients
$$
\mr M_s^{\RR} = \GL_4(\RR) \setminus \mr C_s^{\RR} \supset \GL_4(\RR) \setminus \mr C_0^\RR = \mr M_0^\RR
$$
are the respective moduli spaces of stable and smooth real cubic surfaces. By \cite[Theorem 1.2]{realACTsurfaces}, there are a non-arithmetic lattice $P\Gamma^\RR \subset \PO(4,1)$, a homeomorphism 
\begin{align} \label{align:ACT-stable:body-text}
\mr M_s^\RR \cong P\Gamma^\RR \setminus \RR H^4, 
\end{align}
and a $P\Gamma^\RR$-invariant union $\mr H' \subset \RR H^4$ of two- and three-dimensional subspaces of $\RR H^4$ such that \eqref{align:ACT-stable} restricts to an orbifold isomorphism 
$
\mr M_0^\RR \cong P\Gamma^\RR \setminus \left( \RR H^4 - \mr H'\right)$. Let us explain this result in some detail, before we relate it to the uniformization of $M(\mr L_\eis^4)$, see Theorems \ref{theorem:introduction:uniformization} and \ref{theorem:introduction:non-arithmeticity}. 

As in \cite{realACTsurfaces}, let $\mr C$ be the space of all non-zero cubic forms
in four complex variables, $\Delta$ the discriminant locus, and $\mr C_0 = \mr C - \Delta$ the set of forms defining smooth cubic surfaces. Write $G = \GL(4,\CC)/\langle \zeta_3 \cdot \Id \rangle $ which acts faithfully and properly discontinuously on $\mr C_0$; the orbifold $\mr M_0 = G \sm \mr C_0$ is the moduli space of smooth cubic surfaces in $\PP^3_\CC$. Define 
\begin{align} \label{align:forget-the-framing}
\mr F_0 \longrightarrow \mr C_0 
\end{align}
as in \cite[Section 2.2]{realACTsurfaces}; thus $\mr F_0$ is the space of framed smooth cubic surfaces and \eqref{align:forget-the-framing} is the map that forgets the framing. 

As before, let $\mr A_\eis^4$ be the set of anti-unitary involutions $\alpha \colon \ZZ[\zeta_3]^{4,1} \to \ZZ[\zeta_3]^{4,1}$, let $P\mr A_\eis^4$ be the quotient of $\mr A_\eis^4$ by $\mu_{\QQ(\zeta_3)} = \langle - \zeta_3 \rangle$, and let $C\mr A_\eis^4 \subset P\mr A_\eis^4$ be a set of representatives for the action of the group $L \coloneqq \Aut(\ZZ[\zeta_3]^{4,1})/\langle -\zeta_3 \rangle$ on $P\mr A_\eis^4$ by conjugation. Thus, we have $C\mr A_\eis^4 \cong L \setminus P\mr A_\eis^4$. 

Let $\mr F_0^\RR$ be the preimage of $\mr C_0^\RR$ under the map \eqref{align:forget-the-framing}. As shown in \cite{realACTsurfaces}, each $\alpha \in P\mr A_\eis^4$ naturally lifts to an anti-holomorphic involution $\alpha \colon \mr F_0 \to \mr F_0$. Moreover, $\mr F_0^\alpha \cap \mr F_0^\beta = \emptyset$ for $\alpha \neq \beta \in P\mr A_\eis^4$, and we have
\[
\mr F_0^\RR = \coprod_{\alpha \in P\mr A_\eis^4} \mr F_0^\alpha. 
\]
as subspaces of $\mr F_0$. Let 
$$
g \colon \mr F_0 \longrightarrow \CC H^4
$$
be the period map defined in \cite[equation (2.10)]{realACTsurfaces}. For each $\alpha \in P\Aeisfour$, the map $g$ commutes with the involutions given by $\alpha$ on both sides, hence induces a map
$
g|_{\mr F_0^\alpha} \colon \mr F_0^\alpha \to (\CC H^4)^\alpha = \RR H^4_\alpha
$
for each $\alpha \in P \Aeisfour$. The resulting real analytic map 
\[
g^\RR \colon \mr F_0^\RR = \coprod_{\alpha \in P \Aeisfour} \longrightarrow \coprod_{\alpha \in P \Aeisfour}\RR H^4_\alpha 
\]
is constant on $G^\RR$-orbits, where $G^\RR \coloneqq \GL(4,\RR) \subset G$, and its image is contained in $\sqcup_{\alpha \in P \Aeisfour}\left( \RR H^4_\alpha - \mr H \right)$, where $\mr H \subset \CC H^4$ denotes as usual the hyperplane arrangement defined by the norm one vectors in $\ZZ[\zeta_3]^{4,1}$. By \cite[Theorem 3.3]{realACTsurfaces}, the induced map 
\[
g^\RR \colon G^\RR \sm \mr F_0^\RR \longrightarrow \coprod_{\alpha \in P\Aeisfour} \left( \RR H^4_\alpha - \mr H \right)
\]
is a diffeomorphism. Taking the quotient by $L \coloneqq \Aut(\ZZ[\zeta_3]^{4,1})/\langle - \zeta_3 \rangle$ yields an isomorphism of real analytic orbifolds
\[
\mr M_0^\RR = G^\RR \sm \mr C_0^\RR = (L \times G^\RR) \sm \mr F_0^\RR \longrightarrow L \sm \left( \RR H^4_\alpha - \mr H \right) \cong \coprod_{\alpha \in C \Aeisfour} L_\alpha \sm \left( \RR H^4_\alpha - \mr H \right).
\] 
Let $\mr F_s \supset \mr F_0$ be the Fox completion of $\mr F_0 \to \mr C_0$ over $\mr C_s$, see \cite{fox}. By \cite[(3.17) - (3.19)]{ACTsurfaces}, the period map $ g \colon \mr F_0 \to \CC H^4$ extends to a morphism $g_s \colon \mr F_s \to \CC H^4$ which factors through $G \sm \mr F_s$ and induces an $L$-equivariant diffeomorphism $g_s \colon G \sm \mr F_s \cong \CC H^4$. The latter sends the $k$-nodal cubic surfaces to the locus in $\CC H^4$ where exactly $k$ of the hyperplanes of $\mr H$ meet. The induced map $\mr M_s = G \sm \mr C_s \to L \sm \CC H^4$ is an isomorphism of analytic spaces. 

By \cite[Lemma 11.3]{realACTsurfaces}, for every $\alpha \in P\mr A$, the restriction of the period map $g_s \colon \mr F_s \to \CC H^4$ to $\mr F_s^\alpha \subset \mr F_s$ induces a surjective morphism 
\begin{align} \label{align:galpha}
 \mr F_s^\alpha \longtwoheadrightarrow \RR H^4_\alpha
\end{align} that factors through an isomorphism $G^\RR \sm \mr F_s^\alpha \cong  \RR H^4_\alpha$ of real analytic manifolds. 

Define $\wt Y = \sqcup_{\alpha \in P\mr A_\eis^4}\RR H^4_\alpha$ and let 
$p \colon \wt Y \to Y(\mr L_{\eis}^4)$ be the quotient map. As in \cite[Section 11.1]{realACTsurfaces}, let $\mr F_s^\RR$ be the preimage of $\mr C_s^\RR$ in $\mr F_s$. Note that, as topological spaces, we have
$
\left( L \times G^\RR \right) \sm \mr F_s^\RR = \mr M_s^\RR. 
$

\begin{theorem} \label{th:realstableperiod}
Consider the above notation and assumptions. Let
\begin{align}\label{therealstableperiodmap}
g_s^{\RR} \colon  \coprod_{\alpha \in P \Aeisfour} \mr F_s^\alpha \longtwoheadrightarrow  \coprod_{\alpha \in P \Aeisfour}\RR H^4_\alpha  =  \widetilde Y
\end{align} be the smooth map which is the coproduct of the maps $ \mr F_s^\alpha \to \RR H^4_\alpha$ defined in  \eqref{align:galpha} above. Then the morphism \eqref{therealstableperiodmap} induces the following commutative diagram of topological spaces, in which $h_s^\RR$ and $\overline h_s^\RR$ are homeomorphisms:
\begin{equation*}
\xymatrixcolsep{5pc}
\xymatrix{
&\coprod_{\substack{{\alpha \in P \Aeisfour}}} \mr F_s^\alpha \ar@{->>}[r]^{g_s^\RR}\ar[d] & \wt Y \ar[d]^p\\
&\mr F_s^\RR \ar@{->>}[r]^{\overline g_s^\RR}\ar[d] & Y(\mr L_\eis^4)\ar@{=}[d] \\
 &G^\RR \setminus \mr F_s^\RR \ar[r]^{h_s^\RR}_\sim\ar[d] & Y(\mr L_\eis^4)\ar[d] \\
{\mr M_s^\RR} \ar@{=}[r] &G^\RR \setminus \mr C_s^\RR \ar[r]^{\overline h_s^\RR}_\sim & M(\mr L_\eis^4).
}
\end{equation*}
%
\end{theorem}

\begin{proof}
We claim that the composition $p \circ g_s^\RR$ 
factors through a morphism $\overline g_s^\RR \colon \mr F_s^\RR \to Y(\mr L_\eis^4)$. Let $f_\alpha, g_\beta\in \coprod_{\alpha \in P \Aeisfour} \mr F_s^\alpha$. Then $f_\alpha$ and $g_\beta$ have the same image in $\mr F_s^\RR$ if and only if $f = g \in \mr F_s^\alpha \cap \mr F_s^\beta$, in which case 
$
g_s(f) = g_s(g) \eqqcolon x \in \RR H^4_\alpha \cap \RR H^4_\beta$. Note that $\alpha\beta \in L_f$, 
and $g_s$ induces an isomorphism 
$
L_f \cong G(x)
$
by \cite[Lemma 10.3]{realACTsurfaces}. Hence $\alpha \beta \in G(x)$ so that $(x,\alpha) \sim (x,\beta)$, proving the claim. 

Next, we prove the $G^\RR$-equivariance of $\overline g_s^\RR$. Suppose that 
$
f \in \mr F_s^\alpha, g \in \mr F_s^\beta$ such that $ a \cdot f = g \in \mr F_s^\RR$ for some $ a \in G^\RR$. Then $x\coloneqq g_s(f) = g_s(g) \in \CC H^4$, so we need to show that $\alpha \beta \in G(x)$. The actions of $G$ and $L$ on $\CC H^4$ commute, and the same holds for the actions of $G^\RR$ and $L'$ on $\mr F_s^\RR$, where $L'$ is as in Section \ref{section:anti}. It follows that 
$
\alpha(g) = \alpha (a \cdot f) = a \cdot \alpha(f) = a \cdot f = g,
$
hence $g \in \mr F_s^\alpha \cap \mr F_s^\beta$. This implies in turn that $(\alpha \circ \beta) (g) = g$, hence $\alpha \beta \in L_g \cong G(x)$, see \cite[Lemma 10.3]{realACTsurfaces}. Therefore, $(x,\alpha) \sim (x,\beta)$, so that $\alpha \beta \in G(x)$ as desired. 

Let us prove that the resulting map $h_s^\RR \colon G^\RR \sm \mr F_s^\RR \to Y(\mr L_\eis^4)$ is injective. To do so, let 
$
f \in \mr F_s^\alpha$ and $g \in \mr F_s^\beta$ and suppose that $
x\coloneqq g_s(f) = g_s(g) \in \RR H^4_\alpha \cap \RR H^4_\beta,
$
and $\beta = \phi \circ \alpha$ for some $\phi \in G(x)$. 
We have $\phi \in G(x) \cong L_f$ (cf.\ \cite[Lemma 10.3]{realACTsurfaces}) hence $    \beta(f) = \phi \left(\alpha (f)\right) = \phi(f) = f$. 
Therefore, we have $f,g \in \mr F_s^\beta$; since $g_s(f) = g_s(g)$, it follows from \cite[Lemma 11.3]{realACTsurfaces} that there exists $a \in G^\RR$ such that $a \cdot f = g$. This proves the injectivity of $h_s^\RR$.   
The surjectivity of $h_s^\RR$ is straightforward: it follows from the surjectivity of $g_s^\RR$, see \cite[Lemma 11.3]{realACTsurfaces}. 

Finally, we claim that $h_s^\RR$ is open. Let $U \subset G^\RR \setminus \mr F_s^\RR$ be open. 
Let $V$ be the preimage of $U$ in $\coprod_{\alpha \in P \Aeisfour}\mr F_s^\alpha$. Then 
$
V = (g_s^\RR)^{-1}\left( p^{-1}\left(h_s^\RR(U)\right)\right)
$, hence 
$
    g_s^\RR\left( V \right) = p^{-1}\left(h_s^\RR(U)\right)$, so that it suffices to show that $g_s^\RR(V)$ is an open subset of $\coprod_{\alpha \in P\Aeisfour} \RR H^4_\alpha$. This follows, because $g_s^\RR$ is open, being the coproduct of the maps $\mr F_s^\alpha \to \RR H^4_\alpha$, which are open since they have surjective differential at each point. 
\end{proof}

\begin{proof}[Proof of Theorem \ref{theorem:reprove-ACT}]
Since $\mr M_s^\RR$ is connected, the same holds for $M(\mr L_\eis^4)$ in view of Theorem \ref{th:realstableperiod}. By \cite[Theorem 4.1]{realACTsurfaces}, we know that $\set{\alpha_{4,0}, \dotsc, \alpha_{4,4}}$ is a set of representatives for the action of the group $L = \Aut(\ZZ[\zeta_3]^{4,1})/\langle -\zeta_3 \rangle$ on $P\mr A_\eis^4$ by conjugation. Here, for $i \in \set{0,1,2,3,4}$, the anti-unitary involution $\alpha_{4,i}$ on $\ZZ[\zeta_3]^{4,1}$ is defined in \eqref{def:beta}. 
Hence item 1 follows from Theorem \ref{th:realstableperiod} in conjunction with Theorem 
\ref{theorem:volume} (whose proof shall be provided in Section \ref{section:volume}, and only depends on Theorem \ref{theorem:introduction:uniformization}, Lemmas \ref{normalisator} and \ref{lemma:conjinv:intro}, and the proof of Theorem \ref{th:crucialthm-finitevolume}). Item 2 is clear. 
\end{proof}

\subsection{Binary sextics with a double root at infinity} \label{section:binary}

We use (and extend) the notation of \cite[Section 5]{realACTnonarithmetic}. Consider the real projective line $\PP^1_\RR= \Proj(\RR[x,y])$, and fix $\infty \coloneqq [1\colon0] \in \PP^1(\RR) \subset \PP^1(\CC)$. Let $\ca P_\infty$ be the space of non-zero complex polynomials $F \in \CC[x,y]$ which are homogeneous of degree six that have a double root at $\infty$ such that all other roots of $F$ in $\PP^1(\CC)$ have multiplicity at most two. Let $\ca P_{\infty, \rm{sm}} \subset \ca P_\infty$ be the subspace of polynomials $F \in \ca P_\infty$ such that apart from $\infty$, all roots of $F$ have multiplicity one. For $F \in \ca P_{\infty, \rm{sm}}$, we have $F(x,y) = y^2 \cdot F_1(x,y)$ for some $F_1 \in \CC[x,y]$ such that $F_1(1,0) \neq 0$, hence $y^{-1}F = y F_1$ is a two-variable homogeneous polynomial of degree five with distinct roots in $\PP^1(\CC)$, one of which is $\infty$. Consider the weighted projective plane $\PP(3,3,5)$ and define $C_F 
\subset \PP(3,3,5)$ as the smooth curve
\[
C_F \coloneqq \set{z^3 + y^{-1} F(x,y) = 0} \subset \PP(3,3,5).
\]
Then $C_F$ is a degree three covering of $\PP^1$ branched along five distinct points on $\PP^1$ (among which $\infty = [1 \colon 0]$), hence $C_F$ is a curve of genus three. The cyclic covering group is generated by $C_F \to C_F, (x,y,z) \mapsto (x,y,\zeta_3z)$, and the induced action of $\ZZ/3$ on $\rm H^1(C_F,\ZZ)$ fixes no non-zero element. Thus, $\rm H^1(C_F,\ZZ)$ is a free $\ZZ[\zeta_3]$-module of rank three. Moreover, the hermitian form 
\[
h(x,y) = i \sqrt{3} \int x \wedge \bar y
\]
is $\ZZ[\zeta_3]$-valued on the projection of $\rm H^1(C_F,\ZZ)$ to $\rm H^1(C_F,\CC)_{\bar \zeta_3}$, and the induced hermitian form 
\[
h \colon \rm H^1(C_F,\ZZ) \times \rm H^1(C_F,\ZZ) \longrightarrow \ZZ[\zeta_3]
\]
is unimodular. Moreover, $h$ is positive definite on the one-dimensional space $\rm H^{1,0}(C_F)_{\bar \zeta_3}$ and negative definite on the two-dimensional space $\rm H^{0,1}(C_F)_{\bar \zeta_3}$. Since
$$
\rm H^1(C_F,\ZZ) \otimes_{\ZZ} \RR = \rm H^1(C_F,\CC)_{\bar \zeta_3} = \rm H^{1,0}(C_F)_{\bar \zeta_3} \oplus \rm H^{0,1}(C_F)_{\bar \zeta_3}, 
$$we conclude that $\rm H^1(C_F,\ZZ) \cong \ZZ[\zeta_3]^{1,2}$ as hermitian modules over $\ZZ[\zeta_3]$, where $\ZZ[\zeta_3]^{1,2}$ is the free $\ZZ[\zeta_3]$-module of rank three equipped with the hermitian form $h(x,y) = x_0\bar y_0 - x_1\bar y_1-x_2\bar y_2$. Let 
\[
\mr F_0 \longrightarrow \ca P_{\infty, \rm{sm}}
\]
be the covering space of pairs $(F, [i])$ where $F \in \ca P_{\infty, \rm{sm}}$ and $[i]$ is a $\langle -\zeta_3 \rangle$-equivalence class of isometries $i \colon \rm H^1(C_F,\ZZ) \cong \ZZ[\zeta_3]^{1,2}$. 
\begin{lemma} \label{F0-connected}
The space $\mr F_0$ is connected. Equivalently, fixing a point $F_0 \in \ca P_{\infty, \rm{sm}}$, the monodromy representation $\rho \colon \pi_1(\ca P_{\infty, \rm{sm}},F_0) \to \Aut(\ZZ[\zeta_3]^{1,3})/\langle - \zeta_3 \rangle$ associated to $F_0$ and the local system of $\ZZ[\zeta_3]$-modules with fibre $\rm H^1(C_F,\ZZ)$ is surjective. 
\end{lemma}
\begin{proof}
This follows as in \cite[Theorem 2.1]{realACTnonarithmetic} (cf.\ \cite[Lemma 7.12]{ACTsurfaces}). 
\end{proof}

For any $(F, [i]) \in \mr F_0$, the isometry $i \colon \rm H^1(C_F,\ZZ) \xrightarrow{\sim} \Lambda$ induces an isometry $i_\CC \colon \rm H^1(C_F,\CC)_{\bar \zeta_3} \xrightarrow{\sim} \Lambda_\CC = \CC^{1,2}$, and  
\[
i_\CC\left( \rm H^{1,0}(C_F) _{\bar \zeta_3} \right) \subset i_\CC\left(\rm H^1(C_F,\CC)_{\bar \zeta_3}\right) = \CC^{1,2}
\]
is a positive line in $\CC^{1,2}$. Define $B^2$ as the space of positive lines in $\CC^{1,2}$; thus
$$B^2 = \set{\text{lines } \ell \subset \CC^3 \text{ such that } \va{x_0}^2 - \va{x_1}^2 - \va{x_2}^2 > 0 \text{ for } 0 \neq x \in \ell}.$$
We get a map, the \emph{period map},
\[
g \colon \mr F_0 \longrightarrow B^2, \quad \quad g(F,[i]) = i_\CC\left( \rm H^{1,0}(C_F) _{\bar \zeta_3} \right) \in B^2. 
\]
The map $g$ is holomorphic since the Hodge filtration varies holomorphically. 

Let $$\GL_2(\CC)_\infty \coloneqq \rm{Stab}_{\GL_2(\CC)}(\infty) \subset \GL_2(\CC), \quad \quad G \coloneqq \GL_2(\CC)_\infty/\langle \zeta_6 \rangle.$$
The group $G$ acts on $\ca P_\infty$ by $h \cdot F(x,y) = F(h^{-1}\cdot (x,y))$.   
Let $\mr F_s \to \ca P_{\infty, \rm{st}}$ be the Fox completion of $\mr F_0 \to \ca P_{\infty, \rm{sm}} \hookrightarrow \ca P_{\infty}$ (cf.\ \cite{fox}). The map $g$ extends to a map $g \colon \mr F_s \to B^2$ which is constant on $G$-orbits, hence induces a morphism of complex analytic spaces $g \colon G \sm \mr F_s \to B^2$. 

\begin{theorem}[Picard, Deligne--Mostow] \label{theorem:DM-infinity}
The map $g \colon G \sm \mr F_s \to B^2$ is an isomorphism. 
\end{theorem}
\begin{proof}
See \cite{picardperiods} and \cite{DeligneMostow}. See also \cite[Section 5]{realACTnonarithmetic}. 
\end{proof}

Let $\mr N_s^\RR(\infty)$ be the moduli space of stable real binary sextics that have a double root at $\infty$ (see \cite[Section 5]{realACTnonarithmetic}), and let $\mr N_{\rm{sm}}^\RR(\infty)$ be the space of stable real sextics with a double root at infinity that have no other double roots. In the above notation, we have $\mr N_s^\RR(\infty) = G^\RR \sm \ca P_{\infty}^\RR$ and $\mr N_{\rm{sm}}^\RR(\infty) = G^\RR \sm \ca P_{\infty, \rm{sm}}^\RR$, where $G^\RR$ is the fixed locus of the natural anti-holomorphic involution of $G$, $\ca P^\RR_{\infty}$ denotes the space of polynomials $F \in \ca P_{\infty}$ whose coefficients are real, and $\ca P^\RR_{\infty, \rm{sm}} = \ca P^\RR_{\infty} \cap \ca P_{\infty, \rm{sm}}$. 

\begin{theorem} \label{theorem:connection-ACT} The following assertions are true. 
\begin{enumerate}
\item 
There are a union of geodesic subspaces $\mr H_i \subset \RR H^2$ for $i = 0, \dotsc, 4$ and a homeomorphism $\mr N_s^\RR(\infty) \cong \Gamma_{\rm{eis}}^2 \sm \RR H^2$ 
that restricts to an orbifold isomorphism $\mr N_{\rm{sm}}^{\RR}(\infty) \cong \coprod_{i = 0}^2  \rm{PO}(\Psi_i^2,\ZZ) \setminus \left(\RR H^2 - \mr H_i \right)$.
\item The lattice 
$\GamEis^2 \subset \PO(2,1)$ is conjugate to the lattice 
$\Gamma_\infty^\RR \subset \PO(2,1)$ defined in \cite[Section 5]{realACTnonarithmetic}. 
\end{enumerate}
\end{theorem}
\begin{proof}
Using the set-up and results above, together with the results of \cite[Section 5]{realACTnonarithmetic}, one can prove this in a way similar to the way in which we proved Theorem \ref{theorem:reprove-ACT} in Section \ref{sec:cubic} above; we omit the details. 
\end{proof}

\subsection{Proof of Theorem \ref{theorem:introduction:non-arithmeticity}}  \label{section:proof-of-theorem}

In the proof of Theorem \ref{theorem:introduction:non-arithmeticity}, that we shall provide below, key inputs are Theorem \ref{theorem:totally-geodesic-standard}, Theorem \ref{theorem:connection-ACT}, and the following result.

\begin{theorem} \label{useful-theorem}
Consider the Eisenstein hermitian lattice $\mr L^2_{\eis} = (\QQ(\zeta_3), \ZZ[\zeta_3]^{2,1})$ of rank three. Each anti-unitary involution of $\ZZ[\zeta_3]^{2,1}$ is $\Aut(\ZZ[\zeta_3]^{2,1})$-conjugate to exactly one of the $\pm \alpha_{2,i}$, where the anti-unitary involutions $\alpha_{2,i}$ are defined in equation \eqref{def:beta}. In other words, Condition \ref{condition:connected} holds in this case.  
\end{theorem}

To prove Theorem \ref{useful-theorem}, the idea is to use 
the ball quotient uniformization of Theorem \ref{theorem:DM-infinity}. 
More precisely, we will deduce Theorem \ref{useful-theorem} from the following consequence of Theorem \ref{theorem:DM-infinity}.  
For an anti-unitary involution $\alpha$ on $\ZZ[\zeta_3]^{1,2}$, define an anti-holomorphic involution as follows:
\[
\alpha \colon \mr F_0 \longrightarrow \mr F_0, \quad \quad \alpha(F,[i]) = (\kappa \cdot F, [\alpha \circ i \circ \kappa^\ast]), 
\]
where $\kappa\cdot F$ denotes the binary sextic whose coefficients are the complex conjugates of the coefficients of $F$, and where $\kappa^\ast \colon \rm H^1(C_{\kappa F},\ZZ) \cong \rm H^1(C_F,\ZZ)$ is the pull-back of the canonical anti-holomorphic map $\kappa \colon C_F \xrightarrow{\sim} C_{\kappa F}$. 


\begin{theorem} \label{theorem:non-empty}
For each anti-unitary involution $\alpha$ on $\ZZ[\zeta_3]^{1,2}$, the fixed locus $\mr F_0^\alpha$ of $\alpha$ is non-empty. 
\end{theorem}
Before we prove Theorem \ref{theorem:non-empty}, we show that it implies Theorem \ref{useful-theorem}. To explain this, let, as above, $\ca P^\RR_\infty \subset \ca P_\infty$ be the space of polynomials $F \in \ca P_\infty$ whose coefficients are real, and $\ca P^\RR_{\infty, \rm{sm}} \subset \ca P^\RR_{\infty}$ the open subspace of polynomials $F \in \ca P^\RR_\infty$ such that apart from $\infty$, all roots of $F$ have multiplicity one. Then $\ca P^\RR_{\infty, \rm{sm}}$ has three connected components $\ca P^\RR_{\infty, i}$ $(i = 0,1,2)$ such that $\ca P^\RR_{\infty, i}$ parametrizes real polynomials with exactly $i$ pairs of complex conjugate roots (and $4-2i$ real roots not equal to $\infty$). 
\begin{corollary} \label{corollary:thenthere}
Let $\alpha \colon \ZZ[\zeta_3]^{1,2} \to \ZZ[\zeta_3]^{1,2}$ be an anti-unitary involution. Then there exists a real binary sextic with double root at infinity and no other double roots $F \in \ca P_{\infty, \rm{sm}}^\RR$ and an isometry $i \colon \rm H^1(C_F,\ZZ) \cong \ZZ[\zeta_3]^{1,2}$ such that the composition 
\[
\xymatrix{
\ZZ[\zeta_3]^{1,2} \ar[r]_-{\sim}^-{i^{-1}} & \rm H^1(C_F,\ZZ) \ar[r]_-{\sim}^-{\kappa^\ast}& \rm H^1(C_F,\ZZ) \ar[r]_-{\sim}^-{i} & \ZZ[\zeta_3]^{1,2}
}
\]
agrees with $\alpha$ up to multiplication by $\langle -\zeta_3 \rangle$ (i.e., $[\alpha] = [i \circ \kappa^\ast \circ i^{-1}] \in P\mr A$).
\end{corollary}
\begin{proof}
Indeed, we have that $\mr F_0^\alpha$ is non-empty by Theorem \ref{theorem:non-empty}. Therefore, there exists $(F,[i]) \in \mr F_0$ such that $(\kappa \cdot F, [\alpha \circ i \circ \kappa^\ast]) = \alpha (F,[i]) = (F,[i])$. This implies that $F \in \ca P_{\infty, \rm{sm}}^\RR$ and that $[\alpha \circ i \circ \kappa^\ast] = [i]$, i.e., we have $[\alpha] = [i \circ \kappa^\ast \circ i^{-1}]$. 
\end{proof}

\begin{proof}[Proof of Theorem \ref{useful-theorem}]
Let $\mr A$ be the set of anti-unitary involutions of $\ZZ[\zeta_3]^{1,2}$, let $P\mr A$ be its quotient by $\langle - \zeta_3 \rangle$, 
and define $L = \Aut(\ZZ[\zeta_3]^{1,2})/\langle - \zeta_3 \rangle$. 
Notice that $P\mr A$ equals the set of $\langle -\zeta_3 \rangle$-equivalence classes of anti-unitary involutions $\ZZ[\zeta_3]^{2,1}$, and that $L = \Aut(\ZZ[\zeta_3]^{2,1})/\langle - \zeta_3 \rangle$. By Lemma \ref{lemma:conjinv:intro}, we have that the classes $[\alpha_{2,i}] \in L \sm P\mr A$ for $i = 0,1,2$ are pairwise distinct. Thus, we have $\# (L \sm P\mr A) \geq 3$, and to prove Theorem \ref{useful-theorem}, it suffices to prove that $\# (L \sm P\mr A) \leq 3$. 

To this end, consider the space $\ca P_{\infty, \rm{sm}}^\RR = \ca P^\RR_{\infty, 0} \sqcup \ca P^\RR_{\infty, 1} \sqcup \ca P^\RR_{\infty, 2}$ defined above. For $F  \in \ca P_{\infty, \rm{sm}}^\RR$, choose an isometry $i \colon \rm H^1(C_F,\ZZ) \cong \ZZ[\zeta_3]^{1,2}$, and let $\alpha(F) \colon \ZZ[\zeta_3]^{1,2} \to \ZZ[\zeta_3]^{1,2}$ be the anti-unitary involution defined as $\alpha(F) = i \circ \kappa^\ast \circ i^{-1}$, where $\kappa \colon C_F(\CC) \to C_F(\CC)$ is the natural anti-holomorphic involution of the Riemann surface $C_F(\CC)$. 
Then this construction yields a well-defined function
\begin{align} \label{align:def-a-funct}
\ca P_{\infty, \rm{sm}}^\RR \longrightarrow L \sm P\mr A, \quad \quad F \mapsto \alpha(F). 
\end{align}
By Corollary \ref{corollary:thenthere}, the function \eqref{align:def-a-funct} is surjective. Moreover, the $L$-conjugacy class $[\alpha(F)] \in L \sm P\mr A$ of $\alpha \in P\mr A$ only depends on the connected component of $\ca P_{\infty, \rm{sm}}^\RR $ in which $F$ lies, hence the map \eqref{align:def-a-funct} factors through a surjective map 
\[
\set{0,1,2} = \pi_0\left( \ca P_{\infty, \rm{sm}}^\RR  \right)  \longtwoheadrightarrow L \sm P\mr A. 
\]
It follows that $\# \left( L \sm P\mr A \right) \leq 3$, concluding the proof of the theorem.
\end{proof}

\begin{proof}[Proof of Theorem \ref{theorem:non-empty}]
Define $\Delta_{\mr F_s} = \mr F_s - \mr F_0$. Then $\Delta_{\mr F_s}$ is a closed analytic subset of $\mr F_s$, and the quotient $\Delta \coloneqq G \sm \Delta_{\mr F_s}$ is a closed analytic subset of $G \sm \mr F_s$. Consider the period map $g \colon G \sm \mr F_s \to B^2$, which is an isomorphism by Theorem \ref{theorem:DM-infinity}, and define $U \subset B^2$ as the complement of $g(\Delta)$ in $B^2$. 

Let $\alpha \colon \ZZ[\zeta_3]^{1,2} \to \ZZ[\zeta_3]^{1,2}$ be an anti-unitary involution. Then $\alpha$ induces an anti-holomorphic involution $\alpha \colon B^2 \to B^2$; let $B^2_\alpha \coloneqq (B^2)^\alpha$ denote its fixed locus. Then $B^2_\alpha$ is naturally isometric to real hyperbolic space of dimension two, see Lemma \ref{hyperbolic}. Let $U^\alpha \coloneqq B^2_\alpha \cap U$. We claim that $U^\alpha \neq \emptyset$. Indeed, if $U^\alpha = \emptyset$, then $B^2_\alpha \subset g(\Delta)$ which is absurd since $g(\Delta) \subsetneq B^2$ is a properly contained closed complex analytic subset of $B^2$, hence $g(\Delta)$ cannot contain a closed real analytic submanifold of $B^2$ which is real analytically isomorphic to a real two-ball. 

By the naturality of the Fox completion, the anti-holomorphic involution $\alpha \colon \mr F_0 \to \mr F_0$ extends to an anti-holomorphic involution $\alpha \colon \mr F_s \to \mr F_s$ which is compatible with the natural anti-holomorphic involution $\kappa \colon G \to G$ hence descends to an anti-holomorphic involution $\alpha \colon G \sm \mr F_s \to G \sm \mr F_s$. Moreover, the period map $g \colon G \sm \mr F_s \xrightarrow{\sim} B^2$ is equivariant for the anti-holomorphic involutions $\alpha \colon G \sm \mr F_s \to G \sm \mr F_s$ and $\alpha \colon B^2 \to B^2$, and induces an isomorphism $g \colon G \sm \mr F_0 \xrightarrow{\sim} U$. In particular, 
\[
(G \sm \mr F_0)^\alpha \cong U^\alpha \neq \emptyset. 
\]
Consequently, to prove the theorem, it suffices to show that the natural map 
\begin{align}
G^\kappa \sm \mr F_0^\alpha \longrightarrow (G \sm \mr F_0)^\alpha 
\end{align}
is an isomorphism. For this, by Lemma \ref{lemma:H1lemma} below, it suffices to show that $G$ acts freely on $\mr F_0$ and that $\rm H^1(\ZZ/2, G) = 0$, where the latter denotes degree one non-abelian group cohomology, and where $\ZZ/2$ acts on $G$ via conjugation by $\kappa \colon G \to G$.

Let us first prove that $G$ acts freely on $\mr F_0$. This holds, since if $(F, [i]) \in \mr F_0$ and $(g\cdot F, [i \circ g^\ast]) = (F, [i])$, then $g$ induces an automorphism $\tilde g \colon C_F \xrightarrow{\sim} C_F$ such that $(\tilde g)^\ast \colon \rm H^1(C_F,\ZZ) \to \rm H^1(C_F,\ZZ)$ is multiplication by some power of $-\zeta_3$, thus $\tilde g$, and hence also $g$, is some power of $-\zeta_3$ (Torelli theorem for curves). 

Thus, to finish the proof of Theorem \ref{theorem:non-empty}, it suffices to show that $\rm H^1(\ZZ/2, G) = 0$. For this, we first show that the natural map $\rm H^1(\ZZ/2, \GL_2(\CC)_\infty) \to \rm H^1(\ZZ/2, G)$ is surjective. Let $\GL_2(\CC)_\infty'$ be the set of linear and anti-linear transformations $\CC^2 \cong \CC^2$ that fix $\infty = [1\colon0] \in \PP^1(\CC)$, and let $G' = \GL_2(\CC)_\infty'/\langle - \zeta_3 \rangle$. Let $\chi_0 \in G$ such that $\chi_0 \kappa(\chi_0) = e$. Then $\chi \coloneqq \chi_0 \circ \kappa \in G'$ satisfies $\chi^2 = \chi_0\kappa \chi_0 \kappa = \chi_0\kappa(\chi_0) = e$. Let $\gamma \in \GL_2(\CC)_\infty'$ be an element that maps to $\hat \chi \in G'$. Then $\gamma$ is anti-linear since $\hat \chi$ is anti-holomorphic. Note that $\gamma^2 = \zeta_6^i$ for some $i \in \ZZ/6$ (because $\gamma$ maps to $\hat \chi$, which is an involution). Define $\alpha = \gamma^3$. Then $\alpha$ maps to $(\hat \chi)^3 = \hat \chi$ and $\alpha^2 = (\gamma^2)^3 = (\zeta_6)^{3i} = (-1)^i$. Thus $\alpha^2 = \pm 1$. We claim that $\alpha^2 = 1$. Assume for a contradiction that $\alpha^2 = -1$. Write $\alpha = \kappa \circ \beta$ for some $\beta \in \GL_2(\CC)$, where $\kappa$ is complex conjugation. Note that $\beta \in \GL_2(\CC)_\infty$, hence $\beta = \begin{psmallmatrix}b_1 & b_2\\0 & b_4\end{psmallmatrix}$ for some $b_1, b_2, b_4 \in \CC$. Then $\alpha^2 = \kappa \beta \kappa \beta = \beta \circ \bar \beta = -1$ which implies that $\va{b_1}^2 = \va{b_4}^2 = -1$, which is a contradiction. Hence $\alpha^2 = 1$. 
Thus, we have found an element $\alpha \in \GL_2(\CC)_\infty'$ with $\alpha^2 = \id$, and such that $[\alpha] = \chi \in G$. Define $\alpha_0 \coloneqq \alpha \circ \kappa$. Then $\alpha_0 \circ \kappa(\alpha_0) = \alpha \circ \kappa \circ \kappa(\alpha\circ  \kappa) = \alpha \circ \kappa \circ \kappa(\alpha)= \kappa(\alpha)^2 = \kappa(\alpha^2) = \id$. Hence $\alpha_0$ defines an element in $\rm H^1(\ZZ/2, \GL_2(\CC)_\infty)$ that maps to the class of $\chi_0$ in $\rm H^1(\ZZ/2, G)$ under the natural map $\rm H^1(\ZZ/2, \GL_2(\CC)_\infty) \to \rm H^1(\ZZ/2, G)$, thus, this map is surjective. 

To finish the proof, it suffices to show $\rm H^1(\ZZ/2,\GL_2(\CC)_\infty) = 0$. Let $\alpha \colon \CC^2 \to \CC^2$ an anti-linear involution that fixes $\infty$. It suffices to show that there exists $f \in \GL_2(\CC)_\infty$ such that $\alpha = f^{-1} \kappa f$. Recall that all real structures on a complex vector space are equivalent, so that there exists $h \in \GL_2(\CC)$ with $\alpha = h^{-1}\kappa h$, where $\kappa \in \GL_2(\CC)$. 
Notice that $(h^{-1}\kappa h)(\infty) = \alpha(\infty) = \infty$ implies $\kappa(h(\infty)) = h(\infty)$. In other words, we have $h(\infty) \in \PP^1(\RR) \subset \PP^1(\CC)$. Let $f_1 \in \GL_2(\RR)$ such that $f_1 (h(\infty)) = \infty$, and define $f \coloneqq f_1 \circ h$. Then $f(\infty) = \infty$. Moreover, since $f_1 \in \GL_2(\RR)$, we have $f_1^{-1} \kappa f_1 = \kappa$, so that $f^{-1} \kappa f = h^{-1}(f_1^{-1} \kappa f_1) h  = h^{-1}\kappa h = \alpha$. 

This concludes the proof of the theorem.
\end{proof}

\begin{lemma} \label{lemma:H1lemma}
Let $X$ be a complex analytic space, equipped with an anti-holomorphic involution $\alpha \colon X \to X$. Let $G$ be a complex Lie group equipped with an anti-holomorphic involution $\sigma \colon G \to G$, such that $G$ acts holomorphically on $X$ with $\alpha(g \cdot x) = \sigma(g) \cdot \alpha(x)$ for $x \in X$ and $g \in G$. 
Assume that $\rm H^1(\ZZ/2, G) = 0$ (where $\ZZ/2$ act on $G$ by conjugation by $\sigma$, and $\rm H^1(\ZZ/2, G) $ denotes the first non-abelian cohomology group of $\ZZ/2$ with coefficients in $G$), 
and that $G$ acts freely on $X$. Then the natural map $G^\sigma \sm X^\alpha \to (G\sm X)^\alpha$ is an isomorphism.  
\end{lemma}
\begin{proof}
Let $x \in X$ such that its orbit $[x] \in G\sm X$ is fixed by $\alpha$. Then $\alpha(x) = gx$ for some $g \in G$. We have $x = \alpha^2(x) = \sigma(g)\alpha(x) = \sigma(g) g x$. Since $G$ acts freely on $X$, we get $\sigma(g)\cdot g = e \in G$ and hence $g \cdot \sigma(g) = e$. Since $\rm H^1(\ZZ/2, G) = 0$, the equality $g \cdot \sigma(g) = e$ implies that there exists $h \in G$ such that $g = h^{-1}\sigma(h)$. Thus, we have $\alpha(x) = gx = h^{-1}\sigma(h)x$. Therefore, $h \alpha(x) = \sigma(h)x$, and hence $\alpha(\sigma(h)x) = h \alpha(x) = \sigma(h)x$. In other words, $\sigma(h)x \in X^\alpha$. Since $[x] = [\sigma(h)x] \in (G\sm X)^\alpha$, we conclude that the map $G^\alpha \sm X^\alpha \to (G\sm X)^\alpha$ is surjective. 

To prove injectivity of this map, let $x,y \in X^\alpha$ such that $gx = y$ for some $g \in G$. Then $\sigma(g)x = \alpha(gx) = \alpha(y) = y = gx$. As $G$ acts freely on $X$, we get 
$g \in G^\sigma$. 
\end{proof}



\begin{proof}[Proof of Theorem \ref{theorem:introduction:non-arithmeticity}]
Let $n \geq 2$. 
Consider the Eisenstein hermitian lattice $\mr L^2_{\eis} = (\QQ(\zeta_3), \ZZ[\zeta_3]^{2,1})$ of rank three, and the Eisenstein hermitian lattice $\mr L^n_{\eis} = (\QQ(\zeta_3), \ZZ[\zeta_3]^{n,1})$ of rank $n+1$. Consider moreover the anti-unitary involutions $\alpha_{n,i} \colon \ZZ[\zeta_3]^{n,1} \to \ZZ[\zeta_3]^{n,1}$ defined in \eqref{def:beta}. By definition of $\Gamma^n_{\eis}$, we have $M(\mr L^n_\eis, \alpha_{n,0}) \cong \Gamma^n_{\eis} \sm \RRH^n$. 
Our goal is to prove that $\Gamma^n_{\eis}$ is non-arithmetic, i.e., that the complete connected hyperbolic orbifold $M(\mr L^n_\eis, \alpha_0)$ is non-arithmetic. 
By Theorem \ref{useful-theorem}, each anti-unitary involution of $\ZZ[\zeta_3]^{2,1}$ is $\Aut(\ZZ[\zeta_3]^{2,1})$-conjugate to exactly one of the $\pm \alpha_{2,i}$. Thus, we have $M(\mr L^2_\eis) = M(\mr L^2_\eis, \alpha_{2,0})$ by Lemma \ref{connected}, and we can apply Theorem \ref{theorem:totally-geodesic-standard} to see that $M(\mr L^n_\eis, \alpha_{n,0})$ is non-arithmetic if $M(\mr L^2_\eis)$ is non-arithmetic.  In other words, to prove the theorem, it suffices to prove $\Gamma^2_\eis$ is non-arithmetic. 
By Theorem \ref{theorem:connection-ACT}, the lattice $\Gamma^2_\eis \subset \PO(2,1)$ is conjugate to the lattice $\Gamma^\RR_\infty \subset \PO(2,1)$ defined in \cite[Section 5]{realACTnonarithmetic}. The lattice $\Gamma^\RR_\infty \subset \PO(2,1)$ is non-arithmetic by \cite[page 168]{realACTnonarithmetic}. Therefore, $\Gamma^2_\eis \subset \PO(2,1)$ is non-arithmetic, and we are done. 
\end{proof}

\section{Volume} \label{section:volume}The goal of this section is to prove Theorem \ref{theorem:volume}. We thus return our attention to the Eisenstein hermitian lattices $\mr L_{\rm{eis}}^n = (\QQ(\zeta_3), \ZZ[\zeta_3]^{n,1})$ for $n \geq 2$. 
Let $\mr A_\eis^n$ be the set of anti-unitary involutions $\alpha \colon \ZZ[\zeta_3]^{n,1} \to \ZZ[\zeta_3]^{n,1}$, let $P\mr A_\eis^n$ be the quotient of $\mr A_\eis^n$ by $\mu_{\QQ(\zeta_3)} = \langle - \zeta_3 \rangle$, and let $C\mr A_\eis^n \subset P\mr A_\eis^n$ be a set of representatives for the action of the group $L^n \coloneqq \Aut(\ZZ[\zeta_3]^{n,1})/\langle -\zeta_3 \rangle$ on $P\mr A_\eis^n$ by conjugation. Thus, we have $C\mr A_\eis^n \cong L^n \setminus P\mr A_\eis^n$. 

For $i = 0, \dotsc, n$, consider the integral quadratic form $\Psi_i^n$ defined as 
\[
\Psi_i^n(x_0, \dotsc, x_n) = - x_0^2 + 3 x_1^2 + \cdots + 3x_i^2 + x_{i+1}^2 + \cdots + x_n^2.
\]
For each $i \in \set{0, \dotsc, n}$, define an anti-unitary involution $\alpha_i$ as follows:
\begin{align*} 
\alpha_i(x_0,  \dotsc, x_n) = \left( \bar x_0, -\bar x_1, \dotsc, 
- \bar x_i, \bar x_{i + 1}, \dotsc, \bar x_n \right). 
\end{align*}

\begin{theorem} \label{theorem:identification-of-arithmetic-part}For each $i \in \set{0, \dotsc, n}$, one has a canonical isomorphism
$$
L_{\alpha_i}^n \xrightarrow{\sim} \PO(\Psi_i^n, \ZZ), 
$$
where $\RR H^n_{\alpha_i}$ and $L^n_{\alpha_i} \subset \rm{Isom}(\RR H^n_{\alpha_i})$ are defined in equation \eqref{align:definition:RH-STAB}. 
\end{theorem}
\begin{proof}
For simplicity, write $L = L^n$ and $\Lambda = \ZZ[\zeta_3]^{n,1}$.  Let $i \in \set{0, \dotsc, n}$,  consider the diagram
\begin{align} 
\label{diagram:new-try-works}
\xymatrix{
L_{\alpha_i} & L_{\alpha_i}^I \ar@{_{(}->}[l] \ar@{=}[r] & \PO(\Lambda^{\alpha_i})(\ZZ)^I \ar@{^{(}->}[r] & \PO(\Lambda^{\alpha_i})(\ZZ),
}
\end{align}
see equations \eqref{naturalembedding}, \eqref{align:first-inclusion-equality} and \eqref{last} in the proof of Theorem \ref{th:crucialthm-finitevolume}. We claim that, in this particular case, the inclusions on the left and right of \eqref{diagram:new-try-works} are equalities. 

Let us first prove that $L_{\alpha_i}^I =L_{\alpha_i}$. For this, recall that $L_{\alpha_i} = N_L(\alpha_i)$ (see Lemma \ref{normalisator}). Thus, to prove that $L_{\alpha_i}^I =L_{\alpha_i}$, it suffices to show that there are no $g \in \Aut(\Lambda)$ such that $g \circ \alpha_i = - \alpha_i \circ g$. This is true because of Lemma \ref{lemma:conjinv:intro}, which implies that the elements $\alpha_i$ and $-\alpha_i \in \mr A_\eis^n$ are not $L$-conjugate. 

It remains to prove that $\PO(\Lambda^{\alpha_i})(\ZZ)^I = \PO(\Lambda^{\alpha_i})(\ZZ)$. This comes down to showing that for each $i \in \set{0, \dotsc, n}$, every isometry of the $\ZZ$-lattice $\Lambda^{\alpha_j}$ is induced by an isometry of $
\Lambda = \ZZ[\zeta_3]^{n,1}$. To prove this, note that \begin{align} \label{align:lambdaalpha-i}\Lambda^{\alpha_i} = \ZZ \oplus \theta \ZZ^{\oplus i} \oplus \ZZ^{\oplus (n-i)}.
\end{align}
One can check that the lattice $M \coloneqq \Lambda^{\alpha_i} \cap \theta \Lambda \subset \Lambda$ can be described in terms of $\Lambda^{\alpha_i}$ as $M = 3 (\Lambda^{\alpha_i})^\vee$ where $(\Lambda^{\alpha_i})^\vee \subset \Lambda^{\alpha_i} \otimes_{\ZZ} \QQ$ denotes the lattice dual to $\Lambda^{\alpha_i}$. Therefore, every isometry of $\Lambda^{\alpha_i} \otimes_{\ZZ} \QQ(\zeta_3) = \Lambda \otimes_{\ZZ[\zeta_3]} \QQ(\zeta_3)$ which is induced by an isometry of $\Lambda^{\alpha_i}$ preserves the $\OO_K$-span of $\Lambda^{\alpha_i}$ and $(1/\theta)\cdot M$, which is exactly $\Lambda$. 

We conclude that $L_{\alpha_i} = \PO(\Lambda^{\alpha_i})(\ZZ)$. Finally, observe that, because of \eqref{align:lambdaalpha-i}, we have an isomorphism of $\ZZ$-lattices $\Lambda^{\alpha_i} \cong (\ZZ^{\oplus (n+1)}, \Psi_i^n)$. The theorem follows. 
\end{proof}

\begin{proof}[Proof of Theorem \ref{theorem:volume}]
Consider the anti-unitary involution $\alpha_0 \colon \ZZ[\zeta_3]^{n,1} \to \ZZ[\zeta_3]^{n,1}$ defined as $(x_0, \dotsc, x_n) \mapsto (\bar x_0, \dotsc, \bar x_n)$. Recall (see Section \ref{section:arithmetic-nature}) that $M(\mr L_{\rm{eis}}^n, \alpha_0)  \subset M(\mr L_{\rm{eis}}^n)$ denotes the connected component containing the image of the natural map $\RR H^n_{\alpha_0} \to M(\mr L_{\rm{eis}}^n)$, and that
$
\Gamma^n_{\rm{eis}} \subset \PO(n,1) 
$
is a lattice such that  $M(\mr L_{\rm{eis}}^n, \alpha_0) \cong \Gamma^n_{\rm{eis}} \setminus \RR H^n$, see Theorem \ref{theorem:introduction:uniformization}. For each $i \in \set{0, \dotsc, n}$ and each $j \neq i \in \set{0, \dotsc, j}$, define
\[
H_{ij}' \coloneqq \RR H^n_{\alpha_i} \cap \RR H^n_{\alpha_j}. 
\]
Then $H_{ij}'$ is a geodesic subspace of codimension $\va{i-j}$ in $ \RR H^n_{\alpha_i}$. For instance, if $r_1 = (0,1,0, \dotsc, 0) \in \ZZ[\zeta_3]^{n,1}$, then $r_1$ is an element of norm one, the involution $\alpha_0$ preserves the hyperplane $H_{r_1} \subset \CC H^n$, and $$H_{01}' =  \RR H^n_{\alpha_0} \cap \RR H^n_{\alpha_1}=  \RR H^n_{\alpha_0} \cap H_{r_1} = (H_{r_1})^{\alpha_0} \subset \RR H^n_{\alpha_0},$$ which is a geodesic subspace of codimension one in $ \RR H^n_{\alpha_0}$. By Lemma \ref{lemma:conjinv:intro}, the elements $\alpha_0, \dotsc, \alpha_n \in C\mr A_\eis^n$ are pairwise distinct. Consequently, if we define
\[
\mr H_i' \coloneqq \bigcup_{j \neq i} H_{ij}'  \subset \RR H^n_{\alpha_i}, 
\]
then the natural map 
\[
\coprod_{i = 0}^n L_{\alpha_i} \sm \left( \RR H^n_{\alpha_i} - \mr H_i' \right) \longhookrightarrow M(\mr L, \alpha_0) \cong \Gamma^n_{\rm{eis}} \setminus \RR H^n
\]
is an open immersion of hyperbolic orbifolds, where $L = L^n =  \Aut(\ZZ[\zeta_3]^{n,1})/\langle -\zeta_3 \rangle$. 
As $L_{\alpha_i} \cong \PO(\Psi_i^n,\ZZ)$ by Theorem \ref{theorem:identification-of-arithmetic-part}, the theorem follows in view of Theorem \ref{theorem:introduction:uniformization} and its proof, see Section \ref{section:proof-first-main}. 
\end{proof}


\section{Complex ball quotient and moduli of abelian varieties} \label{unitaryshimura}

The goal of this section is to 
study the structure of the complex ball quotient $L \setminus \CC H^n$ associated to an admissible hermitian lattice $\mr L = (K, \Lambda)$ (cf.\ Section \ref{set-up-2}) as a moduli space of abelian varieties with $\OO_K$-action of hyperbolic signature. 
The theory of this section will allow us in the next section to prove Theorem \ref{orthogonal}, saying that 
the hyperplanes in the arrangement $\mr H \subset \CC H^n$ are orthogonal along their intersection. 

\subsection{Alternating and skew-hermitian forms} 
Let $K$ be a CM field, with ring of integers $\OO_K$, and let $\sigma \colon K \to K$ be the involution compatible with complex conjugation on $\CC$. Let $\mf D_K \subset \OO_K$ be the different ideal. Thus, we have $\mf D_K^{-1} = \set{x \in K \mid \rm{Tr}_{K/\QQ}(x\OO_K) \subset \ZZ}$, where $$\rm{Tr}_{K/\QQ} \colon K \longrightarrow \QQ$$ is the trace morphism, see e.g.\ \cite[Chapter III, \S3]{localfields} or \cite[Chapter III, \S 2]{Neukirch}. 

For a finite free $\OO_K$-module $\Lambda$ we put $\Lambda_\QQ \coloneqq \Lambda \otimes_{\ZZ}\QQ$, and call a $\QQ$-bilinear form 
\[
T \colon \Lambda_\QQ \times \Lambda_\QQ \longrightarrow K
\]
\emph{skew-hermitian} if $T$ is $K$-linear in its first argument and satisfies $T(y,x) = - \sigma(T(x,y))$ for $x,y \in \Lambda_\QQ$. If $T \colon \Lambda_\QQ \times \Lambda_\QQ \to K$ is skew-hermitian, then  $T(x, a\cdot y) = 
\sigma(a)\cdot T(x,y)$ for each $x,y \in \Lambda_\QQ$ and $a \in K$. 

\begin{lemma} \label{lemma:equivalentforms}
Let $n \geq 1$. 
Let $\Lambda$ be a free $\OO_K$-module of rank $n +1$. 
The assignment $T \mapsto \textnormal{Tr}_{K/\QQ} \circ T$ defines a bijection between:
\begin{enumerate}
    \item \label{item:hermitian} The set of skew-hermitian forms $T \colon \Lambda_{\QQ} \times \Lambda_{\QQ} \to K$. 
    \item \label{item:alternating} The set of alternating forms $E\colon \Lambda_{\QQ} \times \Lambda_{\QQ} \to \QQ$ that satisfy $E(a \cdot x,y) = E( x, \sigma(a) \cdot y)$ for $x,y \in \Lambda_\QQ$ and $a \in K$. 
\end{enumerate}
Under this correspondence, $T(\Lambda, \Lambda) \subset \mf D_K^{-1}$ if and only if $E(\Lambda, \Lambda) \subset \ZZ$. 
\end{lemma}

\begin{proof}
Let $T\colon  \Lambda_{\QQ} \times \Lambda_{\QQ} \to K$ be a skew-hermitian form. Define $E_T = \text{Tr}_{K/\QQ} \circ T$. 
Since $T$ is skew-hermitian, we have, for each $x, y \in \Lambda_\QQ$, that $$\text{Tr}_{K/\QQ}T(x,y) = - \text{Tr}_{K/\QQ}\sigma \left( T(y,x) \right).$$ Since $K/\QQ$ is separable, for $x \in K$, we have 
$
\text{Tr}_{K/\QQ}(x) = \sum_{1 \leq i \leq g} \left( \tau_i(x) + \tau_i\sigma(x) \right)$, see e.g.\ \cite[(7-1)]{stevenhagen}. Therefore, $\text{Tr}_{K/\QQ}(\sigma(x)) = \text{Tr}_{K/\QQ}(x)$, so that $$E_T(x,y) = \text{Tr}_{K/\QQ}T(x,y) = - \text{Tr}_{K/\QQ}\sigma \left( T(y,x) \right) =- \text{Tr}_{K/\QQ} T(y,x)  =
- E_T(y,x)$$ for any $x,y \in \Lambda_\QQ$. Moreover, we have $T(a \cdot x,y) = a \cdot T(x,y) = T(x,\sigma(a) \cdot y)$ hence $E_T(a \cdot x,y) = E_T(x,\sigma(a) \cdot y)$ for all $x,y \in \Lambda_\QQ$ and $a \in K$. 

Conversely, let $E \colon \Lambda_\QQ \times \Lambda_\QQ \to \QQ$ be an alternating form such that $E(a\cdot x,y) = E(x, \sigma(a) \cdot y)$ for $x,y \in \Lambda_\QQ$ and $a \in K$. Choose a basis $\{b_1, \dotsc, b_{n+1} \} \subset \Lambda_\QQ$ for $\Lambda_\QQ$ over $K$. Let $\{e_i\}$ be the canonical basis of $K^{n+1}$, 
and let $Q$ be the map $K^{n+1} \times K^{n+1} \to \QQ$ defined by $Q(e_i, e_j) = E(b_i, b_j)$. For $i,j \in \set{1, \dotsc, n+1}$, define a $\QQ$-linear map 
\[
f_{ij} \colon K \longrightarrow \QQ, \quad \quad f_{ij}(a) = Q(a \cdot e_i, e_j). 
\]
Since the trace pairing 
\[
\rm{Tr}_{K/\QQ} \colon K \times K \longrightarrow \QQ
\] 
is non-degenerate (see e.g.\ \cite[\href{https://stacks.math.columbia.edu/tag/0BIL}{Tag 0BIL}]{stacks-project}), for each $i,j \in \set{1, \dotsc, n+1}$ there is a unique element $t_{ij} \in K$ such that $f_{ij}(a) = \text{Tr}_{K/\QQ}(a \cdot t_{ij})$ for every $a \in K$. This gives a matrix $(t_{ij})_{ij} \in \rm{M}_{n+1}(K)$ such that $\sigma(t_{ij}) = - t_{ji}$, and via the basis $\{b_i\}$ of $\Lambda_\QQ$, this matrix $(t_{ij})_{ij}$ induces a skew-hermitian form $T_E\colon \Lambda_\QQ \times \Lambda_\QQ \to K$. It is straightforward to show that the functions $T \mapsto E_T$ and $E \mapsto T_E$ are inverse to each other, providing the desired bijection between the sets \ref{item:hermitian} and \ref{item:alternating}. 

Finally, let $E$ be an alternating form corresponding to a skew-hermitian form $T$ under the above bijection of sets. Then $E(x,y) \in \ZZ$ for all $x,y \in \Lambda$ if and only if for each $x,y \in \Lambda$, we have $\rm{Tr}_{K/\QQ}\left(a \cdot T(x,y)\right) = \rm{Tr}_{K/\QQ}\left(T(a \cdot x,y)\right) \in \ZZ$ for all $a \in \OO_K$. The latter holds if and only if for each $x,y \in \Lambda$, we have $T(x,y) \in \mf D_K^{-1}$. 
\end{proof}

\begin{examples} \label{examplesunitary}
\begin{enumerate}
    \item \label{ex:unitone} Let $K = \QQ(\sqrt \Delta)$ is imaginary quadratic over $\QQ$, with discriminant $\Delta$ and with involution $\sigma \colon K \to K$ defined by $\sigma(\sqrt{\Delta}) = - \sqrt{\Delta}$. Let $E \colon \Lambda \times \Lambda \to \ZZ$ be an alternating form that satisfies $E(a \cdot x,y) = E(x, \sigma(a) \cdot y)$ for $x,y \in \Lambda$ and $a \in \OO_K$. Define a form $T \colon \Lambda \times \Lambda \to \mf D_{M}^{-1} = (\sqrt{\Delta})^{-1}$ as 
    \begin{equation} \label{T-E}
        T(x,y) = \frac{E( \sqrt{\Delta}\cdot  x,y) + E(x,y)\sqrt{\Delta} }{2\sqrt \Delta}. 
    \end{equation}
    Then $T$ is skew-hermitian, because it is $K$-linear in its first variable (since we have $T( \sqrt{\Delta} \cdot x,y)  
    =\sqrt{\Delta} \cdot T(x,y)$) and because $T(y,x) = -\sigma(T(x,y))$. Moreover, $T$ is the skew-hermitian form associated to $E$ by Lemma \ref{lemma:equivalentforms} since we have $\rm{Tr}_{K/\QQ}(T(x,y)) = E(x,y)$, as is readily shown.
    \item \label{ex:unittwo} Let $K = \QQ(\zeta_p)$ where $\zeta_p = e^{2 \pi i/p} \in \CC$ for some prime number $p > 2$. Let $E: \Lambda \times \Lambda \to \ZZ$ be an alternating form with $E(a \cdot x,y) = E(x, \sigma(a) \cdot y)$ for $x,y \in \Lambda$, $a \in \OO_K$. Then $\mf D_K = \left(p/(\zeta_p - \zeta_p^{-1})\right)$, and the skew-hermitian $T\colon \Lambda \times \Lambda \to \left(p^{-1}\cdot(\zeta_p - \zeta_p^{-1})\right)$ associated to $E$ 
   by Lemma \ref{lemma:equivalentforms} is defined as
\begin{equation} \label{def:T}
 T(x,y) = \frac{1}{p}\cdot \sum_{j = 0}^{p-1}\zeta_p^j \cdot E\left( x, \zeta_p^j \cdot  y \right).
\end{equation}
Indeed, the form $T$ in \eqref{def:T} is skew-hermitian because for any $i \in \set{0, \dotsc, p-1}$, one has 
$
T(\zeta_p^i \cdot x,y) = 
\zeta_p^j \cdot p^{-1}  \cdot \sum_{j = 0}^{p-1}\zeta_p^{j-i} \cdot E( x, \zeta_p^{j-i} \cdot  y ) = \zeta_p^i \cdot T(x,y),
$
and we have $T(y,x) = 
- p^{-1}\cdot\sum_{j = 0}^{p-1}\zeta_p^j \cdot E( x,\zeta_p^{-j}\cdot y ) = - \sigma(T(x,y)) $. Moreover, 
\begin{align*}
\rm{Tr}_{K/\QQ}(T(x,y)) = \frac{1}{p}\cdot
\rm{Tr}_{K/\QQ}\left( \sum_{j = 0}^{p-1}\zeta_p^j \cdot E\left( x, \zeta_p^j \cdot  y \right) \right) = E(x,y). 
\end{align*}
\end{enumerate}
\end{examples}


Consider again a finite free $\OO_K$-module $\Lambda$, and define $\Lambda_\QQ = \Lambda \otimes_\ZZ \QQ$ and $\Lambda_\CC = \Lambda \otimes_\ZZ \CC$. Let $E \colon \Lambda_\QQ \times \Lambda_\QQ \to \QQ$ be an alternating form such that $E(a\cdot x,y) = E(x, \sigma(a) \cdot y)$ for $x,y \in \Lambda_\QQ$ and $a \in K$. Let $E_\CC \colon \Lambda_\CC \times \Lambda_\CC \to \CC$ be the $\CC$-linear extension of $E$. By viewing $\Lambda$ as a finite free $\ZZ$-module, the finite dimensional complex vector space $\Lambda_\CC = (\Lambda_\RR) \otimes_\RR \CC$ has a natural anti-holomorphic involution $z \mapsto \bar z$, and the function
\[
\Lambda_\CC \times \Lambda_\CC \longrightarrow \CC, \quad \quad (x,y) \mapsto E_\CC(x, \bar y)
\]
defines a skew-hermitian form on $\Lambda_\CC$. We view $\Lambda_\CC$ as an $\OO_K$-module by putting $a \cdot (x \otimes \lambda) = (ax) \otimes \lambda$ for $x \in \Lambda$ and $\lambda \in \CC$. At the same time, $\Lambda_\CC$ is a $\CC$-vector space in the usual way. Define
$$
 \Lambda_{\CC,\phi} \coloneqq \set{x \in \Lambda_\CC \mid a \cdot x = \phi(a)x \;\; \forall a \in \OO_K}.
$$
For $x,y \in \Lambda_{\CC,\phi}$, we define $E_\CC^\phi(x, \bar y) = E_\CC(x, \bar y)$; this defines a skew-hermitian form $\Lambda_{\CC,\phi} \times \Lambda_{\CC,\phi} \to \CC$.  Note that we have a canonical decomposition
\begin{align} \label{align:decomposition}
\Lambda \otimes_\ZZ \CC = \bigoplus_{\phi \in \Hom(K,\CC)}\left(\Lambda \otimes_\ZZ \CC \right)_{\phi}.\end{align} 
Define $p_\phi$ as the projection $ \Lambda_\CC \to \Lambda_{\CC,\phi}$ induced by \eqref{align:decomposition}, and let $\pi_\phi$ denote the canonical map $\Lambda_\CC \to \Lambda \otimes_{\OO_K} \CC$. The composition 
\begin{align*}
\Lambda_{\CC,\phi} \longhookrightarrow \Lambda_\CC \xlongrightarrow{\pi_\phi} \Lambda \otimes_{\OO_K, \phi} \CC
\end{align*}
is an isomorphism of complex vector spaces. 

Let $T \colon \Lambda_\QQ \times \Lambda_\QQ \to K$ be the skew-hermitian form that corresponds to $E$ via Lemma \ref{lemma:equivalentforms}. For any embedding $\phi \colon K \to \CC$, one obtains a skew-hermitian form $$T^\phi \colon \Lambda \otimes_{\OO_K, \phi} \CC \times \Lambda \otimes_{\OO_K, \phi} \CC\to \CC$$ by putting
$
T^\phi( x \otimes \lambda, y \otimes \mu) = 
\lambda \overline{\mu} \cdot \phi\left( T(x, y) \right)$ for $x,y \in \Lambda, \lambda, \mu \in \CC$ and extending linearly. 

\begin{lemma} \label{lemma:agree}
In the above notation, the following diagram commutes:
\[
\xymatrixcolsep{8pc}
\xymatrixrowsep{3pc}
\xymatrix{
\Lambda_\CC \times \Lambda_\CC \ar[r]^-{(x,y) \mapsto E_\CC(x, \bar y)} \ar@/_5pc/[dd]_-{(\pi_\phi \times \pi_\phi)_\phi} \ar[d]^-{(p_\phi \times p_\phi)_\phi}& \CC \ar@{=}[d] \\
\bigoplus_{\phi} \left(\Lambda_{\CC, \phi} \times \Lambda_{\CC,\phi} \right) \ar[r]^-{(x_\phi, y_\phi)_\phi \mapsto \sum_\phi E_\CC^\phi(x_\phi, \bar y_\phi)} \ar[d]^-\wr & \CC \ar@{=}[d] \\
\bigoplus_{\phi} \left(\Lambda \otimes_{\OO_K, \phi}\CC\right) \times \left(\Lambda \otimes_{\OO_K, \phi}\CC\right) \ar[r]^-{(x_\phi, y_\phi)_\phi \mapsto \sum_\phi T^\phi(x_\phi, y_\phi)}  & \CC,
}
\]
where the direct sums range over the elements $\phi \in \Hom(K,\CC)$. 
In particular, $$E_\CC(x,\bar y) = \sum_{\phi \in \Hom(K,\CC)} T^\phi(\pi_\phi(x), \pi_\phi(y))$$ for every $x,y \in \Lambda_\CC$, 
where $\pi_\phi$ denotes the canonical map $\Lambda_\CC \to \Lambda \otimes_{\OO_K} \CC$.
\end{lemma}
\begin{proof}
This is straightforward and we leave the proof to the reader. 
\end{proof}

\subsection{Moduli of abelian varieties with CM endomorphisms} \label{subsec:moduliabelianvarieties} 
\begin{notation} \label{not:fixhermitian} In the rest of Section \ref{unitaryshimura}, we work with the following data. 
\begin{enumerate} 
\item Let $K$ be a CM field over $\QQ$, and define $\sigma \colon K \to K$ as the involution that corresponds to complex conjugation on $\CC$.
\item \label{condition:eta-herm} Let $\eta \in \OO_K$ be a non-zero element such that 
$\sigma(\eta) = - \eta$. 
\item \label{strange-item} Let $\Lambda$ be a free $\OO_K$-module of rank $n +1$ for some $n \in \ZZ_{\geq 0}$ equipped with a non-degenerate hermitian form $h\colon \Lambda \times \Lambda \to \eta \cdot \mf D_K^{-1} \subset K$. 
\item \label{CM-type-Phi} Let $\Phi  \subset \Hom(K, \CC)$ be the CM type such that $\Im \left(\phi(\eta) \right) > 0$ for all $\phi \in \Phi$.  
\end{enumerate}
These data define a skew-hermitian form 
$
T\colon \Lambda \times \Lambda \to \mf D_K^{-1}$ by putting $T \coloneqq \eta^{-1} \cdot h.
$
The form $T$ is in turn attached to an alternating form (see Lemma \ref{lemma:equivalentforms}):
\[
E\colon \Lambda \times \Lambda \to \ZZ \; \tn{ such that } \; E(a x, y) = E(x, \sigma(a) y)\; \tn{ for all }\; a \in \OO_K, \; x,y \in \Lambda. 
\]
Define $V_\phi = \Lambda_\QQ \otimes_{K, \phi} \CC = (\Lambda_\QQ \otimes_\QQ \CC)_\phi$, and let
$$
h_{\phi} \colon V_\phi \times V_\phi \longrightarrow \CC
$$
be the hermitian form determined by the condition that $h_\phi(x,y) = \phi(h(x,y))$ for $x,y \in \Lambda \subset V_\phi$. For $\phi \in \Phi$, let $(r_\phi,s_\phi)$ be the signature of the hermitian form $h_{\phi}$.
\end{notation}

Let $A$ be a complex abelian variety with dual abelian variety $A^\vee$. Assume that $A$ is equipped with a polarization $\lambda \colon A \to A^\vee$ and a ring homomorphism $\iota \colon \ca O_K \to \text{End}(A)$ satisfying the following conditions (compare~\cite[\S2.1]{kudlarap-special-II}):
\begin{conditions} \label{KRconditions}
\begin{enumerate} 
    \item \label{KR-1}For each $a \in K$, we have $\iota(a)^\dagger = i(\sigma(a))$, where $\dagger \colon \End(A)_{\QQ} \to\End(A)_{\QQ}$ denotes the Rosati involution.  
    \item \label{KR-2}We have 
$
\textnormal{char}(t, \iota(a) | \textnormal{Lie}(A)) =
\prod_{\phi \in \Phi} (t-a^{\phi})^{r_\phi}
\cdot 
(t-a^{\phi\sigma})^{s_\phi} \; \in \;  \CC[t].
$
\end{enumerate}
\end{conditions}
\noindent
Here, $\textnormal{char}(t, \iota(a) | \textnormal{Lie}(A)) \in \CC[t]$ denotes the characteristic polynomial of $\iota(a)$. 
Observe that $\dim(A) = g(n+1)$, where $g = [K \colon \QQ]$. Let \begin{align}\label{align:alternating-pol}E_A \colon \rm H_1(A, \ZZ) \times \rm H_1(A, \ZZ) \longrightarrow \ZZ\end{align} be the alternating form corresponding to the polarization $\lambda \colon A \to A^\vee$. Condition \ref{KRconditions}.\ref{KR-2} implies that $E_A(\iota(a)x, y) = E_A(x, \iota(\sigma(a))y)$ for $x,y \in \rm H_1(A, \QQ)$. 
Let \begin{align} \label{align:skew-pol}T_A \colon \rm H_1(A, \ZZ) \times \rm H_1(A, \ZZ) \longrightarrow \mf D_K^{-1}\end{align} denote the skew-hermitian form attached to $E_A$ via Lemma \ref{lemma:equivalentforms}, and define a hermitian form \begin{align} \label{hermitian:hA}h_A \colon \rm H_1(A, \ZZ) \times \rm H_1(A, \ZZ) \longrightarrow \eta \cdot \mf D_K^{-1} \subset K, \quad \quad h_A \coloneqq \eta \cdot T_A.\end{align} 
\begin{definition} \label{def:shimura}  Consider the above notation. 
\begin{enumerate}
\item A \emph{polarized $\OO_K$-linear abelian variety} is a triple $(A, \lambda, \iota)$ where $(A, \lambda)$ is a polarized abelian variety and $\iota \colon \OO_K \to \End(A)$ a ring homomorphism such that Condition \ref{KRconditions}.\ref{KR-1} is satisfied. A \emph{$\Lambda$-marking} of a polarized $\OO_K$-linear abelian variety $(A, \lambda, \iota)$ is an isomorphism $j \colon \rm H_1(A, \ZZ) \xrightarrow{\sim} \Lambda$ of $\OO_K$-modules which is compatible with the alternating forms $E_A$ and $E$ (equivalently, with the hermitian forms $h_A$ and $h$). We say that a polarized $\OO_K$-linear abelian variety $(A, \lambda, \iota)$ is \emph{of signature $(r_\phi, s_\phi)_{\phi \in \Phi}$} if Condition \ref{KRconditions}.\ref{KR-2} is satisfied. Note that any $\Lambda$-marked polarized $\OO_K$-linear abelian variety is of signature $(r_\phi, s_\phi)_{\phi \in \Phi}$.  
\item 
Let $\widetilde{\textnormal{Sh}}_{K}(\Lambda)$ be the set of isomorphism classes of $\Lambda$-marked polarized $\OO_K$-linear abelian varieties $(A, \lambda, \iota, j)$ (an \emph{isomorphism} $(A_1, \lambda_1, \iota_1, j_1) \xrightarrow{\sim} (A_2, \lambda_2, \iota_2, j_2)$ between such objects is an isomorphism $A_1 \xrightarrow{\sim} A_2$ compatible with the polarizations $\lambda_i$, ring homomorphisms $\iota_i$, and isometries $j_i$). 
\item 
Let $\bb D(V_\phi)$ be the space of negative $s_\phi$-planes in the hermitian space $(V_{\phi}, h_{\phi})$. 
\end{enumerate}
\end{definition}

The following proposition is due to Shimura, see \cite[Theorems 1 \& 2]{Shimura1963ONAF} or \cite[\S 1]{shimuratranscendental}. We give a different proof, since our method will imply Proposition \ref{prop:HcorrespondsNonSimple} below, whereas we did not know how to deduce Proposition \ref{prop:HcorrespondsNonSimple} from the proofs of \cite[Theorems 1 \& 2]{Shimura1963ONAF}. Remark that Shimura assumes $\Lambda$ to be an $R$-module, for an order $R \subset \OO_K$ which is not necessarily maximal; our proof carries over, but we do not need this more general result. 

\begin{proposition} \label{prop:canonicalbijection}
There is a canonical bijection $\widetilde{\textnormal{Sh}}_M(\Lambda) \cong \prod_{\phi \in \Phi} \bb D(V_\phi)$. 
\end{proposition}

\begin{proof}
Let $(A, \lambda, \iota, j)$ be a representative of an isomorphism class that defines a point in $\widetilde{\textnormal{Sh}}_{K}(\Lambda)$. 
Consider the Hodge decomposition $\rm H_1(A, \CC) = \rm H^{-1,0} \oplus \rm H^{0,-1}$. For each $\phi \in \Phi$, we have $
\overline{ \rm H^{-1,0}_{\phi\sigma}} =  \rm H^{0,-1}_{\phi}$, hence there is a decomposition 
\begin{equation} \label{eq:posneg}
\rm H_1(A, \CC)_{\phi}  = \rm H^{-1,0}_{\phi} \oplus \rm H^{0,-1}_{\phi} \quad \text{ with } \quad \dim \rm H^{-1,0}_{\phi} = r_\phi \; \text{ and } \; \dim \rm H^{0,-1}_{\phi} = s_\phi.
\end{equation}
For an embedding $\phi \colon K \hookrightarrow \CC$, let $h_{A,\phi} \colon \rm H_1(A,\CC)_\phi \times \rm H_1(A,\CC)_\phi \to \CC$ be the hermitian form induced by the hermitian form \eqref{hermitian:hA}. By Lemma \ref{lemma:agree}, we have \begin{align} \label{align:we-have-agree}\phi(\eta) \cdot E_{A, \CC}(x,\bar y) = {h}_{A,\phi}(x,y) \quad \quad \forall x,y \in \rm H_1(A, \CC)_\phi, \;\; \phi \in \Hom(K,\CC).\end{align}Since $\Im(\phi(\eta)) > 0$ for every $\phi \in \Phi$, 
the decomposition of $\rm H_1(A, \CC)_{\phi}$ in (\ref{eq:posneg}) is a decomposition into a positive definite $r_\phi$-dimensional subspace and a negative definite $s_\phi$-dimensional subspace. 
The isometry $j \colon \rm H_1(A, \QQ) \xrightarrow{\sim} \Lambda_\QQ$ induces an isometry $j_\phi \colon \rm H_1(A, \CC)_{\phi} \xrightarrow{\sim} V_\phi$ for each $\phi \in \Phi$, and so we obtain a negative $s_\phi$-plane $$j_\phi( \rm H^{0,-1}_{\phi}) \subset V_\phi$$ for each $\phi \in \Phi$. Reversing the argument shows that given a negative $s_\phi$-plane $X_\phi \subset V_\phi$ for each $\phi \in \Phi$, there is a canonical polarized abelian variety $(A, \lambda)$ with $A = \rm H^{-1,0}/\Lambda$, acted upon by $\OO_K$, and inducing the planes $X_\phi \subset V_\phi$. 
\end{proof}
\begin{definition} \label{definition-shimura-2}
Let $\textnormal{Sh}_{K}(\Lambda)$ be the set of isomorphism classes of polarized $\OO_K$-linear abelian varieties $(A, \lambda, \iota)$ of signature $(r_\phi, s_\phi)_{\phi \in \Phi}$ that admit a $\Lambda$-marking. 
\end{definition}
\begin{corollary} \label{cor-uniformization}
There is a canonical bijection $$\textnormal{Sh}_{K}(\Lambda) \xrightarrow{\sim} \Aut(\Lambda) \setminus  \prod_{i = 1}^g \bb D(V_i).$$ 
\end{corollary}
\begin{proof}
The bijection in Proposition \ref{prop:canonicalbijection} is $\Aut(\Lambda)$-equivariant. 
\end{proof}
\begin{remark}
The above corollary generalizes \cite[Proposition 3.1]{kudlarap-special-II}, where the result is proven in the case where $K$ is an imaginary quadratic number field. 
\end{remark}

\subsection{Moduli of abelian varieties and the hyperplane arrangement} Consider Notation \ref{not:fixhermitian}. 
Consider CM type $\Phi \subset \Hom(K,\CC)$, and choose an ordering $\Phi = \set{\phi_1, \dotsc, \phi_g}$ for the set $\Phi$. 
This yields an embedding \begin{align} \label{align:embedding-phi}\Phi\colon \OO_K \longhookrightarrow \CC^g, \quad \quad \Phi(a) = (\phi_1(a), \dotsc, \phi_g(a)),\end{align} such that $\Phi(\OO_K) \otimes_\ZZ \RR = \CC^g$. 
In particular, we obtain a complex torus $\CC^g/\Phi(\OO_K)$ equipped with a natural ring homomorphism $\OO_K \to \End(\CC^g/\Phi(\OO_K))$. 
The map 
\begin{align}\label{ppav-cm}
Q \colon K \times K \longrightarrow \QQ, \quad Q(x,y) = \text{Tr}_{K/\QQ}\left(\eta^{-1} \cdot  x \bar y \right) \quad \quad (\bar y = \sigma(y))
\end{align}
is a non-degenerate alternating $\QQ$-bilinear form with $Q(ax,y) = { Q}(x, \sigma(a) y)$ for $a,x,y \in K$. 
\begin{lemma} \label{lemma:milne}
Consider the element $\eta \in \OO_K$, see Notation \ref{not:fixhermitian}.\ref{condition:eta-herm}. If $\eta^{-1}\in \mf D_K^{-1}$, then  $Q(\OO_K, \OO_K) \subset \ZZ$, and the alternating form $Q \colon \OO_K \times \OO_K \to \ZZ$ defines a polarization on the complex torus $\CC^g/\Phi(\OO_K)$ such that Condition \ref{KRconditions}.\ref{KR-1} is satisfied. 
\end{lemma}
\begin{proof}
We have $Q(\OO_K, \OO_K) \subset \ZZ$ because $\eta^{-1}\in \mf D_K^{-1}$. Thus, by \cite[Example 2.9 \& Footnote 16]{milneCM}, $Q$ defines a polarization on $\CC^g/\Phi(\OO_K)$. The result follows. 
\end{proof}

Consider the signature $(r_\phi, s_\phi )$ of the hermitian form $h_\phi \colon V_\phi \times V_\phi \to \CC$, for $\phi \in \Phi$. 
Let $\tau \in \Phi$, and assume that
\begin{align} \label{assumption:shimura}
\begin{split}
(r_\phi, s_\phi )\quad =\quad 
\begin{cases}
(n,1) \quad &\tn{ if } \phi = \tau,\\
(n+1, 0) \quad &\tn{ if } \phi \neq \tau.
\end{cases}
\end{split}
\end{align}
The space $V \coloneqq \Lambda \otimes_{\OO_K, \tau} \CC
$ 
is equipped with the hermitian form $h_{\tau} \colon V \times V \to \CC$, and to ease notation, we put $h \coloneqq h_\tau$. Define $\CCH^n$ as the space of lines in $V$ which are negative for $h$; in other words, we have:
\begin{align*}
\CC H^n = \set{\tn{lines } \ell \subset V \mid h(x,x) < 0 \tn{ for } 0 \neq x \in \ell}.
\end{align*}
For $r \in \Lambda$ such that $h(r,r) = 1$, let $r_\tau \in V$ be the image of $r$ under the canonical embedding $\Lambda \hookrightarrow  \Lambda \otimes_{\OO_K, \tau} \CC= V$, and put 
\[
H_r \coloneqq \set{\ell \in \CCH^n \mid h(r_\tau,x) = 0 \tn{ for } 0 \neq x \in \ell}, \quad \text{and} \quad \mr H \coloneqq \bigcup_{h(r,r) = 1} H_r \subset \CCH^n.
\]
For polarized $\OO_K$-linear abelian varieties $(A_1, \lambda_1, \iota_1)$ and $(A_2, \lambda_2, \iota_2)$, we say that a morphism of polarized abelian varieties $f \colon (A_1, \lambda_1) \to (A_2, \lambda_2)$ is \emph{$\OO_K$-linear} if $f$ is compatible with the ring homomorphisms $\iota_1$ and $\iota_2$. 
\begin{proposition} \label{prop:HcorrespondsNonSimple}
Consider the above notation. Assume the following holds:
\begin{enumerate}
\item \label{cond:1:prop} Condition \eqref{assumption:shimura} holds, i.e.,
for $\phi \in \Phi$, we have $(r_\phi, s_\phi) = (n,1)$ if $\phi = \tau$ and $(n+1, 0)$ otherwise. 
\item We have $\eta^{-1} \in \mf D_K^{-1}$. In particular, $\left(\CC^g/\Phi(\OO_K), Q\right)$ is a polarized $\OO_K$-linear abelian variety, see Lemma \ref{lemma:milne}. 
\end{enumerate}
Then, under the bijection $  \widetilde{\textnormal{Sh}}_{K}(\Lambda) \cong \CC H^n$ of Proposition \ref{prop:canonicalbijection}, the subset $\mr H \subset \CC H^n$ is identified with the set of isomorphism classes of $\Lambda$-marked polarized $\OO_K$-linear abelian varieties $(A, \lambda, \iota,j)$ 
for which there exist an $\OO_K$-linear homomorphism $$f \colon \CC^g/\Phi(\OO_K)  \longrightarrow A$$ 
of polarized $\OO_K$-linear abelian varieties. 
\end{proposition}

\begin{proof}
Let $\mr H' \subset \CCH^n$ be the subset parametrizing isomorphism classes of $\Lambda$-marked polarized $\OO_K$-linear abelian varieties $(A, \lambda, \iota,j)$ admitting an $\OO_K$-linear homomorphism $f \colon \CC^g/\Phi(\OO_K)  \to A$ of polarized $\OO_K$-linear abelian varieties, cf.\ Proposition \ref{prop:canonicalbijection}. We must show that $\mr H = \mr H'$. 

We first show that $\mr H' \subset \mr H$. To this end, let $[(A,\lambda, \iota,j)] \in \widetilde{\textnormal{Sh}}_{K}(\Lambda) $ and $x \in V$, and assume that $h(x,x) < 0$ and that the point $[x] \in \CCH^n$ attached to $x$ lies in $\mr H' \subset \CCH^n$. Thus, 
there exists an $\OO_K$-linear homomorphism $f \colon \CC^g/\Phi(\OO_K) \to A$ of polarized $\OO_K$-linear abelian varieties. We must show that $[x] \in \mr H \subset \CCH^n$, i.e., that there exists $r \in \Lambda$ such that $h(r,r) = 1$ and $h(r_\tau,x) = 0$, where $r_\tau \in V$ is the image of $r$ under the canonical embedding $\Lambda \hookrightarrow V$. 

To prove this, we may assume that $A = \rm H^{-1,0} / \Lambda$ with $\Lambda \otimes_\ZZ \CC = \rm H^{-1,0} \oplus \rm H^{0,-1}$, and that $T_A = T$ and $E_A =E$. The morphism $f$ induces an $\OO_K$-linear map $$\OO_K \xlongrightarrow{\sim} \Phi(\OO_K) \xlongrightarrow{f_\ast} \rm H_1(A, \ZZ) = \Lambda$$ which, for simplicity, we also denote by $f \colon \OO_K \to \Lambda$. Define \begin{align} \label{def:r=1} r \coloneqq f(1) \in \Lambda. \end{align}
Since $f \colon \CC^g/\Phi(\OO_K) \to A$ is a morphism of polarized abelian varieties, the map $f \colon \OO_K \to \Lambda$ satisfies $Q = f^\ast(E)$, i.e., we have $Q(x,y) = E(f(x), f(y))$ for each $x,y \in \OO_K$. Let $T_Q$ be the skew-hermitian form on $\OO_K$ associated to $Q$ via Lemma \ref{lemma:equivalentforms}. We claim that
\begin{align}\label{claim:TQ}T_Q(x,y) = T(f(x), f(y)) \quad \quad \forall x,y \in \OO_K.\end{align}
To see this, note that $(x,y) \mapsto T(f(x), f(y))$ defines a skew-hermitian form on $K = \OO_K \otimes_\ZZ \QQ$, and that $\rm{Tr}_{K/\QQ}(T(f(x), f(y))) = E(f(x), f(y)) = Q(x,y)$
for $x,y \in \OO_K$. Thus, \eqref{claim:TQ} follows from Lemma \ref{lemma:equivalentforms}. Note that 
\begin{align}\label{by-def}
T_Q(x,y) = \eta^{-1} \cdot x \bar y \quad \quad \forall x,y \in \OO_K,
\end{align}
see 
equation \eqref{ppav-cm} and Lemma \ref{lemma:equivalentforms}. By \eqref{def:r=1}, \eqref{claim:TQ} and \eqref{by-def}, we get 
$ T(r,r) = T(f(1), f(1)) = T_Q(1,1)  = \eta^{-1},$ so that \begin{align}\label{align:norm1}h(r,r) = \eta \cdot T(r,r) = 1.\end{align} We claim that $h(r_\tau,x) = 0$, where the element $r_\tau \in V$ is the image of $r \in \Lambda$ under the natural inclusion $\Lambda \hookrightarrow V$. 
To see this, write 
$
M \coloneqq \Phi(\OO_K), $ with Hodge decomposition $ M_\CC= M \otimes_\ZZ \CC = M^{-1,0} \oplus M^{0,-1}$. Define $h_Q \coloneqq \eta T_Q$, so that $h_Q(x,y) = x \bar y$ for $x,y \in \OO_K$. 
For $\phi \in \Hom(K,\CC)$, let $h_{Q,\phi}$ denote the hermitian form on $ (K \otimes_\QQ \CC)_\phi \cong \CC$ induced by $h_Q$ and the embedding $\phi$, and note that $h_{Q, \phi}(x,y) = x \bar y$. 
  By Lemma \ref{lemma:agree}, we have $\phi(\eta) \cdot Q_\CC(x,\bar y) = h_{Q,\phi}(x,y)$ for all $ x,y \in (K \otimes_\QQ \CC)_\phi.$
  Consider the isomorphism $\Phi \colon \OO_K \otimes_\ZZ \CC \xrightarrow{\sim} M \otimes_\ZZ \CC$ induced by the isomorphism $\Phi \colon \OO_K \xrightarrow{\sim} M$. 
  As the hermitian form $h_{Q,\phi}(\Phi(x), \Phi(y)) = x \bar y$ is positive definite on $(M_\CC)_\phi$ 
  for each $\phi \in \Hom(K,\CC)$, as $\Im(\phi(\eta)) > 0$ for every $\phi \in \Phi$, and as the hermitian $iQ_\CC(\Phi(x), \overline{\Phi(y)})$ is positive definite on $M^{-1,0}$, we have: 
\begin{equation} \label{eq:posneg:2:3}
M^{-1,0} = M^{-1,0}_{\phi} = (M \otimes_\ZZ \CC)_\phi 
\quad \quad \forall \phi \in \Phi.
\end{equation}
Let $r_\tau^{-1,0} \in \rm H^{-1,0}_\tau$ and $r_\tau^{0,-1} \in \rm H^{0,-1}_\tau$ be the components of $r_\tau$ in the Hodge decomposition, so that $r_\tau = r^{-1,0}_\tau + r^{0,-1}_\tau$. As the projection of $\Phi(1) \in M$ onto $M^{0,-1}_\tau$ is zero (since $M^{0,-1}_\tau=0$ by equation \eqref{eq:posneg:2:3}), we get  that $r_\tau^{0,-1}$ is zero as well. 
Moreover, in view of \eqref{align:we-have-agree}, the Hodge decomposition $\rm H_1(A, \CC)_\tau = \rm H^{-1,0}_\tau \oplus \rm H^{0,-1}_\tau$ is orthogonal for the hermitian form $h$ on $\rm H_1(A, \CC)_\tau = V$. Therefore, we have: $$r_\tau = r^{-1,0}_\tau \in \rm H^{-1,0}_\tau = ( \rm H^{0,-1}_\tau )^\perp  = \langle x \rangle ^\perp,\quad \quad \text{hence} \quad \quad h(r_\tau, x) = 0.$$ Combining this with \eqref{align:norm1}, we see that $[x] \in H_r \subset \mr H$, 
hence $\mr H' \subset \mr H$.  

To show that $\mr H \subset \mr H'$, let $x \in V$ be non-zero such that $[x] \in H_r \subset \mr H \subset \CCH^n$ for some $r \in \Lambda$ with $h(r,r) = 1$, and consider the $\Lambda$-marked $\OO_K$-linear polarized abelian variety $A = \rm H^{-1,0}/\Lambda$ corresponding to $[x] \in \CCH^n$. Define an $\OO_K$-linear map $f \colon \OO_K \to \Lambda$ by putting $f(1) = r$. Then $f$ is a morphism of Hodge structures because its $\CC$-linear extension preserves the eigenspace decompositions and $h(x,r) = 0$. 
We obtain an $\OO_K$-linear homorphism $f \colon \CC^g/\Phi(\OO_K) \to A$. The fact that $h(r,r) = 1$ implies that $f$ preserves the polarizations on both sides, so $[x] \in \mr H'$ and we win.
\end{proof}


\section{Orthogonality of the hyperplane arrangement}  \label{sec:orthogonalhyperplane}

The goal of this section is to prove Theorem \ref{orthogonal}. For an integer $n \geq 2$, let 
$
\mr L = (K, \Lambda)
$
be an admissible hermitian lattice of rank $n+1$. In particular, $K$ denotes a CM field, and as usual, we let $F \subset K$ denote its maximal totally real subfield, $\sigma \colon K \to K$ the involution corresponding to complex conjugation on $\CC$, and  $\mf D_K \subset \OO_K$ the different ideal. Let $ \eta \in \OO_K$ such that $\mf D_K = (\eta)$ and $\sigma(\eta) = -\eta$; such an element exists because $\mr L$ is admissible (see Definition \ref{definition:introduction:admissible}). 

Let $\Phi \subset \Hom(K, \CC)$ be the unique CM type such that $\Im(\phi(\eta)) > 0$ for all $\phi \in \Phi$. Then $\Lambda$ is a free $\OO_K$-module of rank $n + 1$ equipped with a hermitian form 
$$
h \colon \Lambda \times \Lambda \longrightarrow \OO_K
$$
of signature $(n,1)$ with respect to one $\tau \in \Phi$ and of signature $(n+1,0)$ with respect to all other $\phi \neq \tau \in \Phi$. These data define a skew-hermitian form 
$
T\colon \Lambda \times \Lambda \to \mf D_K^{-1}$ by putting $T \coloneqq \eta^{-1} \cdot h.
$
By Lemma \ref{lemma:equivalentforms}, the skew-hermitian form $T$ is in turn attached to an alternating form 
$
E\colon \Lambda \times \Lambda \to \ZZ$ such that $E(a x, y) = E(x, \sigma(a) y)$ for all $a \in \OO_K, x,y \in \Lambda$. 
As in Section \ref{set-up-2}, we consider the complex vector space $$V = \Lambda \otimes_{\OO_K, \tau} \CC = \Lambda \otimes_{\OO_F, \tau}\CC$$ equipped with the hermitian form $$h \colon V \times V \longrightarrow \CC$$ induced by the hermitian form $h \colon \Lambda \times \Lambda \to \OO_K$ and the embedding $\tau \colon K \hookrightarrow \CC$, and the space $\CC H^n$ of lines in $V$ which are negative for the hermitian form $h$ on $V$. As for an element $r \in \Lambda$ with $h(r,r)= 1$, 
we define $H_r \subset \CC H^n$ as the space of negative lines $\ell \subset V$ which are orthogonal to $r$ with respect to the hermitian form $h$ on $V$. 

\begin{proof}[Proof of Theorem \ref{orthogonal}]
Let $r, t \in \Lambda$ such that $h(r,r) = h(t,t) = 1$. Suppose that \begin{align}\label{align:HrHt}H_r \cap H_t \neq \emptyset \quad \text{ and } \quad H_r \neq H_t.\end{align}
We must prove that $h(r,t) = 0$. Let $g \coloneqq [F \colon \QQ]$. 
Define $\Phi \colon \OO_K \hookrightarrow \CC^g$ as in \eqref{align:embedding-phi} and consider the complex torus $B \coloneqq \CC^g/\Phi(\OO_K)$. Since the different ideal $\mf D_K$ is the principal ideal $ (\eta) \subset \OO_K$, we have
\begin{align*} 
\{x \in K \mid \textnormal{Tr}_{K/\QQ} \left(x \eta^{-1} \OO_K\right) \subset \ZZ \} = \{x \in K \mid x \cdot \eta^{-1} \OO_K \subset \eta^{-1} \OO_K \} = \OO_K. 
\end{align*}
Therefore, the polarization $\lambda_B$ on $B$ attached to the alternating form 
\[
Q \colon \Phi(\OO_K) \times \Phi(\OO_K) \to \ZZ, \quad \left(\Phi(x), \Phi(y) \right) \mapsto 
\textnormal{Tr}_{K/\QQ} \left(\eta^{-1} \cdot x \cdot y\right),
\] 
see Lemma \ref{lemma:milne}, is in fact a principal polarization on $B$. 

Consider the moduli space $\tn{Sh}_K(\Lambda)$ of Definition \ref{definition-shimura-2}, as well as the isomorphism
$
\tn{Sh}_K(\Lambda) \xrightarrow{\sim} \Aut(\Lambda) \sm \CC H^n
$
of Corollary \ref{cor-uniformization}. 
Let $[x] \in H_r \cap H_t \subset \CC H^n$, and let $[(A,\lambda, \iota)]$ be the point in $\rm{Sh}_K(\Lambda)$ corresponding to the image of $[x]$ in $L \setminus \CC H^n$, such that $(A, \lambda, \iota)$ is a polarized $\OO_K$-linear abelian variety with $\rm H_1(A,\ZZ) = \Lambda$ (so that $A = \rm H^{-1,0} / \Lambda$ with $\Lambda \otimes_\ZZ \CC = \rm H^{-1,0} \oplus \rm H^{0,-1}$), and $E_A =E$ and $T_A = T$, where $E_A$ and $T_A$ are the forms associated to the polarization $\lambda$ of $A$ as in \eqref{align:alternating-pol} and \eqref{align:skew-pol}. 

By Proposition \ref{prop:HcorrespondsNonSimple}, the norm one elements $r$ and $t$ of $\Lambda$ induce $\OO_K$-linear morphisms 
$
f_1 \colon B \to A$ and $f_2 \colon B \to A$ of polarized $\OO_K$-linear abelian varieties. Since the polarization of $B$ is principal, the maps $f_1$ and $f_2$ are embeddings. 
Moreover, since $(B, \lambda_B)$ is a principally polarized abelian variety, there exist an abelian subvariety $C_i \subset A$ for $i = 1,2$ such that 
$ A \cong B \times C_1$ and $ A \cong B \times C_2$ as polarized abelian varieties (cf.\ \cite[Corollary 5.3.13]{birkenhake}), where for $i = 1,2$, the polarization on $B \times C_i$ is the product polarization $\lambda_B \times \lambda_{C_i}$, with $\lambda_{C_i} = \lambda|_{C_i}$ the pull-back of the polarization $\lambda$ along the embedding $C_i \subset A$. 

Note that $B$ is a simple abelian variety because $\End(B) \otimes_\ZZ \QQ = \OO_K \otimes_{\ZZ} \QQ = K$ is a field. In particular $(B, \lambda_B)$ is non-decomposable as a principally polarized principally polarized abelian variety.  By \cite{debarreproduits}, the decomposition of $(A, \lambda)$ into non-decomposable polarized abelian subvarieties is unique, in the strong sense that if $(A_i, \lambda_i)$, $i\in \{1, \dotsc, a\}$ and $(B_j, \mu_j)$, $j \in \{1, \dotsc, b\}$ are polarized abelian subvarieties such that the natural homomorphisms $\prod_i(A_i, \lambda_i) \to (A, \lambda)$ and $\prod_j(B_j, \lambda_j) \to (A, \lambda)$ are isomorphisms, then $a = b$ and there exists a permutation $\sigma$ on $\{1, \dotsc, a\}$ such that $B_j$ and $A_{\sigma(j)}$ are \textit{equal} as polarized abelian subvarieties of $(A, \lambda)$, for every $j \in \{1, \dotsc, a\}$. Thus, for the abelian subvarieties $
B_i \coloneqq f_i(B) \subset A$, we have either that $B_1 = B_2 \subset A$, or that $B_1 \cap B_2 = \{0\} \subset A$. We claim that the first option is impossible. Indeed, if $B_1 = B_2$ as polarized abelian subvarieties of $A$, then if we let $f_i \colon \OO_K \to \Lambda$ denote the map induced by $f_i \colon B = \CC^g/\Phi(\OO_K) \to \rm H^{-1,0}/\Lambda = A$, we have $\OO_K\cdot r = f_1(\OO_K) = f_2(\OO_K)= \OO_K \cdot t \subset \Lambda$, hence $r = \lambda t$ for some $\lambda \in \OO_K^\ast$; but then $H_r = H_t \subset \CC H^n$ which contradicts \eqref{align:HrHt}. Therefore, we must have 
$$A = B_1 \times B_2 \times C$$ as polarized abelian varieties, for some polarized abelian subvariety $C$ of $A$. This implies that \begin{align}\label{AV:orth}
\rm H^{-1,0} = \text{Lie}(A) = \Lie(B_1) \oplus \Lie(B_2) \oplus \Lie(C),
\end{align}
and this decomposition is orthogonal for the positive definite hermitian form $iE_\CC(x,\bar y)$ on $\rm H^{-1,0}$, where $E \colon \rm H_1(A,\ZZ) \times \rm H_1(A,\ZZ) \to \ZZ$ is the alternating form corresponding to the polarization $\lambda$ of $A$. 

Let $r_\tau$ (resp.\ $t_\tau$) be the image of $r$ (resp.\ $t$) in $V = \left(\Lambda \otimes_\ZZ \CC \right)_\tau = \Lambda \otimes_{\OO_K, \tau}\CC$ under the natural map $\Lambda \to V$. Observe that $r_\tau = r^{-1,0}_\tau \in \rm H^{-1,0}_\tau$ and $t_\tau = t^{-1,0}_\tau \in \rm H^{-1,0}_\tau$, see the proof of Proposition \ref{prop:HcorrespondsNonSimple}. By Lemma \ref{lemma:agree}, see also equation \eqref{align:we-have-agree}, we have
\begin{align*}
h(r,t) = h(r_\tau, t_\tau) = \tau(\eta)\cdot T^\tau_\CC(r_\tau, t_\tau) = \tau(\eta) \cdot E_{\CC}(r_\tau, \overline{t_\tau}) = \tau(\eta) \cdot E_\CC (r_\tau^{-1,0}, \overline{t_\tau^{-1,0}}).
\end{align*}
Since $r_\tau^{-1,0} \in \Lie(B_1)$ and $t_\tau^{-1,0} \in \Lie(B_2)$, we have $i E_\CC(r_\tau^{-1,0}, \overline{t_\tau^{-1,0}}) = 0$ by the orthogonality of \eqref{AV:orth}. We conclude that $h(r,t) = 0$, and the proof is finished. 
\end{proof}

\begin{example}
Let $n \geq 1$ be an integer and $\mr L = (K, \Lambda)$ an hermitian lattice of rank $n+1$, such that the associated hermitian form $h \colon V \times V \to \CC$ on the complex vector space $V = \Lambda \otimes_\ZZ \RR$ has signature $(n,1)$. 
Assume that $K$ is the CM field $K = \QQ(\sqrt{d}) \subset \CC$ for some $d \in \ZZ$ with $d < 0$. By Example \ref{proposition:discr}, $\mr L$ is admissible. 

There is an alternative, more elementary proof of Theorem \ref{orthogonal} in this case, see \cite[Lemma 7.29]{ACTsurfaces}. Namely, let $r, t \in \Lambda \subset V$ with $h(r,r) = h(t,t) = 1$. Assume that $H_r \neq H_t$ and that $H_r \cap H_t \neq \emptyset \subset \CC H^n$. We claim that $h(r,t) = 0$. To prove this, let $[x] \in H_r \cap H_t$, lift $[x]$ to an element $x \in V$, and consider the decomposition $V = \langle x \rangle \oplus \langle x \rangle^\perp$. Since $r, t \in \langle x \rangle^\perp$ and the signature of $(V,h)$ is $(n,1)$, the elements $r$ and $t$ span a positive definite subspace $\langle r, t \rangle \subset V$. Therefore, the matrix 
\[
\begin{pmatrix}
h(r,r) & h(r,t) \\
h(t,r) & h(t,t)
\end{pmatrix} = 
\begin{pmatrix}
1 & h(r,t) \\
h(t,r) & 1
\end{pmatrix}
\]
is positive definite, hence $\va{h(r,t)}^2 < 1$. Since $\va{h(r,t)}^2 \in \ZZ_{\geq 0}$, we have $h(r,t) = 0$. 
\end{example}

\printbibliography

\textsc{Olivier de Gaay Fortman, Department of Mathematics, Universiteit Utrecht, Budapestlaan 6, 3584 CD Utrecht, The Netherlands}\par\nopagebreak 
  \textit{E-mail address:} \texttt{a.o.d.degaayfortman@uu.nl}

\end{document}